\documentclass[reqno,12pt]{amsart}

\usepackage[T1]{fontenc}
\usepackage{lmodern}

\makeatletter
\usepackage{fullpage}
\usepackage{textcmds}  
\usepackage{amsmath, amssymb, amsfonts, amstext, verbatim, amsthm, mathrsfs, stmaryrd}
\usepackage[all,cmtip,2cell]{xy}
\usepackage{pgf,tikz,pgfplots}
\pgfplotsset{compat=1.15}
\usetikzlibrary{arrows}
\usetikzlibrary{positioning}
\usetikzlibrary{quotes}
\usepackage{tikz-cd}

\usepackage{enumerate}
\usepackage{enumitem}
\usepackage[colorlinks=true,linkcolor=blue,citecolor=blue,urlcolor=blue,citebordercolor={0 0 1},urlbordercolor={0 0 1},linkbordercolor={0 0 1}]{hyperref} 
\usepackage[shortalphabetic]{amsrefs} 
\usepackage[nameinlink]{cleveref}

\usepackage{enotez}
\setenotez{backref=true}

\def\makeCal#1{%
\expandafter\newcommand\csname c#1\endcsname{\mathcal{#1}}}
\def\makeBB#1{%
\expandafter\newcommand\csname b#1\endcsname{\mathbb{#1}}}
\def\makeFrak#1{%
\expandafter\newcommand\csname f#1\endcsname{\mathfrak{#1}}}

\count@=0
\loop
\advance\count@ 1
\edef\y{\@Alph\count@}%
\expandafter\makeCal\y
\expandafter\makeBB\y
\expandafter\makeFrak\y
\ifnum\count@<26
\repeat

\theoremstyle{plain}
\newtheorem{thm}{Theorem}[section]
\newtheorem{cor}[thm]{Corollary}
\newtheorem{lem}[thm]{Lemma}

\newtheorem{prop}[thm]{Proposition}
\newtheorem{quest}[thm]{Question}

\theoremstyle{definition}
\newtheorem{rem}[thm]{Remark}
\newtheorem{defn}[thm]{Definition}
\newtheorem{hyp}[thm]{Hypotheses}

\newtheorem{notn}[thm]{Notation}
\newtheorem{const}[thm]{Construction}
\newtheorem{ex}[thm]{Example}

\newtheorem*{thm*}{Theorem}
\newtheorem*{prop*}{Proposition}

\def\rm{\mathrm}

\DeclareMathOperator{\rk}{rank}

\DeclareMathOperator{\Aut}{Aut}

\DeclareMathOperator{\Hom}{Hom}
\newcommand{\id}{\mathrm{id}}

\newcommand{\Sch}{\mathrm{Sch}}

\DeclareMathOperator{\Spec}{Spec}

\DeclareMathOperator{\Stab}{Stab}

\DeclareMathOperator{\Pic}{Pic}

\DeclareMathOperator{\im}{im}

\DeclareMathOperator{\Bl}{Bl}
\DeclareMathOperator{\Proj}{Proj}

\DeclareMathOperator{\logZ}{logZ}

\def\mscbar{\mathcal{A}}
\def\Mmscbar{\mathcal{A}}

\def\cal{\mathcal}
\def\bb{\mathbb}

\def\rk{\mathrm{rk}}

\usepackage{pbox}
\usepackage[normalem]{ulem}

\makeatletter

\usepackage{babel}
\begin{document}

\title{Spaces of Multiscaled Lines with Collision}

\author{Antonios-Alexandros Robotis}

\thanks{This material is based upon work supported by the National Science Foundation under Grant No. DMS-1928930 and by the Alfred P. Sloan Foundation under grant G-2021-16778, while the author was in residence at the Simons Laufer Mathematical Sciences Institute (formerly MSRI) in Berkeley, California, during the Spring 2024 semester.}

\dedicatory{for my mother}

\maketitle

\begin{abstract}
    We study varieties $\Mmscbar_n$ arising as equivariant compactifications of the space of $n$ points in $\bC$ up to overall translation. We define $\Mmscbar_n$ and examine its basic geometric properties before constructing an isomorphism to an augmented wonderful variety as studied in \cite{Huhetalsemismall}. We show that $\Mmscbar_n$ is in a canonical way a resolution of the space of scaled curves considered in \cite{Zahariucmarked}, proving along the way that the resolution constructed in \cite{Zahariucresolution} is equivalent to ours. 
\end{abstract}

\tableofcontents

\section{Introduction}

Configuration spaces are classical objects of study in mathematics and physics which often carry rich and intricate geometric structure. Roughly speaking, a configuration space of $n$ points in a space $X$ is a subset of $X^n$ or a related space. For instance, one often considers $\mathrm{Conf}_n(X) = \{(x_1,\ldots, x_n)\in X^n: x_i\ne x_j \text{ for all } i\ne j\}$ or its quotient space $\mathrm{Conf}_n(X)/\mathfrak{S}_n$. However, spaces such as $\mathrm{Conf}_n(X)$ have the defect that they are not compact; indeed, limit points are missing corresponding to paths in $\mathrm{Conf}_n(X)$ in which some $x_i$ and $x_j$ become arbitrarily close. As such, a compactification of $\mathrm{Conf}_n(X)$ consists, at least, of the information of what happens when a pair of points collide. 

One potential strategy is to simply allow points to collide, and one has $X^n$ as a compactific\-ation of $\mathrm{Conf}_n(X)$ when $X$ is compact. Sometimes, however, this is insufficient as one would like to remember more information about colliding points, such as their relative trajectories. A possible solution in algebraic geometry is to consider the Hilbert scheme of $n$ points on $X$, $\mathrm{Hilb}^n(X)$, which remembers information such as tangent directions of colliding points. For a smooth surface $X$, $\mathrm{Hilb}^n(X)$ is smooth and quite well-behaved \cite{Fogarty}*{Thm. 2.4}. However, when $\dim X\ge 3$ these spaces are often rather singular; for instance, $\mathrm{Hilb}^4(\bP^3)$ is singular \cite{Fogarty}*{p.11}.\footnote{Thanks to Ritvik Ramkumar for pointing out this example.} An alternative compactification $X[n]$ is defined in \cite{Fultonmacpherson} with the benefit that it is nonsingular when $X$ is. 

As a final example, consider the space $\mathrm{Conf}_n(\bA^1)/\mathrm{Aff}(\bA^1)$, where $\mathrm{Aff}(\bA^1)$ is the 2-dimen\-sional group of affine transformations of $\bA^1$ and $n\ge 3$. We can consider $\mathrm{Aff}(\bA^1)$ as the subgroup of $\Aut(\bP^1) = \mathrm{PGL}_2(\bC)$ consisting of automorphisms sending the point at infinity to itself. Since the action of $\mathrm{Aff}(\bA^1)$ on $\bA^1$ is 2-transitive, we may identify $\mathrm{Conf}_n(\bA^1)/\mathrm{Aff}(\bA^1)$ with an open subset of $\bA^{n-2}$, given by the complement of the hyperplane arrangement: $\cG = \{z_i = z_j: 1\le i<j\le n-2 \}\cup \{z_i = 0,1\}_{i=1}^{n-2}$. On the other hand, one can regard a point of $\mathrm{Conf}_n(\bA^1)/\mathrm{Aff}(\bA^1)$ as an $(n+1)$-marked genus $0$ curve, $(\bP^1,p_\infty,p_1,\ldots,p_n)$, where we set $p_\infty = \infty, p_1 = 0,$ and $p_2 = 1$. $\mathrm{Conf}_n(\bA^1)/\mathrm{Aff}(\bA^1)$ thus admits a smooth compactification given by the Grothendieck-Knudsen moduli space of genus $0$ curves, $\overline{M}_{0,n+1}$ \cite{Knudsencurves}. Here, ``bubbling'' occurs when marked points collide, expressing limits of paths in $\mathrm{Conf}_n(\bA^1)$ as reducible $(n+1)$-marked nodal genus $0$ curves. Moreover, for all $n\ge 2$ a result of Kapranov expresses $\overline{M}_{0,n+1}$ as an explicit blowup of a hyperplane arrangement in $\bP^{n-2}$ \cite{ChowQuotients}*{Thm. 4.3.3}.

In the present work, we consider the space of configurations of $n$ points in $\bA^1$ up to translation, $\bA^n/\bG_a$, where $\bG_a$ acts diagonally. We construct a complex variety $\Mmscbar_n$ which is a smooth compactification of $\bA^n/\bG_a$. Our original motivation for constructing $\Mmscbar_n$ comes from Bridgeland stability \cite{Br07}: $\Mmscbar_n$ is used in forthcoming work joint work with Daniel Halpern-Leistner \cite{StabSODsII} to model degeneration of stability conditions converging in a partial compactification of $\Stab(\cC)/\bG_a$. We elaborate on this later in the introduction, focusing on the construction of $\Mmscbar_n$ for now.

We regard $p\in \bA^n$ as a configuration of $n+1$ points in $\bP^1$, $(\infty,p_1,\ldots, p_n)$, where the $p_i$ can collide with each other, but not with $\infty$. If $z$ is a coordinate on $\bA^1 \subset \bP^1$, then $dz$ defines a meromorphic differential on $\bP^1$ with a unique pole of order $2$ at $\infty$. $\alpha \in \Aut(\bP^1)$ satisfies $\alpha^*(dz) = dz$ if and only if it restricts to a translation $p\mapsto p+a$ on $\bA^1$. As a consequence, an element of $\bA^n/\bG_a$ is equivalent to an isomorphism class of the data $(\Sigma,p_\infty,\omega,p_1,\ldots, p_n)$ where
\begin{enumerate}
    \item $\Sigma$ is compact Riemann surface of genus $0$;
    \item $\omega$ is a meromorphic differential on $\Sigma$ with a unique pole of order $2$ at $p_\infty$; and 
    \item the marked points $p_1,\ldots, p_n$ may freely collide with each other, but not $p_\infty$.
\end{enumerate}

Denote the space of such data by $\Mmscbar_n^\circ$. To compactify $\bA^n/\bG_a$, we thus need to understand what happens in $\Mmscbar_n^\circ$ when one or more of the $p_i$ tends towards $p_\infty$. As in the case of $\overline{M}_{0,n+1}$, bubbling occurs and the limiting object is a reducible arithmetic genus 0 curve with $n+1$ marked points. Furthermore, the condition (3) above holds.

It is subtle matter to determine what structure the reducible genus $0$ marked curves $(\Sigma,p_\infty,p_1,\ldots, p_n)$ should carry which is a degeneration of the meromorphic differential $\omega$. A natural guess would be to consider a meromorphic section of $\Omega_\Sigma$ with some prescribed zeros and poles, as in the irreducible case. However, the moduli spaces that one obtains in this fashion were already studied by Zahariuc in \cites{Zahariucmarked,Zahariucresolution} under a slightly different guise (see \S5.1 for the relation) and are mildly singular, rendering them insufficient for our intended applications. The insight for how to proceed came from \cite{bcggm}. The input is twofold: (for more precise definitions see \S3)
\begin{enumerate}
    \item[(a)] Augment the combinatorial structure associated to $\Sigma$: instead of considering the dual graph $\Gamma(\Sigma) = (V(\Sigma),E(\Sigma))$ as simply a tree, equip it with the structure of a \emph{rooted level tree}. 
        \begin{enumerate}
            \item[(i)] The root structure is a choice of a distinguished root vertex $v_0$ of $\Gamma(\Sigma)$, which is the one corresponding to the irreducible component containing $p_\infty$. This introduces a partial order on $V(\Sigma)$ by saying $u \subset v$ if $u$ is closer to $v_0$ than $v$. A vertex $u$ is called \emph{terminal} if it is maximal for this partial order.
            \item[(ii)] The level structure $\preceq$ is determined by a function $\lambda:V(\Sigma)\twoheadrightarrow [\ell]$ such that $\lambda(v_0) = 0$, if $u \subset v$ then $\ell(u) <  \ell(v)$, and $\lambda$ maps all terminal vertices to $\ell$. Two vertices are on the same \emph{level} if their value under $\lambda$ coincides. 
        \end{enumerate}
    \item[(b)] Equip $\Sigma$ as above with a meromorphic differential $\omega_v$ on each irreducible component $\Sigma_v$ with a unique pole of order $2$ at the node connecting it to the lower level components of the curve, or at $p_\infty$ when $v = v_0$. We consider such collections $\omega_\bullet=\{\omega_v\}_{v\in V(\Sigma)}$ up to an equivalence relation involving rescaling differentials levelwise by $\bC^*$, except for on the terminal level. Notably, the equivalence relation ensures that the terminal components of $\Sigma$ are equipped with a \emph{bona fide} meromorphic differential.
\end{enumerate}

The space of such data $(\Sigma,\preceq,p_\infty,\omega_\bullet,p_1,\ldots,p_n)$ up to isomorphism is called the space of \emph{multiscaled lines with collision} and denoted $\Mmscbar_n$. The moniker ``with collision'' is to distinguish these spaces from the ones considered in \cite{bcggm} in which marked points are not allowed to collide. Nevertheless, we will simply say ``multiscaled lines'' in what follows for brevity. We give $\Mmscbar_n$ the structure of a complex variety by constructing a system of coordinate charts. With respect to this structure one has:

\begin{thm*}
[\ref{T:spaceconstruction}]
    $\Mmscbar_n$ is a compact complex algebraic manifold of dimension $n-1$ such that $\bA^n/\bG_a \hookrightarrow \Mmscbar_n$ is an open immersion.
\end{thm*}

Despite its relatively \textit{ad hoc} construction, $\Mmscbar_n$ possess some fairly miraculous properties. It is uniformly rational in the sense of \cite{BogomolovBöhning} and by \Cref{P:GactiononMn} carries a natural algebraic action by $G := \bG_a^n/\Delta$ which makes it an equivariant compactification of $\bA^n/\bG_a \cong \bA^{n-1}$ as studied in \cite{HassettTschinkelequiv}. Furthermore, $\Mmscbar_n$ has a pair of natural stratifications. One of them is indexed by the combinatorial data of dual level trees and is quasi-affine, see \Cref{P:stratificationofMn}. The other is by collision of points and reveals some recursive structures on the collection of spaces $\{\Mmscbar_n\}_{n\ge 1}$, see \Cref{P:collisionstratification}.

The space $\Mmscbar_n$ arose fairly early on in the writing of \cite{StabSODsII} and in \S3 we give a preliminary version of the construction, based on joint work with Daniel Halpern-Leistner. We initially attempted to compactify $\bA^n/\bG_a$ by writing down a certain moduli functor, equivalent to the one represented by $\overline{P}_n$ in \cite{Zahariucmarked}. However, in \cite{Zahariucresolution} Zahariuc proves that $\overline{P}_n$ has singularities for all $n\ge 4$, rendering $\overline{P}_n$ unsuitable for our purposes. \cite{Zahariucresolution} gave a possible explanation for the abundance of structure on $\Mmscbar_n$ and inspired this paper. There, Zahariuc proves that an augmented wonderful variety $W_n$, as studied in \cite{Huhetalsemismall}, provides a resolution of singularities $\gamma:W_n\to \overline{P}_n$ for all $n\ge 4$. On the other hand, based on the explicit descriptions of $\Mmscbar_n$ and $\overline{P}_n(\bC)$ (regarded as a complex variety) there is a natural set map $\xi:\Mmscbar_n \to \overline{P}_n(\bC)$ which forgets the non-terminal differentials and level structure and restricts to the identity on $\bA^n/\bG_a$ (see \S5.1). This led to the following question:

\begin{quest}
    Does there exist an isomorphism of algebraic varieties $\Mmscbar_n \to W_n$ such that the natural map $\xi:\Mmscbar_n \to \overline{P}_n(\bC)$ corresponds to the resolution $\gamma$ of \cite{Zahariucresolution}?
\end{quest}

The main results of this work answer this question in the affirmative: 

\begin{thm*}[\ref{T:isothm}]
    There exists an isomorphism $f:\Mmscbar_n \to W_n$ of complex varieties.
\end{thm*}

\begin{thm*}[\ref{T:diagramcommutes}]
    Under the isomorphism $f:\cA_n\to W_n$, one has $\gamma\circ f = \xi$.
\end{thm*}

One might consider the compactification of $\bA^n/\bG_a$ by $\Mmscbar_n$ analogous to the compactification of $\mathrm{Conf}_n(\bA^1)/\mathrm{Aff}(\bA^1)$ by $\overline{M}_{0,n+1}$ as discussed above. $W_n$ is defined explicitly as an iterated blowup of a subspace arrangement of projective space in the sense of \cite{Hucompactification} so that one could also regard \Cref{T:isothm} as an analogue of the blowup description of $\overline{M}_{0,n+1}$ of Kapranov \cite{ChowQuotients}. This blowup description implies projectivity of $\Mmscbar_n$ (\Cref{C:projectivity}) and allows for explicit computation (\Cref{C:chow}) of the ring structure of $\mathrm{CH}^*(\Mmscbar_n)$ by appealing to results of \cite{Huhetalsemismall}. One advantage of considering the resolution of \cite{Zahariucresolution} from this perspective is that one can see rather explicitly what the resolution morphism does. In particular, $\xi$ is tautologically a $G$-equivariant resolution of singularities.

Given that they parametrize similar objects, one should expect a close relationship between the space $\bP\Xi\overline{\mathcal{M}}_{0,n+1}(2,0,\ldots,0)$ of \cite{bcggm} and $\Mmscbar_n$. Thus, structures and properties of $\Mmscbar_n$ may give some insight into the spaces of \emph{loc. cit.} In future work, we will investigate more closely the relationship between $\Mmscbar_n$ and the spaces of \cite{bcggm} 

\subsection*{Connections to Bridgeland stability}

As mentioned above, our main motivation for considering the problem of compactifying $\bA^n/\bG_a$ is the work in preparation \cite{StabSODsII}, which constructs a modular partial compactification of a quotient of the space of Bridgeland stability conditions $\Stab(\cC)$ of a triangulated category $\cC$ \cite{Br07}. For simplicity, we suppose here that $K_0(\cC)$ is of finite rank.\endnote{This assumption is not necessary at all in general, as one typically asks stability conditions to factor through a different finite rank Abelian group $\Lambda$ when $K_0(\cC)$ is of infinite rank. However, it is not important to discuss this here.} There is a natural $\bG_a$-action on $\Stab(\cC)$ and we pass to $\Stab(\cC)/\bG_a$ which is more suitable for partial compactification as the quotient operation removes redundant non-compact directions. Bayer's refinement \cite{Bayer_short} of Bridgeland's deform\-ation theorem implies that given $\sigma \in \Stab(\cC)$ and a set of objects $\{E_1,\ldots, E_n\}$ whose classes generate $K_0(\cC)$, one obtains a holomorphic chart around $\sigma$ by $\tau \mapsto (Z_\tau(E_1),\ldots, Z_\tau(E_n)) \in \bC^n$.

When one passes to $\Stab(\cC)/\bG_a$, the idea needs to be modified since $(Z_\tau(E_1),\ldots, Z_\tau(E_n))$ is not $\bG_a$-invariant. However, one can define logarithms $\logZ_\tau(E_i)$ so that $\logZ_\tau(E_i) - \logZ_\tau(E_j)$ is $\bG_a$-invariant for any pair of indices $1\le i<j\le n$ \cite{StabSODsI}. Consequently, $(\logZ_\tau(E_1),\ldots, \logZ_\tau(E_n)) \in \bC^n/\bG_a$ is well-defined, and gives local coordinates around $[\sigma]\in \Stab(\cC)/\bG_a$. A key idea of the partial compactification in \cite{StabSODsII} is that by analogy to Bridgeland's deformation theorem, a degeneration of a path $\sigma(t):[0,\infty)\to \Stab(\cC)/\bG_a$ towards the boundary should be modeled by convergence of $(\logZ_{\sigma(t)}(E_1),\ldots, \logZ_{\sigma(t)}(E_n))$ in a smooth modular compactification $\Mmscbar_n$ of $\bA^n/\bG_a$. 

Smoothness is an essential feature of $\Mmscbar_n$. Bridgeland's deformation theorem says that deformations of a stability condition $\sigma \in \Stab(\cC)$ are locally controlled by deformations of its central charge $Z\in \Hom(K_0(\cC),\bC)$. By analogy to this, one would like deformations of objects in a partial compactification of $\Stab(\cC)/\bC$ to be modeled on deformations in a smooth space. The model space used in \cite{StabSODsII} is a certain real oriented blowup of $\Mmscbar_n$, which has the structure of a manifold with corners by smoothness of $\Mmscbar_n$. Using this, one is able to formulate a proposal for a generalization of the deformation theorem to a more general class of objects.

A final important feature of $\Mmscbar_n$ is that the period functions $\Pi_{ij}:\bA^n/\bG_a\to \bC$ defined by $\Pi_{ij}(z_1,\ldots, z_n) = z_j-z_i$ extend to define morphisms $\Pi_{ij}:\Mmscbar_n\to \bP^1$, which on curves $\Sigma$ for which $p_i$ and $p_j$ lie on the same terminal component $\Sigma_\tau$, are given by $\Pi_{ij}(\Sigma) = \int_{p_i}^{p_j}\omega_\tau$; here, the integral means the line integral along any chosen path in the smooth locus of $\Sigma_\tau$ connecting $p_i$ to $p_j$. The existence of these period functions on the compactification is necessary for the intended applications, where the period functions record differences between log central charges as $t\to\infty$. This was the main reason for which the related spaces in \cite{bcggm} could not be used.

\subsection*{Acknowledgements and author's note} 
I thank Benjamin Dozier for crucial input during the construction of $\Mmscbar_n$ and Samuel Grushevsky for helpful suggestions and patient expl\-anations of parts of \cite{bcggm}. It is my great pleasure to thank my advisor Daniel Halpern-Leistner for his support and guidance. I am also grateful to Max Hallgren and Andres Fernandez Herrero for helpful conversations and advice. I would also like to thank the Simons Laufer Mathematical Sciences Institute for its hospitality while much of this work was prepared. Finally, I thank Maria Teresa for her unwavering love and encouragement, without which this work would not have been possible.

\S\S3.1-3.2 are a preliminary version of a construction in the forthcoming joint work \cite{StabSODsII} with Daniel Halpern-Leistner. Once the finalized version is available, \S3 of the present work will be rewritten with references to results in that paper. As such, the current work should be regarded as a temporary preliminary version.

\subsection*{Notation}
By default, we work over $\bC$. However, later on in the paper it is essential to distinguish the field $\bC$ from the geometric space $\bA^1$. As such, we use $\bA^1$ and $\bC$ interchangeably to refer to the corresponding geometric spaces with the exception of in \S5. $\bG_a$ is used for the additive group, typically considered only over $\bC$. When $f$ is a function valued in a field (rational, regular, or otherwise) on a space $X$, we write $Z(f)$ for its zero locus and $D(f) = Z(f)^c$.

For $n\ge 0$, we write $[n] = \{0<\cdots<n\}$, which is usually regarded as a totally ordered set. If $X$ is a subset of $\{1,\ldots,n\}$, we sometimes write $\mu(X)$ as a shorthand for $\min X$.

\section{Preliminaries on discrete structures}

In this section, we recall some combinatorial notions that will be crucial in the sequel. Partitions and trees will be used throughout the paper, while matroids will be used in \Cref{C:chow} to give explicit generators and relations for the Chow rings of the spaces $\Mmscbar_n$.

\subsection{Level trees}
Level trees are not a new structure, see for example \cite{Ulyanovpoly}*{\S 2}. However, our original motivation to consider these structures came from \cite{bcggm}. We briefly recall some definitions and constructions involving trees here.

\begin{defn}
\label{D:treestuff}
    A \emph{tree} is an undirected graph $(V,E)$ in which every pair of vertices may be joined by a unique simple path.\endnote{By a simple path we mean a sequence of edges $e_1,\ldots, e_k$ in $E$ so that $e_i$ and $e_{i+1}$ have one vertex in common for all $1\le i \le k-1$.}
    \begin{enumerate} 
        \item A \emph{rooted tree} is a tree $(V,E,v_0)$ with a distinguished vertex $v_0\in V$ called the \emph{root}.
        \item Let $X$ be a set. An $X$\emph{-marked rooted} tree $(V,E,v_0,h)$ is a rooted tree $(V,E,v_0)$ equipped with a function $h:X\to V$. We call $(V,E,v_0,h)$ $n$-marked when $X = \{1,\ldots, n\}$.
    \end{enumerate}
\end{defn}

\begin{defn}
\label{D:rootedporder}
    Associated to a rooted tree $(V,E,v_0)$ is a partial order $\subseteq$ on $V$: we say that $v\subseteq u$ if $v$ lies on the unique simple path connecting $u$ to $v_0$. Maximal elements for $\subseteq$ are called \emph{terminal} vertices. Given $u,v\in V$, their \emph{meet} $u\wedge v$ is the maximal vertex with respect to $\subseteq$ where their unique simple paths to $v_0$ intersect. 
\end{defn}

\begin{defn}
\label{D:level}
    Given a tree $(V,E)$, a \emph{level function} is a function $\lambda:V\to \bZ$ whose image is $[\ell]$. A tree equipped with a level function is called a \emph{level tree}. We call $\ell$ as above the \emph{length} of the level tree. When the tree is rooted, we require that $\lambda:(V,\subseteq) \to (\bZ,\le)$ be a homomorphism\endnote{That is, if $v\subseteq w$, then $\lambda(v)\le \lambda(w)$.} of posets such that the terminal vertices are all mapped to $\ell$ and such that $\lambda^{-1}(0) = \{v_0\}$. For any $k\in [\ell]$, $\lambda^{-1}(k)$ is the set of level $k$ vertices.
\end{defn}

\begin{defn}
\label{D:dualleveltree}
    An $n$-marked rooted level tree $\Gamma = (V,E,v_0,h,\lambda)$ is called a \emph{dual level tree} if 
    \begin{enumerate} 
    \item the image of $h$ is the set of terminal vertices of $V$;\endnote{This will translate to a stability condition when we construct $\Mmscbar_n$: terminal components of the corresponding curves will be required to have at least one marked point.} and
    \item for each non-terminal vertex $v$, there are at least two edges connecting it to higher level vertices.
    \end{enumerate}
    We denote the set of all such trees by $\Delta(n)$. 
\end{defn}

In what follows, we will often write $\Gamma$ to mean a tree with any of the structures defined above, however it should be clear from the context which we are referring to.

\subsection{Partitions}

Let $(X,\le)$ denote a partially ordered set. A \emph{chain} in $X$ is a totally ordered subset of $X$. $\mathfrak{Ch}(X)$ denotes the set of chains in $X$. A poset $(X,\le)$ is a \emph{lattice} if any pair of elements $x,y\in X$ has a least upper bound $x\vee y$, called their \emph{join}, and a greatest lower bound $x\wedge y$, called their \emph{meet}. 

\begin{defn} 
A \emph{partition} of $\{1,\ldots,n\}$ is a collection of disjoint subsets of $\{1,\ldots,n\}$, called \emph{blocks}, whose union is $\{1,\ldots,n\}$. $L_{n}$ denotes the set of such partitions. For $\rho \in L_{n}$, let $B(\rho)$ denote the set of blocks of $\rho$.\endnote{In other articles, such as \cite{Zahariucresolution}, $L_{[n]}$ is written for what we have denoted by $L_n$. We chose the latter notation to avoid confusion with partitions of $[n] = \{0<\cdots<n\}$.}
\end{defn}

There is a natural partial order on $L_{n}$; given $\eta,\pi\in L_{n}$, write $\eta\le \pi$ if $\pi$ \emph{refines} $\eta$; i.e. if every block of $\pi$ is contained in a block of $\eta$. In this case, there is an induced surjection $B(\pi)\twoheadrightarrow B(\eta)$ mapping each block of $\pi$ to the block of $\eta$ containing it. We may also regard $\rho \in L_n$ as an equivalence relation on $\{1,\ldots, n\}$: for $1\le i,j\le n$, we say $i\sim^\rho j$ if $i$ and $j$ are in the same block of $\rho$. Given $\rho \in L_n$ and $b\in B(\rho)$, we write $\mu(b) = \min(b)$ for brevity.

\begin{lem}
$(L_n,\le)$ is a lattice with unique maximal element $\top = 1|2|\cdots|n$ and minimal element $\bot = 12\cdots n$ (i.e. the trivial partition).
\end{lem}

\begin{proof}
    Left as an exercise (cf. \cite{Birkhofflattice}).
\end{proof}

We will write $L_n^- = L_n\setminus \{\bot\}$.

\begin{lem}
\label{L:comparabilitycriteria}
    For any $\eta,\pi \in L_n$, 
    \begin{enumerate}
        \item $\eta\le \pi$ is equivalent to the statement $i\sim^\pi j \Rightarrow i\sim^\eta j$; and
        \item $\eta$ and $\pi$ are incomparable if and only if there exist $i,j,k,\ell \in \{1,\ldots,n\}$ such that $i\sim^\eta j$ but $i\not\sim^\pi j$ and $k\sim^\pi \ell$ but $k\not\sim^\eta \ell$.
    \end{enumerate}
\end{lem}

\begin{proof}
    $\eta \le \pi$ iff for all $b \in B(\pi)$ there exists $b' \in B(\eta)$ such that $b\subset b'$. This in turn is equivalent to the statement $i\sim^\pi j \Rightarrow i\sim^\eta j$. $\eta$ and $\pi$ being incomparable is the negation of both $\eta\le \pi$ and $\eta\ge \pi$, so (2) follows from (1).
\end{proof}

\begin{const}
\label{Const:treesandpartitions}
    Let $\Delta(n)$ be as in \Cref{D:dualleveltree}. We define a function $\Gamma:\mathfrak{Ch}(L_n^-) \to \Delta(n)$. Given $\rho_\bullet = \{\rho_1< \cdots < \rho_\ell\} \in \mathfrak{Ch}(L_n^-)$ we define $\Gamma(\rho_\bullet)$: on level $0$ there is a single vertex; for $1\le k \le \ell$, the level $k$ vertex set is $B(\rho_k)$. $b\in B(\rho_k)$ is connected by an edge to $b'\in B(\rho_{k-1})$ iff $b\subset b'$. For all $1\le k \le \ell-1$, beginning with $k = \ell - 1$, delete any level $k$ vertex $v$ and all edges attached to it if it is connected to exactly one higher vertex $v'$ and one level $k-1$ vertex $v''$; add a ``long'' edge connecting $v'$ and $v''$. Finally, define $h:\{1,\ldots,n\}\to B(\rho_\ell)$ by $h(i) = b$ for $b\ni i$.

    We define $\mathfrak{c}:\Delta(n)\to \mathfrak{Ch}(L_n^-)$ by assigning to $\Gamma$ with level function $\lambda:V\twoheadrightarrow [\ell]$ the chain $\{\rho_1<\cdots<\rho_\ell\}\in \mathfrak{Ch}(L_n^-)$ where $p \sim^{\rho_i} q$ if and only if $\lambda(h(p)\wedge h(q))\ge i$.\endnote{Note that $\rho_i<\rho_{i+1}$ for all $1 \le i \le \ell-1$: if $p\sim^{\rho_{i+1}}q$ then $\lambda(h(p)\wedge h(q)) \ge i+1\ge i$ so that $p\sim^{\rho_i} q$.}
\end{const}

\begin{prop}
\label{P:treeversuschain}
    $\Gamma: \mathfrak{Ch}(L_n^-) \to \Delta(n)$ and $\mathfrak{c}:\Delta(n)\to \mathfrak{Ch}(L_n^-)$ are mutually inverse bijections.
\end{prop}

\begin{proof}
    We leave this to the reader - see \cite{Ulyanovpoly}*{Fig. 7} for a visual explanation. Be warned, however, that our convention for drawing trees is the opposite of that in \emph{loc. cit.} - the roots of our trees are at the bottom of the tree.
\end{proof}

\subsection{Matroids}  

\begin{defn}
\label{D:matroid}
    A \emph{matroid} on a finite set $E$ is a collection of subsets $\cI$ of $E$, called the \emph{independent subsets}, such that 
    \begin{enumerate}
        \item $\varnothing \in \cI$;
        \item if $I\in \cI$ and $I'\subset I$, then $I'\in \cI$; and 
        \item if $I_1$ and $I_2$ are in $\cI$ with $\lvert I_1\rvert < \lvert I_2\rvert$ then there exists $e\in I_2\setminus I_1$ such that $I_1\cup \{e\}\in \cI$. 
    \end{enumerate}
\end{defn}

There are a number of equivalent ways to specify the data of a matroid on a finite set $E$; see \cite{Oxleymatroid} for more information.

\begin{ex} 
    Let $G = (V,E)$ be a graph. The \emph{graphic matroid} associated to $G$, denoted $M(G)$, has as its underlying set $E$, and as independent sets the forests in $G$, denoted $F$.\endnote{A forest $F$ is a graph in which any pair of vertices can be connected by at most one simple path. It follows that a forest is a disjoint union of trees, as the name suggests.} It is an exercise to verify that $M(G) = (E,F)$ defines a matroid in the sense of \Cref{D:matroid}.\endnote{Here we are committing a standard abuse of notation by conflating the set of edges in the forest with the forest they span. $\varnothing$ is vacuously a forest. If $F$ is a forest and $F'\subset F$ a subset, then any two vertices in the graph spanned by $F'$ are connected by at most one path since the edge set is a subset of that of $F$. If $F_1$ and $F_2$ are forests with $|F_1|<|F_2|$ then in fact any edge $\in F_2\setminus F_1$ gives rise to a new forest $F_1\cup \{e\}$.}
\end{ex}

Suppose given a graph $G = (V,E)$. For $A\subset E$ let $v(A)$ denote the number of vertices contained in the subgraph of $G$ spanned by $A$. In the same notation, $\omega(A)$ denotes the number of connected components of the subgraph of $G$ spanned by $A$. $M(G)$ has a \emph{rank} function $r: 2^E\to \bZ$ given by $r(A) = v(A) - \omega(A)$ and \emph{corank} given by $\mathrm{corank}(A) = v(E) - v(A) + \omega(A) = |V| - 1 - v(A) + \omega(A)$.

Given $A\subset E$, we define its \emph{closure} $\rm{cl}(A) = \{x\in E: r(A\cup \{x\}) = r(A)\}$. If $A = \rm{cl}(A)$ then it is called a \emph{flat} of the matroid. The collection of flats of a matroid $M$ forms a lattice $L(M)$, partially ordered by inclusion.

\begin{ex}
\label{E:matroidofKn}
    Of particular importance to us is the complete graph on $n$ vertices, $K_n$. Its vertex set is $\{1,\ldots,n\}$ and a unique edge connects any pair of distinct vertices. Let $G$ denote a subgraph of $K_n$ such that $V(G) = V(K_n)$ and $G$ has as connected components complete graphs, i.e. graphs isomorphic to $K_r$ for some $1\le r\le n$. It is an exercise to see that the flats of $M(K_n)$ are exactly of the form $E(G)$ for such $G$.\endnote{It suffices to compute the closure of a set of edges $A\subset E$. This consists of adding edges to $A$; if $e$ increases the number of vertices in the graph spanned by $A\cup \{e\}$, then $r(A\cup\{e\}) > r(A)$, so any added edge $e$ must connect vertices already spanned by $A$. On the other hand, adding an edge that unites two connected components of the subgraph spanned by $A$ decreases the rank. So, $\mathrm{cl}(A)$ is obtained by adding all edges that connect two vertices on the same connected component of the subgraph spanned by $A$. The result of this is a disjoint union of complete graphs.}
\end{ex}

If $A$ is a flat of $M_n := M(K_n)$, we define $\rho\in L_n$ given by saying $i\sim^\rho j$ if and only if $i$ and $j$ lie on the same connected component of the subgraph spanned by $A$. This gives a map $L(M_n)\to L_n$. 

\begin{prop}
\label{P:lattice}
    $L(M_n)\to L_n$ is an anti-isomorphism of lattices such that the rank function on $L(M)$ corresponds to the function $r(\rho) = n - |B(\rho)|$ and $\mathrm{corank}(\rho) = |B(\rho)| - 1$.
\end{prop}

\begin{proof}
    We can identify $\rho \in L_n$ with the equivalence relation that it defines, which can be regarded as a subset of $\{1,\ldots,n\}^2$. A flat in $M_n$ is a collection of edges in $K_n$, which is also a subset of $\{1,\ldots,n\}^2$. The map $L(M_n)\to L_n$ is just the identification of these two perspectives. $L(M_n)$ is partially ordered by inclusion, but this is easily seen to be the same as the coarsening partial order (i.e. the opposite of refinement as defined above).\endnote{Indeed, given a relations $R$ and $R'$ on a set $X$ regarded as a subsets of $X^2$, $R\subset R'$ means that every equivalence class of $R$ is contained in an equivalence class for $R'$. That is, $R$ refines $R'$.}
\end{proof}

In spite of this anti-isomorphism, we continue to regard $L_n$ as a poset ordered by refinement. This will be more natural in what follows.

\subsection{Subspace arrangements from partitions}

$V = \bC^n/\bG_a \oplus \bC$, where $\bC^n/\bG_a$ has coordinates $P_{12},\ldots, P_{1n}$, where $P_{ij}(p) = z_j(p) - z_i(p)$ for each $i,j$, and the copy of $\bC$ has coordinate $t$.

\begin{defn}
\label{D:subspacearrangement}
    For $\rho \in L_n \setminus \{\top,\bot\}$, $H_\rho\subset \bP(V)$ denotes the linear subspace $Z(\{P_{ij} = 0 : i\sim^\rho j\} \cup \{t\})$. Define $H_\top$ by $Z(t)$ and $H_\bot = \varnothing$. Let $\cH = \{H_\rho:\rho\in L_n\}$, regarded as a poset ordered by inclusion.
\end{defn}

\begin{defn}
\label{D:cleanint}
    A pair of smooth closed subvarieties $U$ and $V$ of a smooth variety $W$ are said to intersect \emph{cleanly} if
    \begin{enumerate} 
        \item the scheme theoretic intersection $U\cap V$ is smooth; and 
        \item for each $p\in U\cap V$ one has $T_p(U\cap V) = T_p(U) \cap T_p(V)$. 
    \end{enumerate}
\end{defn}

\begin{lem}
\label{L:propertiesofHs}
    $L_n \to \cH$ given by $\rho \mapsto H_\rho$ is an isomorphism of posets. Moreover,
    \begin{enumerate}
        \item each $H_\rho$ is a smooth subvariety of $\bP(V)$;
        \item for all $\rho,\pi \in L_n$, $H_\rho$ and $H_\pi$ intersect cleanly;
        \item $\dim H_\pi = |B(\pi)| - 2$;\endnote{Note that by convention, the empty set has dimension $-1$ and so the result holds even for $\pi = \bot$.} and 
        \item $H_\rho \cap H_\pi = H_{\rho \wedge \pi}$.
    \end{enumerate}
\end{lem}

\begin{proof}
    $\rho \mapsto H_\rho$ is by construction bijective. If $\rho_1<\rho_2$, then $i \sim^{\rho_2} j$ implies $i\sim^{\rho_1} j$ by \Cref{L:comparabilitycriteria} and thus $H_{\rho_1}\subset H_{\rho_2}$. (1) is clear as these are linear subvarieties of $\bP(V)$. (2) follows from working locally in affine coordinates. Write $\{b_1,\ldots, b_k\} = B(\pi)$; (3) is obtained by computing that there are $\sum_{j=1}^k (|b_j| -1) + 1$ equations cutting out $H_\pi$ irredundantly. Therefore, its dimension is $n-2 + k -\sum_j |b_j| = |B(\pi)| - 2$. To prove (4), note that $H_\rho\cap H_\pi = Z(\{P_{ij}:i\sim^\rho j \text{ and } i\sim^\pi j\}\cup \{t\}) \subset H_{\pi\wedge \rho}$. Then, $\pi\wedge \rho \le \rho,\pi$ implies $H_{\pi\wedge \rho} \subset H_\rho \cap H_\pi$.
\end{proof}

\begin{notn}
\label{N:dimpartition}
    For $\pi \in L_n$, we write $\dim \pi := \dim H_\pi = |B(\pi)| - 2 = \mathrm{corank}(\rho) - 1$ (by \Cref{L:propertiesofHs}) and $c(\pi) = \rm{codim}_{\bP(V)} H_\pi = n + 1 - |B(\pi)| = r(\pi) + 1$. $\dim \pi$ will be a relevant inductive parameter in \S4. For $k\in \bZ_{\ge 0}$, $(L_n)_{\le k}$ denotes the poset of partitions of dimension $\le k$.
\end{notn}

\begin{defn}
    Given $\pi \in L_n$, we define $H_\pi^\circ = H_\pi \setminus \bigcup_{\rho<\pi} H_\rho$
\end{defn}

\begin{lem}
\label{L:stratificationofH}
    For all $\pi \in L_n$, one has $H_\pi = \bigsqcup_{\rho\le \pi}H_\rho^\circ$.
\end{lem}

\begin{proof}
    Suppose given $\rho,\eta \le \pi$ distinct. By \Cref{L:propertiesofHs}, $H_\rho \cap H_\eta = H_{\rho\wedge \eta}$, and as $\rho\wedge \eta < \rho$ and $\eta$, one sees $H_\rho^\circ \cap H_\eta^\circ = \varnothing$. For $x\in H_\pi$, consider $\Lambda_x = \{\bot<\rho \le \pi: x\in H_\rho\}$. $\Lambda_x$ has minimal elements since $L_n^-$ is a finite poset. However, if $\mu$ and $\nu$ are two such minimal elements, we have $x\in H_\mu \cap H_\nu = H_{\mu\wedge \nu}$, and consequently $\nu = \mu \wedge \nu = \mu$. Therefore, $x\in H_\mu^\circ$ and we have $H_\pi^\circ \subset \bigcup_{\rho \le \pi} H_\rho^{\circ}$. The other inclusion is automatic.
\end{proof}

\begin{defn}
\label{D:Nsets}
    Suppose $\rho < \pi$ are given in $L_n$. Define 
    \[ 
        N_{\rho|\pi} = \{(\mu(b),\mu(b')):b\ne b'\in B(\pi), b\cup b'\subset B\in B(\rho),\text{ and }\mu(B) = \mu(b)\}. 
    \]
    To clarify, $b\cup b'\subset B \in B(\rho)$ means that there exists a $B\in B(\rho)$ in which both $b$ and $b'$ are contained. By convention, $N_{\rho|\rho} = \varnothing$. If $\rho<\pi <\top$, let $P_{\rho|\pi} = \{P_{ij}: (i,j)\in N_{\rho|\pi}\}$. Let $P_{\rho|\top} = \{P_{ij}:(i,j)\in N_{\rho|\pi}\} \cup \{t\} \subset H^0(\bP(V),\cO(1))$. 
\end{defn}

\begin{lem}
\label{L:cuttingout}
    Let $\rho<\pi \in L_n$ be given
    \begin{enumerate}   
        \item if $\rho \ne \bot$, $N_{\rho|\pi}$ is a minimal set of linear equations defining $H_\rho$ as a subvariety of $H_{\pi}$; and 
        \item $N_{\bot|\pi}$ restricts to a set of homogeneous coordinates for $H_\pi$. 
    \end{enumerate}
\end{lem}

\begin{proof}
    By \Cref{L:propertiesofHs}, $H_\rho \subset H_\pi$ and $\dim H_\pi - \dim H_\rho = |B(\pi)| - |B(\rho)|$. However, one can verify that $|B(\pi)| - |B(\rho)| = |N_{\rho|\pi}|$ and that $N_{\rho|\pi}$ is linearly independent. This proves (1). (2) follows from the following: Given $V$ a vector space with $f_1,\ldots, f_n$ a basis of $V^*$, then $f_{c+1},\ldots, f_n$ restrict to a basis of $W^*$ where $W = Z(f_1,\ldots, f_c)$. 
\end{proof}

\section{Multiscaled lines with collision}

\S\S 3.1-3.2 are based on part of a joint work with Daniel Halpern-Leistner \cite{StabSODsII}.

\subsection{Multiscaled lines}

\begin{defn}
\label{D:multiscaledline}
A \emph{multiscaled line} is a tuple $(\Sigma,p_\infty,\preceq,\omega_\bullet)$, where
\begin{enumerate}
    \item $\Sigma$ is a connected nodal complex genus $0$ curve\endnote{This is equivalent to saying that every component of $\Sigma$ is isomorphic to $\bP^1$, and the dual graph of $\Sigma$ is a tree.} with marked smooth point $p_\infty$. We let $\Gamma(\Sigma) = (V(\Sigma),E(\Sigma))$ denote the dual graph of $\Sigma$, and let $v_0 \in V(\Sigma)$ denote the vertex corresponding to the component containing $p_\infty$. This gives $\Gamma(\Sigma)$ the structure of a rooted tree with partial order $\subseteq$ (see \Cref{D:treestuff} and \Cref{D:rootedporder}).
    \item $\preceq$ is a total preorder on the set $V(\Sigma)$ of irreducible components of $\Sigma$ such that:
    \begin{enumerate}
        \item for all $v \in V(\Sigma) \setminus \{v_0\}$, there is a unique $w\in V(\Sigma)$ that is adjacent to $v$ with $w\prec v$. The edge from $v$ to $w$ is called the edge \emph{descending} from $v$; and
        \item $v \sim w$, meaning $v \preceq w$ and $w \preceq v$, for any two vertices $v,w \in V(\Sigma)$ that are maximal with respect to $\subseteq$.
    \end{enumerate}
    \item For every $v \in V(\Sigma)$, $\omega_v$ is a meromorphic section of $\Omega_{\Sigma_v}$ with a pole of order $2$ at the node corresponding to the descending edge from $v$ when $v \neq v_0$, or the marked point $p_\infty$ when $v=v_0$, and no other zeros or poles.
\end{enumerate}
A \emph{complex projective isomorphism} of multiscaled lines is an isomorphism of nodal curves $f : \Sigma \to \Sigma'$ that preserves the respective preorders and marked points such that $f^\ast(\omega'_v) = c_v \omega_v$ for some constant $c_v \in \bC^\ast$, $c_v = c_w$ whenever $v \sim w$, and $c_v=1$ if $v$ is maximal.\endnote{It follows from condition (3) of \Cref{D:multiscaledline} that the $\omega_v$ are all of the form $\lambda \cdot dz$ for a chosen affine coordinate $z:C_v\setminus \{n_v\}\to \bb{C}$, where $n_v$ is the descending node of $v$, and $\lambda \in \bC^*$. A complex projective isomorphism of multiscaled lines is an isomorphism of nodal curves that does not preserve the data of any $\omega_v$ for $v$ non-maximal, but does preserve the information of $\omega_v/\omega_{v'}$ for any $v$ and $v'$. }
\end{defn}

Condition (2) implies that $u \subseteq v \Rightarrow u \preceq v$, that \emph{every} edge of $\Gamma$ is the descending edge for one of its vertices, and that $v_0$ is the unique minimum with respect to the preorder $\preceq$.\endnote{As $V(\Gamma)$ equipped with $\subseteq$ is a finite poset, one can define $\rm{depth}(v)$ to be the maximum length of chains $v\subsetneq v_1\subsetneq \cdots$. We prove $v\subseteq u\Rightarrow v \preceq u$ by inducting on $\mathrm{depth}(v)$. For depth $0$, $v$ is maximal for $\subseteq$ and so $v = u$. Suppose the claim has been proven for $\mathrm{depth}<k$ and let $w$ of depth $k$ be given. Suppose $w\subset v$ and that $v$ and $w$ are joined by an edge. It follows that $\mathrm{depth}(v)\le k-1$. If $v\subseteq u$ then $v\preceq u$ by induction. On the other hand, there can be at most one vertex $w$ such that $w\subset v$ since the graph is a tree. Consequently, $w$ is the descending edge from $v$ and $w\prec v$. So, for any $w\subseteq v \subseteq u$ we have $w\prec v\preceq u$ and the result follows. 

Since $v_0\subseteq u$ for all $u \in V(\Gamma)$, it now follows that $v_0$ is a minimum element with respect to $\preceq$. Lastly, let $e\in E(\Gamma)$ be given connecting vertices $v$ and $w$. Then $v\subset w$ or $w\subset v$. If $v\subset w$, $v\preceq w$ and so $e$ is the descending edge from $w$ to $v$.} Also, terminal vertices (\Cref{D:rootedporder}) form an equivalence class under $\sim$.\endnote{We consider some examples of totally preordered rooted level trees. There is a unique length $0$ totally preordered rooted level tree consisting of a single vertex and no edges. It is easy to write down many length $1$ examples:
\begin{center}
\begin{tikzpicture}[node distance =4 cm and 5cm ,on grid ,
semithick ,
state/.style ={ circle ,top color =white , draw, , minimum width =1 cm}]
\filldraw[black] (0,0) circle (1.5pt) ;
\filldraw[black] (0,-.75) circle (1.5pt) ; 
\draw[black, thick] (0,0) -- (0,-.75) ;
\end{tikzpicture}
\quad
\begin{tikzpicture}[node distance =4 cm and 5cm ,on grid ,
semithick ,
state/.style ={ circle ,top color =white , draw, , minimum width =1 cm}]
\filldraw[black] (0,-.75) circle (1.5pt) ; 
\filldraw[black] (-.5,0) circle (1.5pt) ; 
\filldraw[black] (.5,0) circle (1.5pt) ; 
\draw[black, thick] (0,-.75) -- (-.5,0) ; 
\draw[black, thick] (0,-.75) -- (.5,0) ; 
\end{tikzpicture}
\quad
\begin{tikzpicture}[node distance =4 cm and 5cm ,on grid ,
semithick ,
state/.style ={ circle ,top color =white , draw, , minimum width =1 cm}]
\filldraw[black] (0,0) circle (1.5pt) ; 
\filldraw[black] (-.5,0) circle (1.5pt) ; 
\filldraw[black] (0,-.75) circle (1.5 pt) ;
\filldraw[black] (.5,0) circle (1.5pt) ; 
\draw[black, thick] (0,-.75) -- (-.5,0) ; 
\draw[black, thick] (0,-.75) -- (0,0) ; 
\draw[black, thick] (0,-.75) -- (.5,0) ; 
\end{tikzpicture}
\quad
\begin{tikzpicture}[node distance =4 cm and 5cm ,on grid ,
semithick ,
state/.style ={ circle ,top color =white , draw, , minimum width =1 cm}]
\filldraw[black] (0,-.75) circle (1.5pt) ; 
\filldraw[black] (-.5,0) circle (1.5pt) ; 
\filldraw[black] (-.17,0) circle (1.5pt) ; 
\filldraw[black] (.17,0) circle (1.5pt) ; 
\filldraw[black] (.5,0) circle (1.5pt) ; 
\draw[black, thick] (0,-.75) -- (-.5,0) ; 
\draw[black, thick] (0,-.75) -- (-.17,0) ; 
\draw[black, thick] (0,-.75) -- (.17,0) ; 
\draw[black, thick] (0,-.75) -- (.5,0) ; 
\end{tikzpicture}
\end{center}
Indeed, the length $1$ examples are classified by the number $n\in \bN$ of terminal vertices. The $n=1$ case will not occur as the dual tree of a stable $n$-marked multiscaled line. In these graphs and the subsequent ones, the root node is placed at the bottom. $v\preceq w$ if $w$ has $y$ coordinate at least as large as that of $v$. The lowest level is regarded as level $0$, the next lowest as level $1$, etc. Here are some examples of length $2$ totally preordered rooted level trees:

\begin{center}
\begin{tikzpicture}[node distance =4 cm and 5cm ,on grid ,
semithick ,
state/.style ={ circle ,top color =white , draw, , minimum width =1 cm}]
\filldraw[black] (0,-1.5) circle (1.5pt) ;
\filldraw[black] (-.5,-.75) circle (1.5pt) ;
\filldraw[black] (.5,-.75) circle (1.5pt) ;
\filldraw[black] (-.75,0) circle (1.5pt) ; 
\filldraw[black] (-.25,0) circle (1.5pt) ; 
\filldraw[black] (.25,0) circle (1.5pt) ; 
\filldraw[black] (.75,0) circle (1.5pt) ; 
\filldraw[black] (-.5,0) circle (1.5pt) ;

\draw[black, thick] (0,-1.5) -- (-.5, -.75) ;
\draw[black, thick] (0,-1.5) -- (.5, -.75) ;
\draw[black, thick] (.5,-.75) -- (.25, 0) ;
\draw[black, thick] (.5,-.75) -- (.75, 0) ;
\draw[black, thick] (-.5,-.75) -- (-.25, 0) ;
\draw[black, thick] (-.5,-.75) -- (-.75, 0) ;
\draw[black, thick] (-.5,-.75) -- (-.5,0) ; 
\end{tikzpicture}
\quad 
\begin{tikzpicture}[node distance =4 cm and 5cm ,on grid ,
semithick ,
state/.style ={ circle ,top color =white , draw, , minimum width =1 cm}]
\filldraw[black] (0,-1.5) circle (1.5pt) ;
\filldraw[black] (-.5,-.75) circle (1.5pt) ;
\filldraw[black] (-.75,0) circle (1.5pt) ; 
\filldraw[black] (-.5,0) circle (1.5pt) ; 
\filldraw[black] (-.25,0) circle (1.5pt) ; 
\filldraw[black] (.5,0) circle (1.5pt) ; 
\draw[black, thick] (0,-1.5) -- (-.5, -.75) ;
\draw[black, thick] (-.5,-.75) -- (-.25, 0) ;
\draw[black, thick] (-.5,-.75) -- (-.5, 0) ;
\draw[black, thick] (-.5,-.75) -- (-.75, 0) ;
\draw[black, thick] (0,-1.5) -- (.5,0) ; 
\end{tikzpicture}
\end{center}
The first tree admits three level structures, while the second admits only one. Here are the possible level structures on the first tree:
\begin{center}
\begin{tikzpicture}[node distance =4 cm and 5cm ,on grid ,
semithick ,
state/.style ={ circle ,top color =white , draw, , minimum width =1 cm}]
\filldraw[black] (0,-1.5) circle (1.5pt) ;
\filldraw[black] (-.5,-.75) circle (1.5pt) ;
\filldraw[black] (.5,-.75) circle (1.5pt) ;
\filldraw[black] (-.75,0) circle (1.5pt) ; 
\filldraw[black] (-.25,0) circle (1.5pt) ; 
\filldraw[black] (.25,0) circle (1.5pt) ; 
\filldraw[black] (.75,0) circle (1.5pt) ; 
\filldraw[black] (-.5,0) circle (1.5pt) ;

\draw[black, thick] (0,-1.5) -- (-.5, -.75) ;
\draw[black, thick] (0,-1.5) -- (.5, -.75) ;
\draw[black, thick] (.5,-.75) -- (.25, 0) ;
\draw[black, thick] (.5,-.75) -- (.75, 0) ;
\draw[black, thick] (-.5,-.75) -- (-.25, 0) ;
\draw[black, thick] (-.5,-.75) -- (-.75, 0) ;
\draw[black, thick] (-.5,-.75) -- (-.5,0) ; 
\end{tikzpicture}
\quad
\begin{tikzpicture}[node distance =4 cm and 5cm ,on grid ,
semithick ,
state/.style ={ circle ,top color =white , draw, , minimum width =1 cm}]
\filldraw[black] (0,-1.5) circle (1.5pt) ;
\filldraw[black] (-.5,-1) circle (1.5pt) ;
\filldraw[black] (.5,-.5) circle (1.5pt) ;
\filldraw[black] (-.75,0) circle (1.5pt) ; 
\filldraw[black] (-.25,0) circle (1.5pt) ; 
\filldraw[black] (.25,0) circle (1.5pt) ; 
\filldraw[black] (.75,0) circle (1.5pt) ; 
\filldraw[black] (-.5,0) circle (1.5pt) ; 

\draw[black, thick] (0,-1.5) -- (-.5,-1) ; 
\draw[black, thick] (0,-1.5) -- (.5, -.5) ;
\draw[black, thick] (.5,-.5) -- (.25,0) ;
\draw[black, thick] (.5,-.5) -- (.75,0) ;
\draw[black, thick] (-.5,-1) -- (-.25,0) ;
\draw[black, thick] (-.5,-1) -- (-.5,0) ;
\draw[black, thick] (-.5,-1) -- (-.75,0) ;
\end{tikzpicture}
\quad 
\begin{tikzpicture}[node distance =4 cm and 5cm ,on grid ,
semithick ,
state/.style ={ circle ,top color =white , draw, , minimum width =1 cm}]
\filldraw[black] (0,-1.5) circle (1.5pt) ;
\filldraw[black] (.5,-1) circle (1.5pt) ;
\filldraw[black] (-.5,-.5) circle (1.5pt) ;
\filldraw[black] (.75,0) circle (1.5pt) ; 
\filldraw[black] (.25,0) circle (1.5pt) ; 
\filldraw[black] (-.25,0) circle (1.5pt) ; 
\filldraw[black] (-.75,0) circle (1.5pt) ; 
\filldraw[black] (-.5,0) circle (1.5pt) ; 

\draw[black, thick] (.5,-1) -- (.25,0) ; 
\draw[black, thick] (0,-1.5) -- (-.5, -.5) ;
\draw[black, thick] (0,-1.5) -- (.5, -1) ;
\draw[black, thick] (-.5,-.5) -- (-.25,0) ;
\draw[black, thick] (-.5,-.5) -- (-.5,0) ;
\draw[black, thick] (-.5,-.5) -- (-.75,0) ;
\draw[black, thick] (.5,-1) -- (.75,0) ;
\end{tikzpicture}
\end{center}

}

\begin{defn}\label{D:specialization_totally_preordered_rooted_tree}
A \emph{contraction} of totally preordered rooted trees $f : (\Gamma,\preceq,v_0) \twoheadrightarrow (\Gamma',\preceq',v_0')$ is a surjection $f:V(\Gamma) \to V(\Gamma')$ such that\endnote{This is a slight abuse of notation. We use the same symbol $f$ to denote both the contraction, and the underlying surjection on vertex sets.}
\begin{enumerate} 
    \item $v\preceq w$ implies $f(v)\preceq f(w)$;
    \item for adjacent vertices $v,w \in V(\Gamma)$, $f(v)$ and $f(w)$ are either equal or adjacent.
    \item for any $v \in V(\Gamma')$, $f^{-1}(v)$ spans a connected subgraph.
\end{enumerate}
Condition (1) implies that $f(v_0) = v_0'$.\endnote{Indeed, $v_0$ is uniquely characterized by the property that $v_0\preceq w$ for all $w\in \Gamma(V)$ and likewise for $v_0'$. Therefore, by (1) it follows that $f(v_0)\preceq w'$ for all $w'\in V(\Gamma')$ and thus that $f(v_0) = v_0'$. Note that here we used the surjectivity of $f:V(\Gamma)\to V(\Gamma')$.}
\end{defn}

\begin{lem}\label{L:specializations}
For any totally preordered rooted tree $(\Gamma,\preceq,v_0)$ and any order preserving surjection of totally ordered sets
\[
V(\Gamma) /{\sim} \twoheadrightarrow [n] := \{0<\cdots<n\},
\]
there is a unique contraction $f:(\Gamma, \preceq, v_0) \twoheadrightarrow (\Gamma', \preceq', v'_0)$ such that $V(\Gamma')/{\sim'}$ is isomorphic to $[n]$ under $V(\Gamma)/{\sim}$.\endnote{By this, we mean that there is an order-preserving bijection $V(\Gamma')/{\sim} \to [n]$ whose composition with the canonical map $V(\Gamma)/{\sim} \to V(\Gamma')/{\sim}$ induced by $f$ agrees with the given map $V(\Gamma)/{\sim} \to [n]$. Such a map is unique if it exists.}
\end{lem}

\begin{proof}
    In this proof we write $(i)$ for \Cref{D:specialization_totally_preordered_rooted_tree}(i) for $1\le i \le 3$.
    Write $g$ for the composite $V(\Gamma)\to V(\Gamma)/{\sim} \twoheadrightarrow [n]$. Suppose $f:(\Gamma,\preceq,v_0)\twoheadrightarrow (\Gamma',\preceq',v_0')$ is a contraction with the desired property. For $k\in [n]$, consider $g^{-1}(k)$. Given $v,w\in g^{-1}(k)$, any path connecting them also lies in $g^{-1}(k)$ and hence $g^{-1}(k)$ spans a disjoint union of trees, $\{T_k^i\}$. 
    
    Let $h:V(\Gamma')/{\sim'}\to [n]$ denote the isomorphism under $V(\Gamma)/{\sim}$; one has $V(T_k^i)\to h^{-1}(k)$ for each $i$. By (2), $f(T_k^i)$ is a single vertex for each $i$ and by (3) $f(T_k^i) = f(T_k^j)$ implies $i=j$. So, for each $k\in [n]$ the level $k$ vertices of $\Gamma'$ correspond bijectively to $\{T_k^i\}$. $\preceq'$ is uniquely determined by $V(\Gamma')/{\sim'}\cong [n]$; indeed, $v'\preceq'w'$ if and only if $g(v')\le g(w')$. $T_0 = f^{-1}(v_0)$ is a subtree of $\Gamma$ and $f(T_0)$ is the root of $\Gamma'$. 

    Finally, (2) implies that each $f(T_k^i)$ has a unique descending edge, and $f(T_k^i)$ and $f(T_j^{i'})$ are connected if and only if $T_k^i$ and $T_j^{i'}$ are adjacent in $\Gamma$. There can be no other edges without introducing a loop. This uniquely characterizes $(\Gamma',\preceq',v_0')$. We leave it to the reader to verify that this defines a totally preordered rooted tree with the necessary properties.\endnote{To be precise, define the level $k$ vertices of $\Gamma'$ to be $\{T_k^i\}$. Attach $T_k^i$ to $T_{k-p}^j$ with $p \ge 1$ via an edge if and only if the unique maximal vertex of $T_k^i$ is connected by an edge to a vertex in $T_{k-p}^j$. This connects each $T_k^i$ to a unique vertex on a a level $<k$. We define the marking function $h':\{1,\ldots,n\}\to V(\Gamma')$ by $h' = f\circ h$. Finally, as mentioned in the body of the proof the total preorder $\preceq'$ is uniquely determined by the condition $v'\preceq'w'$ iff $g(v')\le g(w').$}
\end{proof}

If one identifies $V(\Gamma)/{\sim} \cong [\ell]$ for some $\ell\geq 0$, then \Cref{L:specializations} says that contractions $(\Gamma,\preceq,v_0) \twoheadrightarrow (\Gamma',\preceq,v_0')$ correspond bijectively to order-preserving surjections $[\ell] \twoheadrightarrow [n]$ for $n \in [0,\ell]$. 

\begin{ex}
\label{ex:contractinglevels}
    There is a unique isomorphism $V(\Gamma)/{\sim} \cong [\ell]$ of totally ordered sets carrying $v_0$ to $0$. Given $k\in \{1,\ldots,\ell\}$, one defines the \emph{contraction of the level} $k$ as the contraction induced by the unique order preserving surjection $[\ell]\twoheadrightarrow[\ell-1]$ such that $k,k-1\mapsto k-1$. This can be regarded as a degeneracy map $\sigma_{k-1}^\ell:[\ell]\to [\ell-1]$ defined by
    \[
        \sigma_{k-1}^\ell(p) = 
        \begin{cases}
            p&p\le k-1\\
            p-1 & p \ge k.
        \end{cases}
    \]
    Since every order preserving surjection $[\ell]\twoheadrightarrow[\ell-k]$ for $0\le k \le \ell$ factorizes uniquely as $\sigma_{j_1}^{\ell-k+1}\circ \cdots \circ \sigma_{j_k}^\ell$ for $0\le j_1<\cdots<j_k\le \ell-1$ (see \cite{Maclanecategories}*{Lem. p. 173}), every contraction of $V(\Gamma)$ is determined by a choice of levels $1\le i_1<\cdots<i_k\le \ell$ that are contracted. The corresponding contraction is given by $\sigma_{i_1-1}^{\ell-k+1}\circ \cdots \circ \sigma_{i_k - 1}^\ell$. We call this the \emph{contraction of levels} $i_1<\cdots<i_k$.
\end{ex}

\begin{defn}
\label{D:stablenmarkedmultiscaledline}
An $n$-marked multiscaled line is the data of $(\Sigma,p_\infty,\preceq,\omega_\bullet,p_1,\ldots,p_n)$, where $(\Sigma,p_\infty,\preceq, \omega_\bullet)$ is a multiscaled line and the $p_i$ are smooth marked points lying on terminal components of $\Sigma$. An $n$-marked multiscaled line is called \emph{stable} if each terminal component contains at least one of the $p_i$ and each nonterminal component contains at least $3$ nodes or at least $2$ nodes and the point $p_\infty$.\endnote{Note that a stable multiscaled line has trivial complex projective automorphism group.}
\end{defn} 

Associated to an $n$-marked multiscaled line is an $n$-marked totally preordered rooted tree $(\Gamma,\preceq,v_0,h)$, where $(\Gamma,\preceq,v_0)$ is the totally preordered rooted tree of $(\Sigma,p_\infty,\preceq,\omega_\bullet)$ and $h:\{1,\ldots,n\}\to V(\Sigma)_{\rm{term}}$ is defined by $h(i) = v$ if $p_i$ lies on $\Sigma_v$. 

Associated to a totally preordered rooted tree $(\Gamma,\preceq,v_0)$ is a unique isomorphism of totally ordered sets $\lambda:V(\Gamma)/{\sim} \to [\ell]$. This is called the \emph{level function} of $(\Gamma,\preceq,v_0)$ and $\ell$ is the number of levels. Note that $\lambda^{-1}(0) = \{v_0\}$ and that $\lambda^{-1}(\ell) = V(\Sigma)_{\rm{term}}$. In particular, the dual tree of $(\Sigma,p_\infty,\preceq,\omega_\bullet,p_\bullet)$ is a dual level tree in the sense of \Cref{D:dualleveltree}. Conversely, every dual level tree arises from a $n$-marked stable multiscaled line. 

In lieu of $(\Gamma,\preceq,v_0,h)$, we write $(\Gamma,\preceq)$ to emphasize the level structure or sometimes simply $\Gamma$. Since the dual tree of an $n$-marked stable multiscaled line is a dual level tree as in \Cref{D:dualleveltree}, we may also think of $(\Gamma,\preceq)$ as a chain of partitions $\{\rho_1<\cdots<\rho_\ell\} = \rho_\bullet$, by \Cref{P:treeversuschain}, where $\ell$ is the length of $\Gamma$. In the notation of \Cref{P:treeversuschain}, we will often write $\mathfrak{c}(\Sigma)$ instead of $\mathfrak{c}(\Gamma(\Sigma))$ for brevity. 

Note that for $1\le k \le \ell$, the contraction of level $k$ as in \Cref{ex:contractinglevels} corresponds at the level of chains of partitions to replacing $\rho_1<\cdots<\rho_\ell$ with $\rho_1<\cdots<\widehat{\rho}_k<\cdots<\rho_\ell$, where the hat means ``omitted.'' In particular, all contractions of a dual tree correspond to removal of elements of its corresponding chain of partitions.

\subsection{Constructing the space}

The set $\mscbar_n$ consists of complex projective isomorphism classes of $n$-marked stable multiscaled lines $(\Sigma,p_\infty,\preceq,\omega_\bullet,p_\bullet)$. To simplify notation, we denote a point of $\Mmscbar_n$ by $\Sigma$, unless the other data are explicitly used.

\begin{defn}[Period functions]
For any $i,j \in \{1,\ldots,n\}$, we define functions $\Pi_{ij} : \mscbar_n \to \bP^1$ as follows: If $\Sigma$ is a stable $n$-marked multiscaled line, then $\Pi_{ij}(\Sigma) = \int_{p_i}^{p_j} \omega_v$ if $p_i$ and $p_j$ are both contained in the same terminal component $\Sigma_v$, and $\Pi_{ij}=\infty$ otherwise. The integral is taken over any path from $p_i$ to $p_j$ in the smooth locus of $\Sigma_v$.
\end{defn}

Let $\Mmscbar_n^\circ \subset \Mmscbar_n$ denote the subset of isomorphism classes of multiscaled lines with irreducible underlying curve. $\Mmscbar_n^{\circ}$ can be regarded as the configuration space of $n$ points in $\bA^1$ up to simultaneous translation, $\bA^n/\bG_a$. Indeed, $(\Sigma,p_\infty,\omega,p_\bullet) \in \Mmscbar_n^\circ$ is isomorphic to $(\bP^1,\infty,\omega,p_\bullet)$, where $p_i\ne \infty$ for all $1\le i \le n$ and $\omega$ has a order $2$ pole at $\infty$ and no other zeros or poles. $\omega$ determines a coordinate $z:\bP^1\setminus \{\infty\}\to \bA^1$ up to translation such that $dz = \omega$. Consequently, $(\Sigma,p_\infty,\omega,p_\bullet)$ is equivalent to $(\bP^1,\omega,dz)$ together with $(z(p_1),\ldots, z(p_n))\in \bA^n/\bG_a$. Subject to this identification, $\Pi_{ij}$ restricts to the coordinate on $\bA^n/\bG_a$, given by $\Pi_{ij}(z) = z_j - z_i$. 

Given a stable $n$-marked multiscaled line $\Sigma$, for each irreducible component $\Sigma_v$ of $\Sigma$ choose $n_v \in \Sigma_v$, where $n_v$ denotes either an ascending node or a marked point when $\Sigma_v$ is terminal. Write $N_v$ for the set of ascending nodes on $\Sigma_v$ when $\Sigma_v$ is nonterminal or marked points when it is terminal.

\begin{lem}
\label{L:determinedbyintegrals}
    Suppose $\Sigma$ is a stable $n$-marked multiscaled line with dual level tree $\Gamma(\Sigma)$. $\Sigma$ is determined up to isomorphism by $\Gamma(\Sigma)$ and 
    \begin{equation}
    \label{E:integrals}
        \left\{\int_{n_v}^{n} \omega_v\:\bigg|\: n\in N_v \setminus \{n_v\}\right\}_{v\in \Gamma(\Sigma)},
    \end{equation}
    where $\int_{n_v}^n\omega_v$ is the integral along any path in $\Sigma_v$ connecting $n_v$ and $n$. It is uniquely determined up to complex projective isomorphism by $\Gamma(\Sigma)$ and
    \begin{equation}
    \label{E:integrals2}
    \left\{ \left[ 
        \int_{n_v}^{n} \omega_v\:\bigg|\: 
        \begin{array}{c} 
            v\in \lambda^{-1}(m)  \\ 
            n \in N_v\setminus \{n_v\}
        \end{array} 
    \right]_{0\le m \le \ell-1},\left(\int_{n_v}^n\omega_v\:\bigg|\: 
    \begin{array}{c}
    v\in \lambda^{-1}(\ell) \\ 
    n\in N_v\setminus \{n_v\}
    \end{array}
    \right)
    \right\}
    \end{equation}
    where latter data are regarded as elements of $\prod_{m=0}^{\ell-1}\bP^{\nu_m} \times \bA^{\nu_\ell}$ where $\nu_m = \sum_{v\in \lambda^{-1}(m)} |N_v\setminus \{n_v\}|$ for all $0\le m \le \ell$.
\end{lem}

\begin{proof}
    First, suppose $\Sigma$ is irreducible. Then $(\Sigma,\omega,p_\bullet) \in \Mmscbar_n^\circ$ is uniquely determined by $\{\int_{p_1}^{p_j}\omega\}_{j=2}^n$ since this is equivalent to the statement that $\{\Pi_{1j}\}_{j=2}^n$ form coordinates on $\bA^n/\bG_a$. 
    
    For general $\Sigma$, an isomorphism $\alpha:(\Sigma,p_\infty,\preceq,\omega_\bullet,p_\bullet)\to (\Sigma',p_\infty',\preceq',\omega_\bullet',p_\bullet')$ of stable $n$-marked multiscaled lines with dual tree $\Gamma$ is equivalent to the data of $\{\alpha_v:v\in V(\Gamma)\}$, where writing $n_-$ for the descending node of $\Sigma_v$ and $n_-'$ for that of $\Sigma_v'$ (resp. $p_\infty$ and $p_\infty'$ for $v = v_0$), $\alpha_v$ is an isomorphism of irreducible scaled lines $(\Sigma_v,n_-,\omega_v,n\in N_v) \to (\Sigma_v',n_-',\omega_v',n'\in N_v')$. Thus, the isomorphism class of $(\Sigma,p_\infty,\preceq,\omega_\bullet,p_\infty)$ is determined by \eqref{E:integrals}.

    When $\alpha$ is a complex projective isomorphism, we have again $\{\alpha_v:v\in V(\Gamma)\}$ except that $\alpha_v^*(\omega_v') = c_v\omega_v$, where $c_v\in \bC^*$ is a constant depending only on the level of $v$, and equal to $1$ on level $\ell$. The second claim follows.
\end{proof}

For any stable $n$-marked multiscaled line $\Sigma$, let 
\[
    I_{ij}:= \int_{n_i}^{n_j}\omega_{v}
\]
where $v:= h(i)\wedge h(j) \in V(\Sigma)$ and $n_i,n_j\in \Sigma_v$ denote either 
\begin{enumerate}
    \item[(1)] the marked points $p_i$ and $p_j$ if $h(i)=h(j)$; or
    \item[(2)] the nodes in $\Sigma_{v}$ that connect to $\Sigma_{h(i)}$ and $\Sigma_{h(j)}$, respectively. 
\end{enumerate} 
The integral is taken along any path connecting $n_i$ and $n_j$ in the smooth locus of $\Sigma_v$. For any given dual level tree $\Gamma$, we introduce the following subset of $\Mmscbar_n$,
\[
    U_\Gamma := \left\{ \Sigma \in \mscbar_n \left| \begin{array}{c} \exists \text{ contraction } \Gamma \twoheadrightarrow \Gamma(\Sigma), \text{ and} \\\Pi_{ij}(\Sigma) \neq 0 \text{ if } h(i)\neq h(j) \text{ in } \Gamma \end{array} \right. \right\}.
\]
Now consider a stable $n$-marked multiscaled line $\Sigma$ whose isomorphism class lies in $U_\Gamma$, and let $f : \Gamma \twoheadrightarrow \Gamma(\Sigma)$ be the (unique!) contraction. Suppose $\Gamma/{\sim} = [\ell]$ and for each $0\le m \le \ell-1$ choose a vertex $v_m$ on level $m$. Also, choose $i_m,j_m \in \{1,\ldots, n\}$ such that $h(i_m)\wedge h(j_m) = v_m$. Put $s_\ell = 1$ and for each $0\le m \le \ell-1$, put $s_m = I_{i_mj_m}$. Next, for $1\le m \le \ell$, let 
\[
        t_m := \left\{ \begin{array}{ll} s_{m}/s_{m-1}, & \text{if } f(v_m) \sim f(v_{m-1}) \\
        0& \text{otherwise}\end{array} \right.
\] 
Finally, for any $i,j\in \{1,\ldots, n\}$ such that $h(i)\wedge h(j)$ is on level $m$, put 
\[
    z_{ij} := \frac{I_{ij}}{s_m} = \frac{1}{s_m} \int_{n_i}^{n_j}\omega_{f(v_m)}
\] 
for all $i,j \in \{1,\ldots, n\}$.

\begin{lem}
\label{L:equivalentindices}
    Let $A$ denote the set of pairs $i<j$ such that $h(i) \neq h(j)$ in $V(\Gamma)$, and $B$ the set of pairs with $h(i) = h(j)$. The resulting function
    \begin{equation}
    \label{E:Mn_coordinates}
        (z_{ij},t_m)_{\substack{1\leq i<j \leq n\\ 1 \leq m \leq \ell}} : U_\Gamma \to (\bC^\ast)^A \times \bC^B \times \bC^\ell,
    \end{equation}
    is well-defined. If different vertices $v_m$ and indices $i_m,j_m$ are chosen for each $m =0,\ldots,\ell-1$, the resulting function differs from the original by composition with a monomial automorphism of $(\bC^\ast)^A \times \bC^B \times \bC^\ell$.
\end{lem}

\begin{proof}
    It suffices to prove the result changing one pair of indices $(i_m,j_m) = (i,j)$ to $(i_{m}',j_m') = (k,\ell)$ on some level $0\le m \le \ell-1$. Write $z_{ij}'$ for the coordinates defined with respect to the indices $(i_m',j_m')$. One computes that $z_{\alpha\beta} = z_{\alpha\beta}'$ for all $\alpha,\beta$ with $h(\alpha)\wedge h(\beta)$ not on level $m$. Otherwise, we have $z_{k\ell}' = z_{ij}$ and $z_{\alpha \beta}' = z_{\alpha\beta}/z_{k\ell}$. For $k\not\in \{m,m+1\}$, one has $t_k' = t_k$. Finally, $t_m' = t_m\cdot z_{k\ell}$ and $t_{m+1}' = t_{m+1}/z_{k\ell}$. It is an exercise to verify that this map is invertible.\endnote{Since $z_{\alpha\beta} = I_{\alpha\beta}/I_{ij}$ and $z_{\alpha\beta}' = I_{\alpha\beta}/I_{k\ell}$, the inverse map is given by $z_{\alpha\beta} = z_{\alpha\beta}'\cdot (z_{ij}')^{-1}$. Furthermore, one can verify that $t_k = t_k'$ for $k\ne m,m+1$, $t_{m+1} = t_{m+1}'/z_{ij}'$, and $t_m = t_m'\cdot z_{ij}'.$}
\end{proof}

The functions $(z_{ij},t_m)$ depend on a choice of level tree $\Gamma$ which defines the open set $U_\Gamma$ and indices $\{(i_m,j_m)\}_{m=0}^{\ell-1}$. We will always suppress the choice of indices from the notation, however to emphasize $\Gamma$ we may write $(z_{ij}^\Gamma,t_m^\Gamma)$. By a \emph{choice of indices for} $\Gamma$, we mean a choice of $\{(i_m,j_m)\}_{m=0}^{\ell-1}$ as above.\endnote{By \Cref{L:equivalentindices}, the choice of indices does not affect the map $U_\Gamma \to (\bC^*)^A\times \bC^B \times \bC^\ell$ up to a simple type of isomorphism. However, these indices are analogous to choosing a basis of a vector space and different choices are technically useful in different scenarios.} The default notation will be to write an alternative choice of indices as $\{(i_m',j_m')\}_{m=0}^{\ell-1}$ and the resulting coordinates as $(z_{ij}',t_m')$.

\begin{rem}
\label{R:periodsintermsofzij}
    We write $*$ for the tree with one vertex and note that $U_* = \Mmscbar_n^\circ$. There are no $t$ functions for $\Gamma = *$ and under the identification $U_* = \Mmscbar_n^\circ = \bA^n/\bG_a$, the functions $z_{ij}^*$ correspond to the period functions $\Pi_{ij}(z) = z_j - z_i$.
\end{rem}

\begin{cor}
\label{C:determinedbyz}
    Let $\Sigma$ be given whose equivalence class lies in $U_\Gamma$. $\{z_{ij}(\Sigma),t_m(\Sigma)\}$ depend only on the complex projective isomorphism class of $\Sigma$. Furthermore, if $\Gamma(\Sigma) = \Gamma$, then $\Sigma$ is uniquely determined up to complex projective isomorphism by $\{z_{ij}(\Sigma)\}_{1\le i<j\le n}$.
\end{cor}

\begin{proof}
    Since each $z_{ij}$ is defined as a ratio $I_{ij}/I_{i_mj_m}$ where $h(i) \wedge h(j)$ is on level $m$, it depends only on the complex projective isomorphism class of $\Sigma$.\endnote{I.e., this quantity is invariant under levelwise rescaling of the differentials on nonterminal levels.} For each $0\le m \le \ell-1$, choose $n_v\in N_v$ for each $v\in \lambda^{-1}(m)$ such that $n_v$ is the node connecting $\Sigma_v$ to $\Sigma_{h(i_m)}$ when $v = v_m$. $s_m = \int_{n_v}^{n}\omega_{v_m}$ for some $n\in N_{v_m}$, by definition. Choose $0\le m \le \ell-1$, $v\in \lambda^{-1}(m)$, and $n'\in N_v\setminus \{n_v\}$. Then $z_{ij}(\Sigma) = \int_{n_v}^{n'} \omega_v/\int_{n_{v_m}}^{n}\omega_{v_m}$, where $h(i)\wedge h(j) = v$, $n_v$ connects $\Sigma_v$ to $\Sigma_{h(i)}$, and $n'$ connects $\Sigma_v$ to $\Sigma_{h(j)}$. So, all ratios of homogeneous coordinates in \eqref{E:integrals2} can be recovered from $\{z_{ij}(\Sigma)\}$. Similar reasoning shows that each $\int_{n_v}^n \omega_v$ for $v\in \lambda^{-1}(\ell)$ equals some $z_{ij}(\Sigma)$. Therefore, by \Cref{L:determinedbyintegrals}\eqref{E:integrals2} we are done.
\end{proof}

\begin{lem}
\label{L:intersection}
    Suppose $f:\Gamma\twoheadrightarrow \Gamma'$ is the contraction of levels $1\le i_1<\cdots< i_k\le \ell$. Then $U_\Gamma \cap U_{\Gamma'} = D(t_{i_1}^\Gamma\cdots t_{i_k}^\Gamma)$.
\end{lem}

\begin{proof}
    One verifies that $U_\Gamma \cap U_{\Gamma'} = \{\Sigma \in U_\Gamma: \exists \text { contraction } \Gamma'\twoheadrightarrow\Gamma(\Sigma)\}$.\endnote{By definition, $U_\Gamma \cap U_{\Gamma'}$ consists of those $\Sigma$ such that there exist contractions $\Gamma \twoheadrightarrow \Gamma(\Sigma)$ and $\Gamma'\twoheadrightarrow \Gamma(\Sigma)$ and $\Pi_{ij}(\Sigma)\ne 0$ if $h(i) \ne h(j)$ or $h'(i) \ne h'(j)$. However, $h'(i)\ne h'(j)$ implies $h(i) \ne h(j)$ since $\Gamma'$ is a contraction of $\Gamma$. Therefore, if $\Sigma \in U_\Gamma$, the condition of being in $U_{\Gamma'}$ is just that there exists a contraction $\Gamma'\twoheadrightarrow\Gamma(\Sigma)$.} This is equivalent to saying that $g:\Gamma\twoheadrightarrow\Gamma(\Sigma)$ factors through $f$. This, in turn, is equivalent to the statement: if $f$ contracts a level $m$ then $g$ contracts the level $m$. This happens exactly on the locus where $t_{i_1}^\Gamma\cdots t_{i_k}^\Gamma$ does not vanish.
\end{proof}

\begin{lem}
\label{L:goodindicesCOV}
Suppose $\Sigma$ lies in $U_\Gamma \cap U_{\Gamma'}$ where $\Gamma \twoheadrightarrow \Gamma'$ is the contraction corresponding to $\alpha:[\ell]\twoheadrightarrow [\ell-k]$, deleting levels $1\le i_1< \cdots < i_k\le \ell$. Given a choice of indices for $\Gamma$, there is a choice of indices for $\Gamma'$ such that 
    \begin{enumerate}
        \item if $h(i)\wedge h(j)$ is on level $m$ of $\Gamma$ and 
        \begin{enumerate} 
            \item $\{i_1,\ldots, i_k\}\cap \bZ_{\ge m+1}$ contains a maximal consecutive sequence $\{m+1,\ldots, m+p\}$ for $p\ge 1$ then $z_{ij}' = z_{ij}/t_{m+1}\cdots t_{m+p}$; 
            \item if $m+1\not \in \{i_1,\ldots, i_k\}$ then $z_{ij}' = z_{ij}$; and 
        \end{enumerate}
        \item for all nonzero $j\in [\ell-k],$ one has $t_j' = \prod_{i\in \alpha^{-1}(j)} t_{i}$.
    \end{enumerate}
\end{lem}

\begin{proof}
    A choice of indices for $\Gamma$ is equivalent to a choice of scale parameter $s_m$ for each $0\le m \le \ell-1$. For $m\in [\ell-k]$, put $s_m' := s_{\max \alpha^{-1}(m)}$. This is sensible since if $h(i)\wedge h(j)$ lies on a level in $\alpha^{-1}(m)$, then $h'(i)\wedge h'(j)$ lies on level $m$. 

    Suppose a single level $m+1$ is contracted and that $h(i) \wedge h(j)$ is on level $k$. If $k\le m-1$, $z_{ij} = I_{ij}/s_k = z_{ij}'$. If $k = m$, then $h'(i)\wedge h'(j)$ is also on level $m$ so $z_{ij} = I_{ij}/s_{m}$ and $z_{ij}'= I_{ij}/s_m'$. Thus, $z_{ij}/z_{ij}' = s_{m+1}/s_m = t_{m+1}$ and $z_{ij}' = z_{ij}/t_{m+1}$. Finally, if $k \ge m+1$, then $h'(i) \wedge h'(j)$ is on level $k-1$. So, $z_{ij} = I_{ij}/s_k = I_{ij}/s_{k-1}' = z_{ij}'$. 

    Next, $t_m' = s_m'/s_{m-1}' = s_{m+1}/s_{m-1} = t_m\cdot t_{m+1}$. If $1\le k <m$, then $t_k' = s_k'/s_{k-1}' = s_k/s_{k-1} = t_k$. For $m<k\le \ell-1$, one has $t_k' = s_k'/s_{k-1}' = s_{k+1}/s_k =t_{k+1}$. This verifies (2).

    Consider the contraction of levels $i_1<\cdots<i_k$, $\Gamma\twoheadrightarrow \Gamma''$, recalling from \Cref{ex:contractinglevels} that all contractions are of this form. It factorizes as $\Gamma\twoheadrightarrow\Gamma'\twoheadrightarrow\Gamma''$ where $\Gamma\twoheadrightarrow\Gamma'$ is the contraction of levels $i_2<\cdots<i_{k}$ and corresponds to $\beta:[\ell]\twoheadrightarrow [\ell-k+1]$ and $\Gamma'\twoheadrightarrow\Gamma''$ contracts $\beta(i_1) = i_1$ and corresponds to $\gamma:[\ell-k+1]\twoheadrightarrow[\ell-k]$.\endnote{In the notation of \Cref{ex:contractinglevels}, $\beta$ can be written as $\sigma_{i_2-1}\circ \cdots \circ \sigma_{i_k-1}$ with some indices suppressed and as $i_1\le i_2-1<\cdots<i_k-1$ we have that $\beta(i_1) = i_1$.} By induction, there are choices of indices for $\Gamma''$ and $\Gamma'$ so that $t_j'' = \prod_{i\in \gamma^{-1}(j)}t_i'$ and $t_i' = \prod_{l\in \beta^{-1}(i)}t_l$, whence $t_j'' = \prod_{i\in \gamma^{-1}(j)}\prod_{l\in \beta^{-1}(i)} t_l = \prod_{l\in \alpha^{-1}(j)} t_l$. So, (2) follows. 
    
    Suppose $h(i)\wedge h(j)$ is on level $m$ and that $\{m+1,\ldots, m+p\}$ is a maximal consecutive sequence of integers in $\{i_1,\ldots, i_k\} \cap \bZ_{\ge m+1}$ containing $m+1$. If $m\ge i_1$ then $\{m+1,\ldots,m+p\}\subset \{i_2,\ldots, i_k\}$. So, $z_{ij}' = z_{ij}/t_{m+1}\cdots t_{m+p}$ by induction. Now, $\beta(m) \ge i_1$ so $\beta(m) + 1 > i_1$ and thus $z_{ij}'' = z_{ij}' = z_{ij}/t_{m+1}\cdots t_{m+p}$. 

    If $m <i_1-1$ then $z_{ij}''=z_{ij}' = z_{ij}$. Finally, if $m = i_1-1$, one has $z_{ij}' = z_{ij}$ by induction. Since $\beta(m) = i_1-1$, we have $\beta(m) + 1\in \{i_1\}$ and $z_{ij}'' = z_{ij}'/t_{i_1}' = z_{ij}/\prod_{l\in \gamma^{-1}(i_1)} t_l = z_{ij}/t_{m+1}\cdots t_{m+p}$.
\end{proof}

\begin{cor}
\label{C:goodCOVgeneric}
    For any $\Gamma$, on $U_* \cap U_\Gamma$ one has $\Pi_{ij} = z_{ij}^\Gamma/t_{m+1}\cdots t_\ell$ where $h(i)\wedge h(j)$ is on level $m$ of $\Gamma$.
\end{cor}

\begin{proof}
    This is an immediate consequence of \Cref{L:goodindicesCOV}.
\end{proof}

\begin{prop}
\label{P:identification}
    For any dual level tree $\Gamma$ and choice of indices, \eqref{E:Mn_coordinates} is injective and identifies $U_\Gamma$ with a smooth algebraic variety of dimension $n-1$.
\end{prop}

\begin{proof}
    For each $v\in V(\Gamma)$, consider $N_v$ as in the discussion preceding \Cref{L:determinedbyintegrals}. When $v = v_m$ for $0\le m\le \ell-1$, we choose $n_v$ to be the node connecting $\Sigma_{v_m}$ to $\Sigma_{h(i_m)}$. For each nonterminal $v$ and each $n\in N_v$, choose $\iota_n \in \{1,\ldots, n\}$ such that $n$ connects $\Sigma_v$ to $\Sigma_{h(\iota_n)}$. We write $\iota_v = \iota_{n_v}$. When $v = v_m$, choose $\{\iota_n:n \in N_{v_m}\}$ so it contains $i_m,j_m$. When $v$ is terminal, $N_v = \{p_i:h(i) = v\}$ and we define $\iota_i = i$. Consider 
    \[
    A' := \left(\bigcup_{v \ne \lambda^{-1}(\ell)}\{(\iota_v,\iota_n):n\in N_v\setminus n_v\}\right) \setminus \{(i_m,j_m)\}_{m=0}^{\ell-1},\:\: B' := \bigcup_{v\in \lambda^{-1}(\ell)}\{(\iota_v,\iota_n):n\in N_v\setminus n_v\}. 
    \]
    One has $A'\subset A$ and $B'\subset B$ and $|A'\cup B'| = n - 1 - \ell$. Consider $\bC^A := \Spec \bC[Z_{ij}:(i,j)\in A]$, $\bC^B = \Spec \bC[Z_{ij}:(i,j) \in B]$ and $\bC^\ell = \Spec \bC[T_1,\ldots, T_\ell]$. If we define $\bC^{A'}$ and $\bC^{B'}$ in the analogous fashion, then there is a natural inclusion map $\bC^{A'}\times \bC^{B'}\times \bC^\ell \hookrightarrow \bC^A\times \bC^B\times \bC^\ell$. Define $\varphi: U_\Gamma \to \bC^{A'}\times \bC^{B'}\times \bC^\ell$ by $\varphi^*(Z_{ij}) = z_{ij}$ and $\varphi^*(T_m) = t_m$ for all relevant indices. One can verify that $V := \im(\varphi)$ is the complement of a hyperplane arrangement defined by the conditions $Z_{\iota_v\iota_n} \ne 0$ and $Z_{\iota_v\iota_n} \ne Z_{\iota_v\iota_{n'}}$ for all $v$ and $n,n'\in N_v\setminus \{n_v\}$.\endnote{Elements of $\im(\varphi)$ satisfy these nonvanishing conditions, by the description of the elements of $U_\Gamma$. Indeed, suppose $\Gamma\twoheadrightarrow\Gamma'$ is a contraction and $\Sigma$ is given with $\Gamma(\Sigma) = \Gamma'$. If $z_{\iota_v\iota_n}(\Sigma) = 0$, it must be the case that $h'(\iota_v)= h'(\iota_n)$. Now, if $h(\iota_v) = h(\iota_n)$, then this does not correspond to one of the hyperplanes which has been removed. If $h(\iota_v) \ne h(\iota_n)$, the condition $z_{\iota_v\iota_n}(\Sigma) = 0$ is equivalent to $\Pi_{\iota_v\iota_n}(\Sigma) = 0$ which is prohibited by definition of $U_\Gamma$. Analogous reasoning applies for the conditions $z_{\iota_v\iota_n} = z_{\iota_v\iota_{n'}}$. On the other hand, given a set of values $\{z_{ij}(\Sigma),t_m(\Sigma)\} \in V$ where $(i,j)\in A'\cup B'$ one can freely construct $\Sigma \in U_\Gamma$ with these coordinates.} In particular, $V$ is a quasi-affine variety.

    Finally, we define a morphism $g:\bC^{A'}\times \bC^{B'}\times \bC^\ell \to \bC^{A\setminus A'}\times \bC^{B\setminus B'}$: For all $(i,j)\in (A\setminus A')\cup (B\setminus B')$ such that $h(i)\wedge h(j) = v$, put $g^*(Z_{ij}) = Z_{\iota_v\iota_{n'}} - Z_{\iota_v\iota_n}$ where $\Sigma_v$ is connected to $\Sigma_{h(j)}$ by $n'$ and to $\Sigma_{h(i)}$ by $n$. In particular, $g$ is a smooth morphism which restricts to a map $g:V\to (\bC^*)^{A\setminus A'}\times \bC^{B\setminus B'}$. It follows that $(\id,g)\circ \varphi:U_\Gamma \to (\bC^*)^A\times \bC^B\times \bC^\ell$ equals \eqref{E:Mn_coordinates} and in particular this identifies the image of $U_\Gamma$ with the graph of a smooth morphism of quasi-affine varieties. Projection to $V$ implies that the image is dimension $n-1$.\endnote{Consider a morphism $f:V\to W$ of quasi-affine varieties. $\Gamma_f\subset V\times W$ is defined by $\{(x,f(x)):x\in V\}$. $\Gamma_f$ is a quasi-affine variety itself; writing $V\subset \bA^n$ and $W\subset \bA^m$ and $f$ as $f(x) = (f_1(x),\ldots, f_m(x))$ one has that $\Gamma_f$ is defined by the vanishing of the equations $\{y_i-f_i(x)\}_{i=1}^m$.
    
    Next, there is a canonical map $(\id,f):V\to \Gamma_f$ by $x\mapsto (x,f(x))$ which is surjective. In our situation, we need to verify that $(\id,g)\circ \varphi$ is the same as \eqref{E:Mn_coordinates}. If $(i,j) \in A'\cup B'$, one has $(\id,g)^*(Z_{ij}) = Z_{ij}$, $\varphi^*(Z_{ij}) = z_{ij}$. If $(i,j)\in (A\setminus A')\cup (B\setminus B')$, then $(\id,g)^*(Z_{ij}) = Z_{\iota_v\iota_{n'}} - Z_{\iota_v\iota_n}$, where $n'$ and $n$ are as in the body of the proof, and so $\varphi^*((\id,g)^*(Z_{ij})) = z_{\iota_vj} - z_{\iota_vi} = z_{ij}$ and we are done.}
\end{proof}

\begin{rem}
\label{R:projectioncoords}
    In the proof of \Cref{P:identification}, we have shown moreover that each $U_\Gamma$ is algebraically isomorphic to the complement of a hyperplane arrangement in $\bA^{n-1}$: indeed, projection from the image of $U_\Gamma$ to $V$ induces the isomorphism and $\{z_{ij},t_m:(i,j)\in A'\cup B',1\le m\le \ell\}$ form a system of algebraic coordinates on $U_\Gamma$. This implies that $\Mmscbar_n$ is uniformly rational in the sense of \cite{BogomolovBöhning}.\endnote{Note that these coordinates are far from unique and we made many choices in their definition. In this sense, it is more canonical to work with all $(z_{ij},t_m)$ at once as in \eqref{E:Mn_coordinates}.}
\end{rem}

\begin{thm}
\label{T:spaceconstruction}
The set $\mscbar_n$ admits the unique structure of an algebraic variety over $\bC$ such that every $U_\Gamma$ is a Zariski open subset, and the functions \eqref{E:Mn_coordinates} are closed immersions. With respect to this structure:
\begin{enumerate}
    \item $\bA^n/\bG_a$ is an open dense subspace;
    \item the functions $\Pi_{ij}$ on $\bA^n/\bG_a$ extend to morphisms $\Mmscbar_n \to \bP^1$; and
    \item $\Mmscbar_n$ is smooth and proper of dimension $n-1$.
\end{enumerate}
\end{thm}

\begin{proof}
    To emphasize dependence on $\Gamma$, we denote \eqref{E:Mn_coordinates} by $\psi_\Gamma$. When $\Gamma'$ is a coarsening of $\Gamma$, it follows from \Cref{L:goodindicesCOV} that $\psi_{\Gamma'}\circ \psi_{\Gamma}^{-1}$ is algebraic. For a general $\Gamma'$, $U_\Gamma \cap U_{\Gamma'}$ is covered by $U_\Gamma \cap U_{\Gamma'} \cap U_{\Gamma''}$ where $\Gamma''$ ranges over level trees that are a coarsening of both $\Gamma'$ and $\Gamma$. However, on $U_{\Gamma}\cap U_{\Gamma'}\cap U_{\Gamma''}$ one has $\psi_{\Gamma'}\circ \psi_{\Gamma}^{-1} = (\psi_{\Gamma'}\circ \psi_{\Gamma''}^{-1}) \circ (\psi_{\Gamma''}\circ \psi_{\Gamma}^{-1})$ and as the composition of algebraic maps is algebraic, the result follows. (3) is by \Cref{P:identification}. $\bA^n/\bG_a$ is identified with $U_*$ as above and so openness is clear. For any $\Gamma$, $U_*\cap U_\Gamma = D(t_1\cdots t_\ell)$ by \Cref{L:intersection} and density follows. 
    
    For the remaining claims, we consider $\Mmscbar_n$ with its analytic topology. Suppose given a convergent net $(\Sigma_\alpha)\to \Sigma$ with $\Sigma$ reducible. On $U_{\Gamma(\Sigma)}\cap U_*$, we have $\Pi_{ij} = z_{ij}/t_{m+1}\cdots t_\ell$ as in \Cref{C:goodCOVgeneric}, where $h(i)\wedge h(j)$ is on level $m$. Therefore, writing $t_k^\alpha = t_k(\Sigma_\alpha)$ for all $1\le k \le \ell$, one has $\lim_\alpha t_k^\alpha =0$. Similarly, for all $i,j$ where $h(i)\ne h(j)$, $\lim_\alpha z_{ij}^\alpha\in \bC^*$. Consequently, $\lim_\alpha \Pi_{ij}(\Sigma_\alpha) = \infty$ if and only if $h(i)\ne h(j)$ and if $h(i) = h(j)$, then $\lim_\alpha \Pi_{ij}(\Sigma_\alpha) = \Pi_{ij}(\Sigma)$. Therefore, $\Pi_{ij}:\Mmscbar_n\to \bP^1$ is continuous.

    A topological space is Hausdorff if and only if there is a dense subspace $Y\subset X$ such that any net $(y_\alpha)$ in $Y$ has at most one limit in $X$.\endnote{Suppose given is a dense subspace $Y\subset X$ such that the claimed property holds. Suppose $(x_\alpha)_{\alpha \in A}$ is a net in $X$ with limit points $x$ and $x'$. Choose a net $(y_{\alpha}^\beta)_{\beta \in B_\alpha}$ such that $\lim_\beta y_\alpha^\beta = x_\alpha$ for each $\alpha \in A$. By the theorem on iterated limits of nets \cite{KelleyTopology}*{p.69} it follows that the net $(y_\alpha^\beta)$ indexed by the product directed set $A\times \prod_{\alpha\in A} B_\alpha$ converges to $x$ and $x'$. Therefore, $x=x'$. It follows that $X$ is Hausdorff. Conversely, take $Y=X$.} In our case, let $Y\subset \bA^n/\bG_a$ denote the open subspace where $\Pi_{ij} \ne 0$ for all $i<j$ and consider a net $(\Sigma_\alpha)$ in $Y$.

    Suppose given a net $(\Sigma_\alpha)$ in $Y$ such that for all pairs $i<j$ and $k<l$ the nets $\Pi_{ij}^\alpha := \Pi_{ij}(\Sigma_\alpha)$ and $\Pi_{kl}^\alpha/\Pi_{ij}^\alpha$ converge in $\bP^1$. Write $\lim_\alpha \Pi_{ij}^\alpha = \Pi_{ij}$. We define an $n$-marked rooted level tree from these data as follows: define the level $\ell$ vertex set by $\{1,\ldots,n\}/{\sim}$, where $i\sim j$ if $\lim_\alpha\Pi_{ij}^\alpha \in \bC$. Define the markings by $h(i) = [i]$ for all $1\le i \le n$. We define a total preorder on $\{(i,j):h(i)\ne h(j)\}$ by $(i,j)\le (k,l)$ if $\lim_\alpha \Pi_{ij}^\alpha/\Pi_{kl}^\alpha \in \bC$.\endnote{Totality follows from the assumption that $\lim_\alpha \Pi_{ij}^\alpha/\Pi_{kl}^\alpha$ converges in $\bP^1$ for all $i<j$ and $k<l$. Reflexivity is by $\lim_\alpha \Pi_{ij}^\alpha/\Pi_{ij}^\alpha = 1$ and transitivity follows from the identity $\frac{\Pi_{ij}^\alpha}{\Pi_{kl}^\alpha}\cdot \frac{\Pi_{kl}^\alpha}{\Pi_{mn}^\alpha } = \frac{\Pi_{ij}^\alpha}{\Pi_{mn}^\alpha}$.
    } Note that this preorder depends only on the classes of $i,j,k,l$ under $\sim$.\endnote{Suppose $i\sim i'$. Then $\Pi_{i'j}^\alpha/\Pi_{kl}^\alpha = (\Pi^\alpha_{i'i} + \Pi^\alpha_{ij})/\Pi_{kl}^\alpha$. Now, as $h(k)\ne h(l)$, $\Pi_{kl}^\alpha \to \infty$ and $\Pi_{i'i}^\alpha$ converges in $\bC$. Therefore, $\Pi_{i'i}^\alpha/\Pi_{kl}^\alpha \to 0$ and we have $\lim_\alpha \Pi_{i'j}^\alpha/\Pi_{kl}^\alpha = \lim_\alpha\Pi_{ij}^\alpha/\Pi_{kl}^\alpha$. A similar argument applies for the indices in the denominator.}

    Suppose a level forest $\Gamma_{\ge k+1}$ has been constructed with levels $k+1\le \cdots\le \ell$. Consider $\cal{F}_{k+1} = \{([i],[j]):\nexists \: [i]\wedge [j] \in \Gamma_{\ge k+1}\}$. Define a level $k$ vertex denoted $[i]\wedge[j]$ iff $([i],[j])$ are minimal with respect to the preorder $\le$ restricted to $\cF_{k+1}$. Define edges between $[i]\wedge[j]$ and the vertices of minimal level in $\Gamma_{\ge k+1}$ to which $[i]$ and $[j]$ are connected, respectively. It follows that the output $\Gamma_{\ge k}$ is a level forest with levels $k\le \cdots\le \ell$.\endnote{$\Gamma_{\ge k+1}$ is a level forest and we have connected each new $[i]\wedge [j]$ to distinct components. Therefore, each pair of vertices is connected by a unique path in $\Gamma_{\ge k}$ as needed.}
    
    Since a preorder on a finite set always has minimal elements, $|\cF_{k}|<|\cF_{k+1}|$ and this process eventually terminates. At the last stage, one vertex is attached which we define to be the root. Shift the level indexing so that the root is on level $0$. The output $\Gamma$ is an $n$-marked rooted level tree: by induction $\Gamma$ is an $n$-marked level forest and by construction for all terminal vertices $[i]$ and $[j]$, $[i]\wedge [j]$ exists. 

    Suppose $(\Sigma_\alpha)\to \Sigma$. We claim that $\Gamma = \Gamma(\Sigma)$. $p_i$ and $p_j$ are on the same terminal component of $\Sigma$ iff $\Pi_{ij}(\Sigma) \ne \infty$, so level $\ell$ of the trees agrees. Suppose the trees agree up to level $k+1\le \ell$. $i$ and $j$ meet on level $k$ of $\Gamma(\Sigma)$ if and only if $(i,j)$ is minimal among pairs that do not meet on level $\ge k+1$ with respect $\le$ as above.\endnote{By \Cref{L:goodindicesCOV}, if $i$ and $j$ meet on level $k$ of $\Gamma(\Sigma)$ then in $U_\Gamma$ coordinates we have $\Pi_{ij}^\alpha = z_{ij}^\alpha/t_{k+1}^\alpha \cdots t_\ell^\alpha$. So, if $k$ and $\ell$ meet on level $p\le k$ one has $\Pi_{ij}^\alpha/\Pi_{k\ell}^\alpha = \frac{z_{ij}^\alpha}{z_{k\ell}^\alpha}\cdot \frac{t_{p+1}^\alpha\cdots t_\ell^\alpha}{t_{k+1}^\alpha\cdots t_\ell^\alpha}.$ This converges to $0$ if and only if $p>k$ and to an element of $\bC^*$ otherwise.} Therefore, $\Gamma = \Gamma(\Sigma)$. 

    Now, $\Sigma_\alpha \in U_\Gamma \cap U_*$ if and only if $\Pi_{ij}(\Sigma_\alpha) \ne 0$ for all pairs $i,j$ such that $h(i)\ne h(j)$. However, $h(i) = h(j)$ if and only if $\lim_\alpha \Pi_{ij}^\alpha \in \bC$. So, if $h(i)\ne h(j)$ then $\lim_{\alpha} \Pi_{ij}^\alpha =\infty$ and so there exists $\alpha_0$ such that $\alpha \ge\alpha_0$ implies $\Sigma_\alpha \in U_\Gamma \cap U_*$. 

    Henceforth suppose $\alpha \ge \alpha_0$. Since $\Sigma_\alpha$ is convergent in $U_\Gamma$, the associated coordinate nets $\{z_{ij}^\alpha,t_m^\alpha\}$ converge and determine $\Sigma$ uniquely by \Cref{C:determinedbyz}. This implies that $\Mmscbar_n$ is Hausdorff with its analytic topology. In particular, $\Mmscbar_n$ has the structure of a complex manifold. It follows that compactness is equivalent to every sequence in $Y\subset \Mmscbar_n$ admitting a convergent subsequence in $\Mmscbar_n$.\endnote{Suppose $Y\subset X$ is given, where $X$ is a manifold of positive dimension and $Y$ is dense. Suppose every sequence in $Y$ admits a subsequence convergent in $X$. $X$ is metrizable so choose some metric $d$ and consider a sequence $(x_n)$ in $X$. For each $n\in \bN$ choose $y_n \in Y$ such that $d(y_n,x_n) < 1/n$. $(y_n)$ admits a convergent subsequence with limit $y$, which we index again as $y_n$ without loss of generality. Now, $d(x_n,y) \le d(x_n,y_n) + d(y_n,y)$ and in particular $(x_n)\to y$.}

    Consider a sequence $(\Sigma_\alpha)_{\alpha\in \bN}$ in $Y$ and up to passing to a subsequence suppose that $\Pi_{ij}^\alpha$ and $\Pi_{kl}^\alpha/\Pi_{ij}^\alpha$ converge in $\bP^1$ for all $i<j$ and $k<l$. By the previous argument, we may associate to $\{\Pi_{ij},\Pi_{kl}/\Pi_{ij}\}$ a level tree $\Gamma$. By construction, for all $\alpha$ sufficiently large, $\Sigma_\alpha \in U_\Gamma$. By \Cref{L:goodindicesCOV}, if $h(i) \wedge h(j)$ is on level $0\le p \le \ell$, one has $\Pi_{ij}= z_{ij}/t_{p+1} \cdots t_\ell$ and in particular $\Pi_{i_mj_m}^\alpha = 1/t_{m+1}^\alpha\cdots t_\ell^\alpha$. By construction of $\Gamma$, $\lim_\alpha \Pi^\alpha_{i_{m}j_{m}}/\Pi^\alpha_{i_{m-1}j_{m-1}} = \lim_\alpha t_m^\alpha = 0$ for each $1\le m \le \ell$. Also, for any $i,j$ with $h(i)\wedge h(j)$ on level $m$, $\lim_\alpha \Pi_{ij}^\alpha/\Pi_{i_mj_m}^\alpha = \lim_\alpha z_{ij}^\alpha \in \bC^*$ for $m<\ell$ and is in $\bC$ for $m= \ell$. So, $(\Sigma_\alpha)\to \Sigma \in U_\Gamma$ with $\Gamma(\Sigma) = \Gamma$. This gives compactness.

    Finally, a coordinate calculation shows that $\Pi_{ij}:\Mmscbar_n \to \bP^1$ is holomorphic\endnote{Consider $\Sigma$ where $\Pi_{ij}(\Sigma) = \infty$. Then on $\Gamma = \Gamma(\Sigma)$ one has $h(i) \ne h(j)$ and so on $U_\Gamma$ we have $\Pi_{ij} \ne 0$. In particular, $\Pi_{ij}^{-1} = t_{p+1}\cdots t_\ell/z_{ij}$ where $z_{ij}$ is nonvanishing on $U_\Gamma$ and $p<\ell$. Therefore, $\Pi_{ij}^{-1}$ is holomorphic at $\Sigma$ as needed.} and by proper GAGA \cite{SGA1}*{Expos\'{e} XII, Cor. 4.5}, it follows that $\Pi_{ij}:\Mmscbar_n \to \bP^1$ is algebraic.
\end{proof}

\subsection{Level tree stratification}

Let $\Gamma$ denote a dual level tree, by abuse of notation. Let $S_\Gamma^\circ \subset \Mmscbar_n$ denote the set of $\Sigma$ such that $\Gamma(\Sigma) = \Gamma$. By definition, $S_\Gamma^\circ \subset U_\Gamma$. Write $S_\Gamma$ for the closure of $S_\Gamma^\circ$ in $\Mmscbar_n$. 

\begin{lem}
\label{L:closuredescription}
    For any dual level tree $\Gamma$, one has $S_\Gamma = \{\Sigma\in \Mmscbar_n:\exists \text{ contraction }\Gamma(\Sigma)\twoheadrightarrow\Gamma\}$. Moreover, $S_\Gamma$ is a smooth closed subvariety of codimension the length of $\Gamma$. 
\end{lem}

\begin{proof}
    Write $\ell$ for the length of $\Gamma$. We verify the claims using the open cover $\{U_{\Gamma'}\}$. $U_{\Gamma'}\cap S_\Gamma^\circ \ne \varnothing$ if and only if there is a contraction $\Gamma'\twoheadrightarrow \Gamma$ and in this case $U_{\Gamma'}\cap S_\Gamma^\circ$ corresponds to the locus where $t_{j_1} = \cdots  = t_{j_\ell} = 0$ for some $j_1<\cdots<j_\ell$ and no other $t_m$ vanish. By this description, $\Sigma \in U_{\Gamma'} \cap (S_\Gamma\setminus S_\Gamma^\circ)$ if and only if $t_m(\Sigma) = 0$ for some $m\not \in \{j_1,\ldots,j_\ell\}$. So, $\Sigma \in U_{\Gamma'} \cap (S_\Gamma \setminus S_\Gamma^\circ)$ if and only if there exists a contraction $\Gamma(\Sigma) \twoheadrightarrow \Gamma$. Furthermore, $U_{\Gamma'}\cap S_\Gamma = Z(t_{j_1},\ldots, t_{j_\ell})$ and so $S_\Gamma$ is smooth of codimension $\ell$.
\end{proof}

 In light of \Cref{P:treeversuschain}, $\Gamma$ with $\Gamma/{\sim} \cong [\ell]$ corresponds to a unique chain $\mathfrak{c}(\Gamma) = \{\rho_1<\cdots<\rho_\ell\} \in \mathfrak{Ch}(L_n^-)$. As a consequence, we write $S_{\rho_\bullet}^\circ$ for $S_\Gamma^\circ$ and $S_{\rho_\bullet}$ for $S_\Gamma$. Put $S_\bot^\circ = U_*$ and note that $S_\bot = \Mmscbar_n$. In what follows, we favor the partition notation as introduced in \S2.

\begin{lem}
\label{L:strataprops}
    In the above notation
    \begin{enumerate}
        \item $S_{\rho_\bullet} = \bigcup_{\rho_\bullet \subseteq \chi_\bullet} S_{\chi_\bullet}^\circ$, where $\chi_\bullet$ runs over $\mathfrak{Ch}(L_n^-)$;
        \item for any $\rho_\bullet \in \mathfrak{Ch}(L_n^-)$, $S_{\rho_\bullet} \subset \Mmscbar_n$ is a smooth subvariety of codimension $|\rho_\bullet|$;
        \item $S_{\rho_1} \cap S_{\rho_2} \ne \varnothing$ if and only if $\rho_1,\rho_2\in L_n$ are comparable; 
        \item $D = \bigcup_{\rho \in L_n^-} S_\rho$ is a simple normal crossings divisor;
        \item For $\rho_\bullet = \{\rho_1< \cdots < \rho_\ell\}$, $S_{\rho_\bullet}^\circ = (S_{\rho_1}\cap \cdots \cap S_{\rho_\ell}) \setminus \bigcup_{\pi \in L_n^- \setminus \rho_\bullet} S_\pi$ and $S_{\rho_\bullet} = \bigcap_{i=1}^\ell S_{\rho_i}$. 
    \end{enumerate}
\end{lem}

\begin{proof}
(1) and (2) are restatements of \Cref{L:closuredescription}. For (3), by \Cref{L:closuredescription} $S_{\rho_1}\cap S_{\rho_2}$ consists of those $\Sigma$ for which there exist contractions $\Gamma(\Sigma)\twoheadrightarrow \Gamma(\rho_1)$ and $\Gamma(\Sigma)\twoheadrightarrow \Gamma(\rho_2)$. In particular, $\mathfrak{c}(\Sigma)$ contains $\rho_1$ and $\rho_2$ which implies they are comparable. Conversely, if $\rho_1<\rho_2$ are comparable, then $\Sigma$ with $\mathfrak{c}(\Sigma) = \{\rho_1<\rho_2\}$ lies in $S_{\rho_1}\cap S_{\rho_2}$. For (4), for any $U_\Gamma$ for which $U_\Gamma \cap S_{\rho_1}\cap S_{\rho_2}\ne \varnothing$, one has $S_\rho \cap U_\Gamma = Z(t_i)$ and $S_{\rho'}\cap U_\Gamma = Z(t_j)$ which implies the snc property. We leave (5) as an exercise.\endnote{$(S_{\rho_1}\cap \cdots \cap S_{\rho_\ell}) \setminus \bigcup_{\pi \in L_n^- \setminus \rho_\bullet} S_\pi$ is the set of $\Sigma \in \Mmscbar_n$ such that $\mathfrak{c}(\Sigma)$ contains $\rho_1,\ldots, \rho_\ell$ but no other $\pi$, by \Cref{L:closuredescription}. This is exactly $S_{\rho_\bullet}^\circ$. $S_{\rho_\bullet}$ consists of those $\Sigma$ for which $\{\rho_1,\ldots, \rho_\ell\}\subset \mathfrak{c}(\Sigma)$. This is precisely $\bigcap_i S_{\rho_i}$.}
\end{proof}

\begin{prop}
\label{P:stratificationofMn}
The collection $\{S_\bot^\circ\} \cup \{S_{\rho_\bullet}^\circ\}_{\rho_\bullet\in \mathfrak{Ch}(L_n^-)}$ gives a quasi-affine stratification of $\Mmscbar_n$ by smooth subvarieties.\endnote{By a \emph{stratification} of a topological space $X$, we mean a collection $\{S_i\}_{i\in I}$ of locally closed subspaces such that (1) $S_i\cap S_j = \varnothing$ if $i \ne j$, (2) $\bigcup_{i\in I} S_i = X$, and (3) if $S_i\cap \overline{S}_j \ne \varnothing$, then $S_i\subset \overline{S}_j$.
} 
\end{prop}

\begin{proof}
    Consider $\rho_\bullet =\{\rho_1<\cdots<\rho_\ell\}$. $S_{\rho_\bullet}^\circ = U_{\rho_\bullet} \cap Z(t_1,\ldots, t_\ell)$, so each $S_{\rho_\bullet}$ is locally closed and quasi-affine. Since each $\Sigma \in \Mmscbar_n$ has a unique dual level tree, $\Mmscbar_n = S_\bot^\circ \sqcup \bigsqcup_{\rho_\bullet} S_{\rho_\bullet}^\circ$. Finally, if $S_{\rho_\bullet} \cap S_{\chi_\bullet}^\circ$, then there exists a coarsening $\Gamma(\chi_\bullet)\twoheadrightarrow \Gamma(\rho_\bullet)$ and so $S_{\chi_\bullet}^\circ \subset S_{\rho_\bullet}$.  
\end{proof}

We call this the \emph{level tree stratification} of $\Mmscbar_n$.

\begin{rem}
    \Cref{P:stratificationofMn} combined with \cite{EisenbudHarris3264}*{Prop. 1.17} implies that $\mathrm{CH}^*(\Mmscbar_n)$ is generated as an Abelian group by the classes of the closed strata $S_{\rho_\bullet}$ for $\rho_\bullet \in \mathfrak{Ch}(L_n^-)$ and $\Mmscbar_n$. We will not pursue this here, however, as an explicit description of the ring structure $\mathrm{CH}^*(\Mmscbar_n)$ will be an immediate corollary of \Cref{T:isothm} and results of \cite{Huhetalsemismall}: see \Cref{C:chow}. 
\end{rem}

\subsection{Symmetric group action and collision stratification}

Define $\mathfrak{S}_n \times \Mmscbar_n \to \Mmscbar_n$ by $\sigma \cdot (\Sigma,\preceq,p_\infty,\omega_\bullet,p_1,\ldots, p_n) = (\Sigma,\preceq,p_\infty,\omega_\bullet,p_{\sigma(1)},\ldots, p_{\sigma(n)})$. 

Next, given $(\Gamma,h) \in \Delta(n)$ and $\sigma \in \mathfrak{S}_n$, we define $\sigma\cdot(\Gamma,h) = (\Gamma,h^\sigma)$ where $h^\sigma(i) = h(\sigma^{-1}(i))$. This defines an action of $\mathfrak{S}_n$ on $\Delta(n)$. We abusively write this as $\Gamma \mapsto \Gamma^\sigma$. 

\begin{prop}
\label{P:Snaction}
    $\mathfrak{S}_n\times \Mmscbar_n\to \Mmscbar_n$ defines an algebraic action of $\mathfrak{S}_n$ on $\Mmscbar_n$.
\end{prop}

\begin{proof}
    That this defines an action in the category of sets is immediate. Given $\sigma \in \mathfrak{S}_n$, we show that $\sigma:\Mmscbar_n \to \Mmscbar_n$ is algebraic. Note that $\sigma$ maps $U_\Gamma$ to $U_{\Gamma^\sigma}$. Now, consider the functions $\{z_{ij},t_m\}$ on $U_\Gamma$ defined by choosing indices $(i_m,j_m)$ for all $0\le m\le \ell-1$. We define functions $\{z_{ij}',t_m'\}$ on $U_{\Gamma^\sigma}$ by letting $(i_m',j_m') = (\sigma(i_m),\sigma(j_m))$ for all $0\le m \le \ell-1$. We leave it as an exercise to verify that with respect to these choices $\sigma^*(z_{ij}')  = z_{\sigma^{-1}(i)\sigma^{-1}(j)}$ and $\sigma^*(t_m') = t_m$.\endnote{First, observe that $\sigma:U_\Gamma \to U_{\Gamma^\sigma}$ lifts to a map $\widetilde{U}_\Gamma \to \widetilde{U}_{\Gamma^\sigma}$, where $\widetilde{U}_\Gamma$ denotes the set of stable $n$-marked multiscaled lines whose class lies in $U_\Gamma$. So, $h^\sigma(i)\wedge h^\sigma(j) = h(\sigma^{-1}(i))\wedge h(\sigma^{-1}(j))$. It follows that $I_{ij}(\sigma(\Sigma)) = I_{\sigma^{-1}(i)\sigma^{-1}(j)}(\Sigma)$ and thus that $\sigma^*(I_{ij}) = I_{\sigma^{-1}(i)\sigma^{-1}(j)}$. Therefore, $\sigma^*(s_m') = I_{\sigma(i_m)\sigma(j_m)} = I_{i_mj_m} = s_m$ and hence $\sigma^*(t_m') = t_m$ for all $m$. Similarly, $z_{ij}'(\sigma(\Sigma)) = I_{ij}(\sigma(\Sigma))/s_m'(\sigma(\Sigma)) = I_{\sigma^{-1}(i)\sigma^{-1}(j)}/s_m = z_{\sigma^{-1}(i)\sigma^{-1}(j)}(\Sigma)$ as claimed.
    } 
\end{proof}

\begin{rem}
\label{R:genericfree}
    The $\mathfrak{S}_n$ action of \Cref{P:Snaction} is generically free, but not free. Indeed, the transposition $(ij)$ is stabilizes $\Sigma$ if and only if $\Pi_{ij}(\Sigma) = 0$. 
\end{rem}

\begin{defn}
\label{D:Tstrata}
    For $\rho \in L_n$, let $T_\rho = \{\Sigma \in \Mmscbar_n: i\sim^\rho j\Rightarrow p_i =p_j\}$ and $T_{\rho}^\circ = \{\Sigma \in \Mmscbar_n: p_i = p_j \iff i\sim^\rho j\}$. 
\end{defn}

\begin{prop}
\label{P:collisionstratification}
    For all $\rho \in L_n$, $T_\rho$ is a closed subvariety of $\Mmscbar_n$ and $T_{\rho}^\circ$ is a locally closed subspace of $\Mmscbar_n$, dense in $T_\rho$. Furthermore, $\{T_\rho^\circ\}_{\rho \in L_n}$ forms a stratification of $\Mmscbar_n$. 
\end{prop}

\begin{proof}
    $T_\rho \subset \Mmscbar_n$ is a closed subspace, since it is globally defined by the conditions $\Pi_{ij}(\Sigma) = 0$ for all $i\sim^\rho j$. $T_\rho^\circ$ is locally closed, being the intersection of $T_\rho$ with the open sets $D(\Pi_{k\ell})$ for all $k\not\sim^\rho \ell$. The reader may verify that $\Mmscbar_n = \bigsqcup_{\rho \in L_n}T_\rho^\circ$ and that given $\eta,\rho \in L_n$, $T_\eta^\circ \cap T_\rho\ne \varnothing$ implies that $T_\eta^\circ \subset T_\rho$.\endnote{To see $\Mmscbar_n = \bigsqcup_{\rho \in L_n}T_\rho^\circ$ just note that each $\Sigma \in \Mmscbar_n$ is in $T_\rho^\circ$ if and only if $\Pi_{ij}(\Sigma) = \iff i\sim^\rho j$. This uniquely characterizes $\rho$. For the second claim, $T_\eta^\circ\cap T_\rho \ne \varnothing$ if and only if $i\sim^\rho j$ implies $i\sim^\eta j$. In particular, it follows that $\eta \le \rho$ and implies the desired containment.} 
\end{proof}

We call the stratification $\{T_\rho^\circ\}_{\rho \in L_n}$ of $\Mmscbar_n$ the \emph{collision stratification}. Note that $T_\bot$ consists of a single point, namely $[0]\in \bA^n/\bG_a$. On the other hand, $T_\top^\circ$ is the open dense subset of $\Mmscbar_n$ in which no points collide. Finally, note that by \Cref{R:genericfree} this stratification can be interpreted as being indexed by the stabilizer subgroup of $\mathfrak{S}_n$-action of \Cref{P:Snaction}. Occasionally, the following notion will be useful.

\begin{defn}
\label{D:Amarked}
    Let $A$ be a finite set. An $A$\emph{-marked} multiscaled line is a multiscaled line $(\Sigma,p_\infty,\preceq,\omega_\bullet)$ equipped with a map $p:A\to \Sigma$ such that $p(a) =: p_a$ is a smooth point on a terminal component of $\Sigma$ for all $a\in A$. $(\Sigma,p_\infty,\preceq,\omega_\bullet,p:A\to \Sigma)$ is \emph{stable} if each terminal component contains at least one $p_a$ and the other conditions of \Cref{D:stablenmarkedmultiscaledline} hold. 
    
    We define complex projective isomorphism of these objects as in \Cref{D:multiscaledline} and let $\Mmscbar_A$ denote the set of complex projective isomorphism classes of $A$-marked stable multiscaled lines. Choosing a bijection $q:\{1,\ldots,n\}\to A$ gives an identification $\Mmscbar_n \to \Mmscbar_A$ which we use to equip $\Mmscbar_A$ with the structure of a variety which is independent of choices up to isomorphism.\endnote{Indeed, given another bijection $q':\{1,\ldots, n\}\to A$, the variety structures on $\Mmscbar_A$ differ by a composition with an automorphism of $\Mmscbar_n$ given by a unique $\sigma \in \mathfrak{S}_n$.}
\end{defn}

\begin{prop}
\label{P:recursivesub}
    For any $\rho \in L_n$, the map $\kappa_\rho: T_\rho \to \Mmscbar_{B(\rho)}$ given by 
    \[
        (\Sigma, p_\infty,\preceq,\omega_\bullet,p_1,\ldots, p_n) \mapsto (\Sigma,p_\infty,\preceq,\omega_\bullet,\{p_{\mu(b)}:b\in B(\rho)\})
    \]
    is an isomorphism.
\end{prop}

\begin{proof}
    $\kappa_\rho(\Sigma)$ is a stable $B(\rho)$-marked multiscaled line by definition, since every terminal component contains at least one marked point. Define $\lambda_\rho:\Mmscbar_{B(\rho)}\to T_\rho$ by attaching marked points $p_j$ for all $j\in b\setminus \{\mu(b)\}$ and $b\in B(\rho)$ such that $p_j = p_{\mu(b)}$ if $j\in b$. One can verify that $\lambda_\rho$ and $\kappa_\rho$ are mutually inverse. 

    Let $\Gamma \in \Delta(n)$ be given such that $U_\Gamma \cap T_\rho \ne \varnothing$ and choose indices $(i_m,j_m)$ for $\Gamma$ such that $i_m$ and $j_m$ are minimal in the blocks of $\rho$ in which they lie.\endnote{This is possible because $i\sim^\rho j$ implies that $\Pi_{ij}(\Sigma) = 0$ by definition of $T_\rho$ and hence that $h(i) = h(j)$ by definition of $U_\Gamma$. Consequently, if $i_m\in b$ and $j_m\in b'$, then $h(i_m) = h(\mu(b))$ and $h(\mu(b')) = h(j_m)$ and so $h(i_m)\wedge h(j_m) = h(\mu(b))\wedge h(\mu(b'))$.} Let $\cI(\rho) = \{1,\ldots,n\}\setminus (\bigcup_{b\in B(\rho)} b\setminus \mu(b))$ and observe that by definition, $\kappa_\rho(\Sigma)$ has the same underlying level tree as $\Sigma$, except that now it is marked by $\cI(\rho)$ which is in canonical bijection with $B(\rho)$. Write the resulting $B(\rho)$-marked level tree by $\Gamma_-$. Choose indices for $U_{\Gamma_-}\subset \cA_{B(\rho)}$ by $i_m' = i_m$ and $j_m' = j_m$, possible by our choices above. Then, one can verify that $\kappa_\rho^*(z_{ij}') = z_{ij}$ for all $i,j\in \cI(\rho)$ and $\kappa_\rho^*(t_m') = t_m$ for all $m$.\endnote{$\widetilde{U}_\Gamma$ denotes the set of stable $n$-marked multiscaled lines whose class lies in $U_\Gamma$. The restriction of $\kappa_\rho$ to $U_\Gamma$ lifts to a map $\widetilde{U}_\Gamma \to \widetilde{U}_{\Gamma_-}$, which we denote by $\kappa_\rho$ again by abuse of notation. Then, for all $i,j \in \cI(\rho)$ one has $I_{ij}(\kappa_\rho(\Sigma)) = I_{ij}(\Sigma)$ and so $\kappa_\rho^*(I_{ij}) = I_{ij}$. Therefore, $\kappa_\rho^*(s_m') = s_m$ for all $m$ and $\kappa_\rho^*(t_m') = t_m$ and $\kappa_\rho^*(z_{ij}') = z_{ij}$.} In particular, $\kappa_\rho$ is an algebraic bijection between smooth varieties and consequently an isomorphism.
\end{proof}

\subsection{Multiscaled lines as an equivariant compactification}

$\bb{A}^n/\bb{G}_a$ carries a free and transitive $\bb{G}_a^n/\Delta$-action. We put $G := \bG_a^n/\Delta$ and prove that $\Mmscbar_n$ is a $G$-equivariant compact\-ification of $\bA^n/\bG_a$ as studied in \cite{HassettTschinkelequiv}. 

Given $a = (a_1,\ldots, a_n) \in \bb{G}_a^n$, and $(\Sigma,p_\infty,\omega_\bullet,p_1,\ldots, p_n) \in \Mmscbar_n$ we define 
\[
    a\cdot (\Sigma,\preceq, p_\infty,\omega_\bullet, p_1,\ldots, p_n) := (\Sigma,\preceq, p_\infty,\omega_\bullet,p_1+a_1,\ldots, p_n+a_n).
\]
Here, $p_i+a_i$ means translating the marked point $p_i$ by $a_i$ in its terminal component $\Sigma_{h(i)}$. This can be defined by choosing an isomorphism of the smooth locus of $\Sigma_{h(i)}$ with $\bA^1 = \Spec \bC[z]$ such that $dz$ pulls back to $\omega$. It is independent of choices up to isomorphism of $\Sigma$. 

This defines an action $\bG_a^n \times \Mmscbar_n \to \Mmscbar_n$ in the category of sets which preserves dual level trees and extends the translation action of $\bG_a^n$ on $\bA^n/\bG_a$. The diagonal subgroup $\Delta \subset \bG_a^n$ acts trivially and so this induces an action by $G$.

\begin{prop}
\label{P:GactiononMn}
    $\bG_a^n \times \Mmscbar_n \to \Mmscbar_n$ defines an action in the category of algebraic varieties which induces a $G$-action with respect to which $\bA^n/\bG_a\hookrightarrow\Mmscbar_n$ is equivariant.
\end{prop}

\begin{proof}
    All of the claims follow from the above discussion except for algebraicity. Let $a\in \bG_a^n$ be given and consider the induced map $a:U_\Gamma \to U_\Gamma$. Choose arbitary indices for $U_\Gamma$ and consider the resulting functions $\{z_{ij},t_m\}$. One can verify that $a^*(z_{ij}) = z_{ij}$ for all $i,j$ such that $h(i)\wedge h(j)$ is on level $m\le \ell-1$ and that for all $m$, $a^*(t_m) = t_m$. Finally, for $i,j$ such that $h(i) = h(j)$, one has $a^*(z_{ij}) = z_{ij} + a_j-a_i$.\endnote{$a:U_\Gamma \to U_\Gamma$ lifts to a map $\widetilde{U}_\Gamma$, where $\widetilde{U}_\Gamma$ denotes the set of stable $n$-marked multiscaled lines whose class lies in $U_\Gamma$. We denoted the lifted map by $a$ by abuse of notation. By definition of $a$, one has that $a^*(I_{ij}) = I_{ij}$ for all $i,j$ such that $h(i)\ne h(j)$. In the case where $h(i) = h(j)$, $p_i$ is translated by $a_i$ and $p_j$ by $a_j$, so that $z_{ij}'(a\cdot(\Sigma)) = z_{ij} + a_j - a_i$ as claimed.} 
\end{proof}

\begin{rem}
     Fixed points for the $G$-action are in general not isolated: one can verify that the $G$-fixed locus in $\Mmscbar_n$ is the divisor $S_\top$, i.e. the set of $\Sigma$ for which each terminal component contains exactly one marked point.\endnote{The description of $S_\top$ is by \Cref{L:strataprops}(1). $(\Sigma,p_1,\ldots,p_n)\in \Mmscbar_n$ is in the fixed locus of $G$ if and only if $\Pi_{ij}(a\cdot \Sigma) = \Pi_{ij}(\Sigma)$ for all $a\in G$. This is only possible if no pair $i,j$ lie on the same terminal component; if not, acting by $a$ where $a_i\ne a_j$ would result in $\Pi_{ij}(a\cdot \Sigma) = \Pi_{ij}(\Sigma) + a_j-a_i\ne \Pi_{ij}(\Sigma)$.} Note that the $\mathfrak{S}_n$ and $G$-actions do not commute.\endnote{This can be verified just by considering $\mathfrak{S}_2$ and $\bG_a^2/\Delta$ acting on $\bA^2/\bG_a$. Indeed, let $\sigma \ne e\in \mathfrak{S}_2$ and let $(a_1,a_2)\in \bG_a^2$ be given. One has $\sigma(a\cdot(x,y)) = (y+a_2,x+a_1)$ whereas $a\cdot(\sigma(x,y)) = (y+a_1,x+a_2)$.}
\end{rem}

\section{Comparison with the augmented wonderful variety}

\subsection{Iterated blowup constructions}

In this section, we recall a construction of the augmented wonderful variety $W_n$ as defined in \cite{Huhetalsemismall}. First, we recall some techniques for varieties constructed as iterated blowups from the literature \cites{Hucompactification,Liwonderful,Ulyanovpoly}. We have modified the notation slightly to be compatible with ours.

\begin{defn}
\label{D:arrangement}
    A set $\cH = \{H_i\}_{i\in I}$ of subvarieties of a nonsingular variety $X$ is called an \emph{arrangement} if 
    \begin{enumerate}
        \item each $H_i$ is smooth;
        \item every pair $H_i$ and $H_j$ meets cleanly (\Cref{D:cleanint}); and 
        \item $H_i\cap H_j = \varnothing$ or a disjoint union of $H_\ell$.
    \end{enumerate}
\end{defn}

\begin{thm}
[\cite{Hucompactification}]
\label{T:Hu}
    Let $X^0$ be an open subset of a nonsingular algebraic variety $X$. Suppose $X\setminus X^0$ can be decomposed as a union $\bigcup_{i\in I} H_i$ of closed irreducible subvarieties that form an arrangement. It follows that $\cH = \{H_i\}_{i\in I}$ is a poset. Let $r$ be the rank of $\cH$. There is a sequence of well-defined blowups 
    \[
        \Bl_{\cH}X\to \Bl_{\cH_{\le r-1}} X \to \cdots \to \Bl_{\cH_{\le 0}} X \to X
    \]
    where $\Bl_{\cH_{\le 0}} X \to X$ is the blowup of $X$ along $H_i$ of rank $0$, and, inductively, $\Bl_{\cH_{\le s}} X \to \Bl_{\cH_{\le s-1}}X$ is the blowup of $\Bl_{\cH_{\le s-1}}X$ along the proper transforms of $H_j$ of rank $s$, such that 
    \begin{enumerate}
        \item[(a)] $\Bl_{\cH} X$ is smooth;
        \item[(b)] $\Bl_{\cH} X\setminus X^0 = \bigcup_{i\in I} \widetilde{H}_i$ is a normal crossings divisor;
        \item[(c)] $\widetilde{H}_{i_1}\cap \cdots \cap \widetilde{H}_{i_n}$ is nonempty if and only if $H_{i_1},\ldots, H_{i_n}$ form a chain in the poset $\cH$. Consequently, $\widetilde{H}_i$ and $\widetilde{H}_j$ meet if and only if $H_i$ and $H_j$ are comparable.
    \end{enumerate}
\end{thm}

\begin{defn}
\label{D:wonderfulvariety}
    We specialize now to $\cH$ as in \Cref{D:subspacearrangement}. By \Cref{L:propertiesofHs} this is an arrangement of varieties in the sense of \Cref{D:arrangement} inside of $\bP(V)$. We define the \emph{augmented wonderful variety} $W_n$ to be the variety obtained by applying the procedure of \Cref{T:Hu} to $\cH$. Write $D_\rho$ for $\widetilde{H}_\rho \subset W_n$.
\end{defn}

\begin{rem}
    \Cref{D:wonderfulvariety} is equivalent to the one given in \cite{Huhetalsemismall}, though the reader should be mindful of the fact that the ``rank'' in \Cref{T:Hu} is the $\text{corank}-1$ in the language of \cite{Huhetalsemismall}. This is essentially a consequence of the choice of partial order on $L_n$ (see \Cref{P:lattice}, \Cref{N:dimpartition}, and the related discussion).
\end{rem}

\begin{cor}
\label{C:wonderfulaug}
    $W_n$ is a smooth projective variety equipped with a birational morphism $\pi:W_n\to \bP(V)$, containing $\bA^n/\bG_a\subset V$ such that 
    \begin{enumerate}
        \item $W_n \setminus (\bA^n/\bG_a) = \bigcup_{\rho \in L_n^-} D_\rho$ is a simple normal crossings divisor; and 
        \item $D_{\rho_1}\cap \cdots \cap D_{\rho_k}$ is nonempty if and only if $\rho_1,\ldots, \rho_k$ form a chain in $L_n$. 
    \end{enumerate}
\end{cor}

\begin{proof}
    This is all an immediate consequence of \Cref{T:Hu}. 
\end{proof}

\subsection{The isomorphism}

We first construct a morphism $f_0: \Mmscbar_n \to W_n$. There is a birational map $f_0:\Mmscbar_n \dashrightarrow \bP(V)$ given by identifying $\bA^n/\bG_a \subset \Mmscbar_n$ with $\bA^n/\bG_a \subset \bP(V)$. Thus, 
\begin{equation} 
\label{E:pullbackperiod}
\Pi_{ij} = f_0^*(P_{ij}/t) = f_0^*(P_{ij})/f_0^*(t).
\end{equation}

\begin{prop}
\label{P:morphismtopn}
    $f_0$ extends uniquely to a morphism of algebraic varieties $f_0:\Mmscbar_n \to \bP(V)$.
\end{prop}

\begin{proof}
    Uniqueness is automatic if the extension exists. $P_{12},\ldots, P_{1n},t$ are global generators of $\cO_{\bP(V)}(1)$. Let $\Sigma\in \Mmscbar_n$ be given with reducible underlying curve and write $\Gamma = \Gamma(\Sigma)$. Choose $j$ such that $h(1)\wedge h(j)$ is the root. In particular, $\Pi_{1j}\ne 0$ on $U_{\Gamma}$. There is an induced rational map $f_0:U_\Gamma \dashrightarrow D(P_{1j})\subset \bP(V)$. $D(P_{1j}) \cong \Spec \bb{C}[\frac{t}{P_{1j}},\frac{P_{1i}}{P_{1j}}]$ so $f_0$ extends to a morphism on $U_\Gamma$ if the pullbacks of $t/P_{1j}$ and the $P_{1i}/P_{1j}$ are regular on $U_\Gamma$. $f_0^*(t/P_{1j}) =  1/\Pi_{1j}$ and $f_0^*(P_{1i}/P_{1j}) = \Pi_{1i}/\Pi_{1j}$ both of which are regular on $U_\Gamma$ by our choice of $j$.\endnote{By \Cref{L:goodindicesCOV}, $\Pi_{1j} = (t_1\cdots t_\ell)^{-1}$. For any other $i$, let $h(1)\wedge h(i)$ be on level $m\ge 0$. Then, $\Pi_{1i}/\Pi_{1j} = (t_1\cdots t_\ell)/(t_{m+1}\cdots t_\ell) = t_1\cdots t_m$ and the ratio is regular on $U_\Gamma$.} Consequently, $f_0$ extends to $\Sigma$.
\end{proof}

The rest of this section is dedicated to proving:

\begin{thm}
\label{T:isothm}
    For any $n \ge 1$, there exists an isomorphism $f:\Mmscbar_n \to W_n$ over $\bP(V)$. Furthermore, for all $\rho_\bullet \in \mathfrak{Ch}(L_n^-)$, $f$ maps $S_{\rho_\bullet}^\circ$ to $D_{\rho_\bullet}^\circ$.
\end{thm}

Since $W_n$ is constructed as a sequence of blowups, the proof consists of a fairly intricate induction. We construct a map $f_k:\Mmscbar_n \to \Bl_{\le k} \bP(V)$ inductively by applying \Cref{L:univpropblowup}. Then, we need to analyze the map $f_k$ to deduce enough properties to build $f_{k+1}$. We begin by gathering some notation for later reference. The reader is also advised to refer to \Cref{N:dimpartition} and \Cref{D:Nsets}.

\begin{notn}
\label{N:mainnotation}
Fix $0\le k \le n-2$. 
    \begin{enumerate} 
        \item $\Bl_{\le k}\bP(V)$ denotes the space $\Bl_{\cH_{\le k}} \bP(V)$ in the notation of \Cref{T:Hu}. We write $\epsilon_k: \Bl_{\le k}\bP(V)\to \Bl_{\le k-1}\bP(V)$ for the blowup morphism and $\epsilon_{\le k}$ for the composite $\epsilon_1\circ\cdots\circ \epsilon_k$.  
        \item If $\pi \in (L_n^-)_{\le k}$, $D_\pi$ is the divisor in $\Bl_{\le k}\bP(V)$ obtained by iterated strict transforms of $H_\pi$. This depends on $k$ but it should be clear from the context which $D_\pi$ we are referring to. Note also that this coincides with notation in \Cref{D:wonderfulvariety} when $k = n - 2$. 
        \item We write $\widetilde{H}_\pi$ to refer to the strict transform of $H_\pi$ in $\Bl_{\le j} \bP(V)$ for $\dim \pi > j$. When $\dim \pi \le j$, $\widetilde{H}_\pi = D_\pi$.
        \item For $\rho_\bullet = \{\rho_1<\cdots<\rho_\ell\} \in \mathfrak{Ch}(L_n^-)_{\le k}$, as subspaces of $\Bl_{\le k} \bP(V)$ put  
        \begin{enumerate}
            \item $D_{\rho_\bullet} = D_{\rho_1} \cap \cdots \cap D_{\rho_\ell}$
            \item $D_{\rho_\bullet}^k = D_{\rho_\bullet} \setminus (\bigcup_{\dim \lambda \le k,\lambda \not\in \rho_\bullet} D_\lambda)$
            \item $D_{\rho_\bullet}^\circ = D_{\rho_\bullet} \setminus (\bigcup_{\lambda \not \in \rho_\bullet} \widetilde{H}_\lambda)$
            \item $S_{\rho_\bullet} = S_{\rho_1}\cap \cdots \cap S_{\rho_k}$ 
            \item $S_{\rho_\bullet}^k = S_{\rho_\bullet} \setminus (\bigcup_{\dim \lambda \le k,\lambda \not\in \rho_\bullet} S_\lambda)$ 
            \item $S_{\rho_\bullet}^\circ = S_{\rho_\bullet} \setminus (\bigcup_{\lambda \not \in \rho_\bullet} S_\lambda)$
            \item $\partial S_{\rho_\bullet} = S_{\rho_\bullet}\setminus S_{\rho_\bullet}^\circ$ and $\partial S_{\rho_\bullet}^k = S_{\rho_\bullet}^k\setminus S_{\rho_\bullet}^\circ$
            \item When $\ell = 1$, we write $D_{\rho_1} = D_{\rho_\bullet}$, etc.
        \end{enumerate}
        \item Given $\pi <\eta<\top$ in $L_n^-$, we define $\bP(N_{\pi|\eta}) = \Proj\bC[s_{ij}:(i,j)\in N_{\pi|\eta}]$. For $\pi \in L_n^-$ with $\pi<\top$, we define $\bP(N_{\pi|\top}) = \Proj \bC[s_{ij},T:(i,j)\in N_{\pi|\top}]$. Given $\pi < \eta < \top$, since $\eta > \pi$, for all $(i,j) \in N_{\eta|\top}$ we can write 
        \[
            P_{ij} = \sum_{(\alpha,\beta)\in N_{\pi|\top}} a_{ij}^{\alpha\beta}P_{\alpha\beta} \:\text{  where }\: a_{ij}^{\alpha\beta}\in \bC.
        \]
        Define $s_{ij} = \sum_{(\alpha,\beta)} a_{ij}^{\alpha\beta}s_{\alpha\beta}$ and $Z(N_{\eta|\top}) = Z(\{s_{ij}:(i,j)\in N_{\eta|\top}\} \cup \{T\}) \subset \bP(N_{\pi|\top})$. Also, we define $\bA(N_{\pi|\top}) = \Spec \bC[s_{ij}/T: (i,j)\in N_{\pi|\top}]$. In these constructions, the coordinates corresponding to pairs $(i,j)$ are ordered lexicographically, and the coordinate corresponding to $T$ is last. 
        \item For $\rho <\pi< \top$, put $\Pi_{\rho|\pi} = \{\Pi_{ij}:(i,j)\in N_{\rho|\pi}\}$ and $\Pi_{\rho|\top} = \{\Pi_{ij}:(i,j)\in N_{\rho|\pi}\} \cup \{1\}$, where we regard $1$ as the constant function (compare with \Cref{D:Nsets}). Finally, put $\Pi_{\rho|\top}^- = \{\Pi_{ij}:(i,j)\in N_{\rho|\pi}\}.$
        \item For $1\le \alpha <\beta\le n$, and $\rho < \pi$ put
        \[ 
        P_{\rho|\pi}\cdot P_{\alpha\beta}^{-1} := 
        \begin{cases} 
        \{P_{ij}/P_{\alpha\beta}:(i,j) \in N_{\rho|\pi}\}& \pi \ne \top\\
        \{P_{ij}/P_{\alpha\beta}:(i,j) \in N_{\rho|\top}\} \cup \{t/P_{\alpha\beta}\} &\pi = \top 
        \end{cases}
        \] 
        and 
        \[
        \Pi_{\rho|\pi}\cdot \Pi_{\alpha\beta}^{-1} := 
        \begin{cases} 
        \{\Pi_{ij}/\Pi_{\alpha\beta}:(i,j)\in N_{\rho|\pi}\}& \pi \ne \top\\
        \{\Pi_{ij}/\Pi_{\alpha \beta}:(i,j)\in N_{\rho|\top}\} \cup \{\Pi_{\alpha\beta}^{-1}\}& \pi =\top.
        \end{cases}
        \]
        \item Let $X$ be a variety, $D\subset X$ a divisor, $U\subset X$ an open subset, and $F =\{f_1,\ldots, f_c\} \subset \Gamma(U,\cO_X)$ vanishing along $U\cap D$. If for any $x\in D$ one has $df_{i,x}\ne 0$ for some $i$, we write $[dF_x] := [df_{1,x}:\cdots:df_{c,x}] \in \bP^{c-1}$, where $df_{j,x}/df_{i,x} \in \bC$ refers to the ratio of the values of $df_{j,x}$ and $df_{i,x}$ regarded elements of $(\cN_{D|X})^\vee_x$.
    \end{enumerate}
\end{notn}

\begin{lem}
\label{L:openpreimage}
For all $1\le i<j\le n$, one has $f_0^{-1}D(P_{ij}) = D(\Pi_{ij})$. 
\end{lem}

\begin{proof}
    This is immediate from \eqref{E:pullbackperiod}.\endnote{$\Sigma \in f_0^{-1}(D(P_{ij}))$ if and only if $P_{ij}(f_0(\Sigma)) \ne 0$. If $f_0(\Sigma) \in D(t)$, then $P_{ij}(f_0(\Sigma)) \ne 0$ if and only if $\Pi_{ij}(\Sigma) \ne 0$, by $f_0^*(t/P_{ij}) = \Pi_{ij}$. If $f_0(\Sigma) \in Z(t)\cap D(P_{ij})$ then $f_0(\Sigma) \in Z(t/P_{ij})$ and this is equivalent to $\Pi_{ij}^{-1}(\Sigma) = 0$.}
\end{proof}

\begin{lem}
\label{L:stratificationofBlk}
    For all $1\le k \le n-2$, $\bigcup_{\dim \pi \le k}D_\pi$ is a simple normal crossings divisor in $\Bl_{\le k}\bP(V)$. Write $\Bl_{\le k}\bP(V) = U \sqcup \bigcup_{\dim \pi \le k}D_\pi$. Then $\{U\}\cup\{D_{\rho_\bullet}^k:\rho_\bullet \in \mathfrak{Ch}(L_n^-)_{\le k}\}$ is a stratification of $\Bl_{\le k} \bP(V)$.
\end{lem}

\begin{proof}
    The simple normal crossings claim comes from applying the procedure of \Cref{T:Hu} to the poset $\cH_{\le k}$. It is a general fact that a decomposition of a variety $X = U \cup D$ where $D = \bigcup_{i=1}^m D_i$ is a simple normal crossings divisor induces a stratification of the claimed type.\endnote{Suppose given a decomposition of a variety $X = U \cup D$ where $D = \bigcup_{i=1}^m D_i$ is a simple normal crossings divisor. We claim $\{U\} \cup \{D_{i_1,\ldots, i_k}^\circ\}_{1\le i_1<\cdots<i_k\le m}$ forms a stratification of $X$, where $D_{i_1,\ldots, i_k}^\circ = D_{i_1}\cap \cdots \cap D_{i_k} \setminus \bigcup_{j \not \in \{i_1,\ldots, i_k\}} D_j$. Each $D_{i_1,\ldots, i_k}^\circ$ is locally closed as is $U$. It is clear that each $x\in X$ lies either in $U$ or some $D_{i_1,\ldots,i_k}^\circ$. Finally, if $x\in D_{i_1,\ldots, i_k}^\circ \cap D_{j_1,\ldots, j_\ell} \ne \varnothing$, where $D_{j_1,\ldots, j_\ell} = D_{j_1}\cap \cdots \cap D_{j_\ell}$ is the closure of $D_{j_1,\ldots, j_\ell}^\circ$, we have that each $D_{j_a}$ is one of the $D_{i_b}$ for all $1\le a \le \ell$. Consequently, $D_{i_1,\ldots, i_k}^\circ \subset D_{j_1,\ldots, j_\ell}$ as claimed.}
\end{proof}

\begin{lem}
\label{L:zeroratio}
    For $\pi \in L_n^-$, if $\alpha \sim^\pi \beta$ and $i\not\sim^\pi j$ then $S_\pi \subset D(\Pi_{ij})$ and $\Pi_{\alpha\beta}/\Pi_{ij}$ vanishes along $S_\pi$. 
\end{lem}

\begin{proof}
    By \Cref{L:strataprops}, $S_\pi = \{\Sigma \in \Mmscbar_n:\pi \in \mathfrak{c}(\Sigma)\}$. So, $\pi \le \max \mathfrak{c}(\Sigma)$ and $p_i$ and $p_j$ are on distinct terminal components of $\Sigma$. Therefore, $\Pi_{ij}(\Sigma)\ne 0$ on $S_\pi$ and in fact $\Pi_{ij}^{-1}(\Sigma) = 0$. As $\alpha \sim^\pi \beta$, for all $\Sigma \in S_\pi^\circ$ we have $\lvert \Pi_{\alpha\beta}(\Sigma)\rvert <\infty$. Consequently, $(\Pi_{\alpha\beta}/\Pi_{ij})(\Sigma) = 0$ for all $\Sigma \in S_\pi^\circ$; however, as $S_\pi$ is the closure of $S_\pi^\circ$, the result follows.
\end{proof}

\begin{lem}
\label{L:preimageofsubspace}
    For all $\pi \in L_n^-$, one has $f_0^{-1}(H_\pi) = \bigcup_{\rho\le \pi} S_\rho$.
\end{lem}

\begin{proof}
    We first prove $\bigcup_{\rho \le \pi} S_\rho \subset f_0^{-1}(H_\pi)$. If $\Sigma \in S_\rho^\circ$ for $\rho \le \pi$, then if $i\sim^\pi j$ and $k\not\sim^\rho \ell$, one has $k\not\sim^\pi \ell$ and thus $(\Pi_{ij}/\Pi_{k\ell})(\Sigma) = 0$ by \Cref{L:zeroratio}. By \Cref{L:openpreimage}, $\Pi_{k\ell}(\Sigma) \ne 0$ implies $f_0(\Sigma) \in D(P_{k\ell})$. So, $S_\rho^\circ \subset Z(\Pi_{ij}/\Pi_{k\ell}:i\sim^\pi j)$ and thus maps into $H_\pi \cap D(P_{k\ell})$. By closedness of $H_\pi$, it follows that $S_\rho \subset f_0^{-1}(H_\pi)$. 
    
    Suppose $\Sigma \in S_\eta \cap f_0^{-1}(H_\pi)$ and that $\eta$ and $\pi$ are not comparable. Write $\mu = \max \mathfrak{c}(\Sigma)$. Using \Cref{L:comparabilitycriteria}, choose $(k,\ell)$ and $(\alpha,\beta)$ such that $\alpha \sim^\pi \beta$ but $\alpha \not\sim^\eta \beta$, and $k\sim^\eta \ell$ but $k\not\sim^\pi \ell$. By \Cref{L:zeroratio}, $S_\eta\subset D(\Pi_{\alpha \beta})$. As $\Sigma \in f_0^{-1}(H_\pi)$, we also have $(\Pi_{\alpha\beta}/\Pi_{k\ell})(\Sigma) = 0$ where $\Pi_{k\ell}(\Sigma)\ne 0$ by \Cref{L:openpreimage}. Since $\Pi_{\alpha\beta}(\Sigma) \ne 0$, one has $\Pi_{k\ell}^{-1}(\Sigma) = 0$, and it follows that $p_k$ and $p_\ell$ are on different components of $\Sigma$. That is: $k\sim^\eta \ell$ and $k\not\sim^\pi \ell$ implies $k\not\sim^\mu \ell$. However, this implies $\pi \le \mu$: indeed, if $i\sim^\mu j$, then $i\sim^\eta j$ by $\eta \le \mu$ and thus $i\sim^\pi j$, else by the previous sentence $i\not\sim^\mu j$. So, $\Sigma \in S_\mu \cap f_0^{-1}(H_\pi)$ with $\mu \ge \pi$ and it suffices to show $f_0^{-1}(H_\pi)\cap S_\eta = \varnothing$ for $\eta > \pi$. 

    Suppose $\eta>\pi$ and that $\Sigma \in S_\eta\cap f_0^{-1}(H_\pi)$ is given so that $\pi<\eta\le \mu.$ There exists $(k,\ell)$ such that $k\sim^\pi \ell$ but $k\not\sim^\mu \ell$. As $H_\pi$ is covered by open sets of the form $H_\pi \cap D(P_{ij})$ for $i\not\sim^\pi j$, we may suppose $\Sigma \in f_0^{-1}(H_\pi \cap D(P_{ij}))$ for some such $i\not\sim^\pi j$ and by \Cref{L:openpreimage} that $\Sigma \in D(\Pi_{ij})$. Now, $(\Pi_{k\ell}/\Pi_{ij})(\Sigma) \ne 0$ and so $\Sigma \not \in f_0^{-1}(H_\pi)$, giving a contradiction.
\end{proof}

\begin{lem}
\label{L:quasiclosedpreimage}
    For any $\rho \in L_n^-$ of dimension $k$, one has $f_0^{-1}(H_\rho^\circ) = S_\rho^k$ and in particular $S_\rho^\circ \subset f_0^{-1}(H_\rho^\circ)\subset S_\rho$. 
\end{lem}

\begin{proof}
    Applying \Cref{L:preimageofsubspace}, we have 
    \[
    f_0^{-1}(H_\rho^\circ) = f_0^{-1}(H_\rho \setminus \bigcup_{\chi<\rho} H_\chi) 
         = f_0^{-1}(H_\rho)\setminus \bigcup_{\chi <\rho} f_0^{-1}(H_\chi) 
         = \bigcup_{\chi \le \rho} S_\rho \setminus (\bigcup_{\chi<\rho}\bigcup_{\eta\le \chi} S_\eta) = S_\rho^k.
    \]
    Now, as $S_\rho^\circ \subset S_\rho^k \subset S_\rho$, the result follows.
\end{proof}

\begin{lem}
\label{L:nonvanishingdifferential}
    Let $\pi \in L_n^-$ be given and suppose $i\not\sim^\pi j$. By \Cref{L:zeroratio}, $S_\pi \subset D(\Pi_{ij})$ so $\Pi_{ij}^{-1}$ is regular in a neighborhood of $S_\pi$. For all $\Sigma \in S_\pi^\circ$, $d\Pi_{ij}^{-1}(\Sigma) \ne 0$.
\end{lem}

\begin{proof}
    On $U_\pi$, there is a single coordinate $t = t_1$ and a pair of indices $(k,\ell) = (i_1,j_1)$ so that $\Pi_{k\ell} = 1/t_1$. Also, $S_\pi^\circ \subset U_\pi$. By a coordinate calculation, $(\Pi_{ij}/\Pi_{k\ell})(\Sigma) \in \bC^*$ for all $\Sigma \in S_\pi^\circ$ and $i\not\sim^\pi j$.\endnote{$h(i) \ne h(j)$ so $\Pi_{ij} = z_{ij}/t$, where $z_{ij}$ is nonvanishing on $U_\Gamma$. Consequently, $\Pi_{ij}/\Pi_{k\ell} = z_{ij}$.} By \Cref{R:projectioncoords}, we may choose indices $(k,\ell)$ such that $\{z_{ij},t\}$ identifies $U_\pi$ with the complement of a hyperplane arrangement in $\bA^{n-1}$. Therefore, $dt = d(\Pi_{k\ell}^{-1})$ does not vanish at any $\Sigma \in S_\pi^\circ$ and for any $i<j$ as in the statement, the Leibniz rule gives $d(\Pi_{ij}^{-1})_\Sigma = (\Pi_{k\ell}/\Pi_{ij})(\Sigma)\cdot d(\Pi_{k\ell}^{-1})_\Sigma$ and the result follows by $(\Pi_{k\ell}/\Pi_{ij})(\Sigma) \ne 0$.
\end{proof}

By default, $\rho_\bullet$ denotes a chain $\{\rho_1<\cdots<\rho_\ell\}\subset \mathfrak{Ch}(L_{n}^-)$ of length $\ell$. By convention we put $\rho_0 = \bot$ and $\rho_{\ell+1} = \top$.

\begin{lem}
\label{L:injectivemap}
    The map $S_{\rho_\bullet}^\circ\to \prod_{k=0}^{\ell-1} \bP(N_{\rho_k|\rho_{k+1}}) \times \bA(N_{\rho_\ell|\top})$ given by 
    \[
        \Sigma \mapsto ([\Pi_{\bot|\rho_1}(\Sigma)],\ldots, [\Pi_{\rho_{\ell-1}|\rho_{\ell}}(\Sigma)], \Pi_{\rho_\ell|\top}^-(\Sigma))
    \]
    is well-defined and injective.
\end{lem}

\begin{proof}
    Under the correspondence of \Cref{Const:treesandpartitions} and \Cref{P:treeversuschain}, vertices on level $0\le k\le \ell-1$ of $\Gamma = \Gamma(\rho_\bullet)$ correspond to blocks $B\in B(\rho_k)$. Identify $B$ with the vertex it corresponds to and let $N_B$ denote the set of nodes connecting $\Sigma_B$ to higher level components of $\Sigma$. $N_B$ is naturally labeled as $\{n_b: b\in B(\rho_{k+1})\text{ and } b\subset B\}$. For $\Sigma \in U_\Gamma$, let $n_B$ be the node connecting $\Sigma_B$ to the marked point of minimal index above it. $\bigcup_{B\in B(\rho_k)} \{(n_B,n_b):n_b\in N_B\setminus n_B\}\to N_{\rho_k|\rho_{k+1}}$ given by $(n_B,n_b)\mapsto (\mu(B),\mu(b))$ is a bijection such that $\int_{n_B}^{n_b}\omega_B = I_{\mu(B)\mu(b)}$. In particular, for any $(i,j),(\alpha,\beta) \in N_{\rho_k|\rho_{k+1}}$ we have $\Pi_{ij}(\Sigma)/\Pi_{\alpha\beta}(\Sigma) = z_{ij}/z_{\alpha\beta} = I_{ij}/I_{\alpha\beta}$ is defined and nonzero.\endnote{$h(i)\wedge h(j)$ and $h(\alpha)\wedge h(\beta)$ are on the same nonterminal level $k$ so that $\Pi_{ij}/\Pi_{\alpha\beta} = z_{ij}/z_{\alpha\beta}$ is defined and nonzero on $U_\Gamma$ by \Cref{C:goodCOVgeneric}.} Conversely, in the notation of \Cref{L:determinedbyintegrals} for each $v\in \lambda^{-1}(k)$ and $n\in N_v\setminus \{n_v\}$, we can find $(i,j)\in N_{\rho_k|\rho_{k+1}}$ such that $\frac{1}{s_k}\cdot \int_{n_v}^n\omega_v = z_{ij}$. So, for each $v\in \lambda^{-1}(k)$ we have determined $[\int_{n_v}^n \omega_v: n\in N_v\setminus n_v]$.

    In the case of $k = \ell$, all $(i,j) \in N_{\rho_\ell|\top}$ have $h(i) = h(j)$, or equivalently $i,j$ lie in the same $b\in B(\rho_\ell)$, and $\Pi_{ij}(\Sigma)\in \bC$. Consequently, we recover $\{\int_{p_{\mu(b)}}^j\omega_b:j\in b\setminus \mu(b)\}_{b\in B(\rho_\ell)}.$ By \Cref{L:determinedbyintegrals} these data determine $\Sigma \in \Mmscbar_n$ uniquely and injectivity follows.
\end{proof}

\begin{prop}
\label{P:step1}
    There exists a morphism $f_1:\Mmscbar_n\to \Bl_{\le 1}\bP(V)$ such that 
    \begin{enumerate} 
        \item $\epsilon_1 \circ f_1 = f_0$
        \item for all $\rho$ of dimension $0$, one has $f_1^{-1}(D_\rho) = S_\rho$ and there is a commutative diagram 
        \[
        \begin{tikzcd}
            S_\rho \arrow[r,"f_\rho"]\arrow[dr,dashed]&D_\rho\arrow[d]\\
            &\bP(N_{\rho|\top})
        \end{tikzcd}
        \]
        where $f_\rho$ is the restriction of $f_1$ to $S_\rho$ and dashed arrow is $\Sigma \mapsto \Pi_{\rho|\top}^\dagger(\Sigma)$ where 
        \[
            \Pi_{\rho|\top}^\dagger(\Sigma) = 
            \begin{cases}
                [\Pi_{\rho|\top}^-(\Sigma):0]& \Sigma \in \partial S_\rho\\
                [\Pi_{\rho|\top}^-(\Sigma):1]& \Sigma \in S_\rho^\circ.
            \end{cases}
        \]
        Furthermore, 
        \begin{enumerate}
            \item[(a)] $f_\rho$ restricts to an isomorphism $S_\rho^\circ \to D_\rho^\circ$
            \item[(b)] $D_\rho$ is mapped isomorphically onto $\bP(N_{\rho|\top})$ 
            \item[(c)] for all $\pi > \rho$, $\widetilde{H}_\pi \cap D_\rho$ corresponds under this isomorphism to $Z(N_{\pi|\top}) \subset \bP(N_{\rho|\top})$
        \end{enumerate}
        \item for all $\pi$ with $\dim \pi = 1$, one has $f_1^{-1}(\widetilde{H}_\pi) = S_\pi$. 
    \end{enumerate}
\end{prop}

\begin{proof}
    (1) Consider $\epsilon_1:\Bl_{\le 1} \bP(V)\to \bP(V)$ and $\rho \in L_n^-$ of dimension $0$. $H_\rho \subset \bP(V)$ is a point by \Cref{L:propertiesofHs} and by \Cref{L:preimageofsubspace}, $f_0^{-1}(H_\rho) = S_\rho$. Each $S_\rho$ is a smooth divisor in $\Mmscbar_n$, so by \Cref{L:univpropblowup} there is an induced morphism $f_1:\Mmscbar_n \to \Bl_{\le 1} \bP(V)$ such that $\epsilon_1\circ f_1 = f_0$. 
    
    (2) By definition, $D_\rho = \epsilon_1^{-1}(H_\rho)$. Since $P(N_{\rho|\top})$ is a minimal set of linear equations defining $H_\rho$ by \Cref{L:cuttingout}, one has $D_\rho = \bP(N_{\rho|\top})$ by \Cref{L:linearstricttransform}. Restricting $f_1$ gives $f_\rho:S_\rho \to D_\rho = \bP(N_{\rho|\top})$, whence (b). Given $i\not\sim^\rho j$, a minimal set of equations for $H_\rho$ is $P_{\rho|\top}\cdot P_{ij}^{-1} = 0$ on $D(P_{ij})$. This pulls back along $f_0$ to $\Pi_{\rho|\top}\cdot \Pi_{ij}^{-1} = 0$ in $D(\Pi_{ij})\supset S_\rho^\circ$. 
    
    So, by \Cref{L:blowupmapexplicit}, $f_\rho(\Sigma) = [d(\Pi_{\rho|\top}\cdot \Pi_{ij}^{-1})_\Sigma]$. Since $d\Pi_{ij}^{-1}$ does not vanish on $S_{\rho}^\circ$ by \Cref{L:nonvanishingdifferential}, \Cref{L:blowupmapexplicit} implies $f_1|_{S_\rho^\circ}(\Sigma) = [\Pi_{\pi|\top}^-(\Sigma):1]$. By \Cref{L:injectivemap}, this is an isomorphism and (a) follows. For $\Sigma \in \partial S_\pi$, $\mathfrak{c}(\Sigma) = \{\pi<\eta<\cdots\}$ and so on $U_{\Gamma(\Sigma)}$ one sees $\Pi_{ij}^{-1} = t_2\cdots t_\ell/z_{ij}$ and by \Cref{C:goodCOVgeneric} it follows that $(d\Pi_{ij}^{-1})_\Sigma = 0$. As a consequence, $\partial S_\pi\to Z(T) \subset \bP(N_{\pi|\top})$ via the formula $\Sigma \mapsto [\Pi_{\pi|\top}^-(\Sigma):0]$ by \Cref{L:blowupmapexplicit}. Next, for $\pi > \rho$, by \Cref{L:linearstricttransform} one has $\widetilde{H}_\pi \cap D_\rho = Z(N_{\pi|\top})\subset \bP(N_{\rho|\top})$, yielding (c). 
    
    (3) Suppose $\dim \pi = 1$. By \Cref{L:stratificationofH}, $H_\pi = H_\pi^\circ \sqcup\bigsqcup_{\rho < \pi} H_\rho^\circ$ and so $\widetilde{H}_\pi = \epsilon_1^{-1}(H^\circ_\pi)\cup \bigcup_{\rho<\pi} \widetilde{H}_\pi \cap D_\rho$. By \Cref{L:quasiclosedpreimage}, $S_\pi^\circ \subset f_0^{-1}(H_\pi^\circ) = f_1^{-1}(\epsilon_1^{-1}(H_\pi^\circ)) \subset S_\pi$. The explicit description of $f_\rho$ implies\endnote{$\widetilde{H}_\pi$ is given by $Z(N_{\pi|\top}) \subset \bP(N_{\pi|\top})$. In particular, this is contained in $Z(T)$ and so $f_1^{-1}(\widetilde{H}_\pi \cap D_\rho) = f_\rho^{-1}(\widetilde{H}_\pi \cap D_\rho)$ is contained in $\partial S_\rho$. Pulling back the other defining conditions of $Z(N_{\pi|\top})$ gives vanishing of $\Pi_{\alpha\beta}/\Pi_{ij}$, noting that $\Pi_{ij}$ does not vanish on $\widetilde{H}_\pi \cap D_\rho$.}
    \[
        f_1^{-1}(\widetilde{H}_\pi \cap D_\rho) = 
        \left\{\Sigma \in \partial S_\rho \:\bigg|\:\frac{\Pi_{\alpha\beta}}{\Pi_{ij}}(\Sigma) = 0\text{ for all } (\alpha,\beta) \in N_{\pi|\top}, (i,j) \in N_{\rho|\pi}\right\}.
    \]
    Note that $N_{\rho|\pi} = \{(i,j)\}$ is a singleton, since $\dim \pi - \dim \rho =1$. We prove that $f_1^{-1}(\widetilde{H}_\pi \cap D_\rho)\subset S_\pi$. For this, consider $\Sigma \in f_1^{-1}(\widetilde{H}_\pi \cap D_\rho)$. As $\Sigma \in \partial S_\rho$ and $\dim \rho = 0$, $\mathfrak{c}(\Sigma) = \{\rho<\eta<\cdots\}$. By hypothesis, $i\sim^\rho j$, so $p_i\wedge p_j$ in $\Gamma(\Sigma)$ is on the level of $\rho$ or above. Also, as $\rho<\pi$, if $\alpha \sim^\pi \beta$ then $p_\alpha \wedge p_\beta$ is at or above the level of $\rho$. Since $(\Pi_{\alpha \beta}/\Pi_{ij})(\Sigma) = 0$, $p_\alpha \wedge p_\beta$ is at a strictly higher level of the tree than $\rho$.\endnote{By \Cref{C:goodCOVgeneric}, $\Pi_{\alpha\beta} = z_{\alpha\beta}/t_{p+1}\cdots t_\ell$ for $h(\alpha) \wedge h(\beta)$ on level $p$ and $\Pi_{ij} = z_{ij}/t_{p'+1}\cdots t_\ell$ for $h(i)\wedge h(j)$ on level $p'$. Then $\Pi_{\alpha\beta}/\Pi_{ij} = t_{p'+1}\cdots t_{p+1}\cdot (z_{\alpha\beta}/z_{ij}) = 0$, which implies that $p'<p$.} In particular, for all $\alpha \sim^\pi \beta$ we see $p_\alpha \wedge p_\beta$ is above $p_i\wedge p_j$. Consequently, $\alpha \sim^\eta \beta$ and thus $\eta \le \pi$. However, as $\dim \pi = 1$, it follows that $\eta = \pi$. Therefore, $\Sigma \in S_\pi$.
    
    Finally, since $f_1^{-1}(\widetilde{H}_\pi \cap D_\rho) \subset S_\pi$, $f_1^{-1}(\widetilde{H}_\pi)$ is a closed codimension zero subspace of $S_\pi$, whence $f_1^{-1}(\widetilde{H}_\pi) = S_\pi$ by irreducibility of $S_\pi$.\endnote{For instance, this holds because $S_\pi^\circ$ is an irreducible open dense subspace given by $Z(t)$ in the coordinates on $U_\Gamma$.} This gives (3).
\end{proof}

\begin{hyp}
\label{H:hypothesis}
For $1\le k \le n-2$, consider $P(k):$ 
\begin{enumerate}
        \item[$(1)_k$] There exist morphisms $f_j:\Mmscbar_n \to \Bl_{\le j} \bP(V)$ for all $0\le j\le k$ such that $\epsilon_j \circ f_j = f_{j-1}$.
        \item[$(2)_k$] For all $\rho \in (L_n^-)_{\le k}$, $f_k^{-1}(D_{\rho}) = S_{\rho}$ and as a consequence for all $\rho_\bullet \in \mathfrak{Ch}(L_n^-)_{\le k}$, $f_k^{-1}(D_{\rho_\bullet}) = S_{\rho_\bullet}$ and $f_k^{-1}(D_{\rho_\bullet}^k) = S_{\rho_\bullet}^k$. Write $f_{\rho_\bullet}^k$ for the restriction of $f_k$ to $S_{\rho_\bullet}^k$. There is a commutative diagram:
        \begin{equation}
        \label{E:inductivediagram}
            \begin{tikzcd}
                S_{\rho_\bullet}^k\arrow[r,"f_{\rho_\bullet}^k"]\arrow[dr,dashed]& D_{\rho_\bullet}^k \arrow[d,"\text{open}",hook]\\
                &\prod_{i=0}^\ell \bP(N_{\rho_i|\rho_{i+1}})
            \end{tikzcd}
        \end{equation}
        such that the dashed arrow is $\Sigma \mapsto ([\Pi_{\rho_i|\rho_{i+1}}(\Sigma)]_{i=0}^{\ell-1},\Pi_{\rho_\ell|\top}^\dagger(\Sigma))$ where 
        \[
        \Pi_{\rho_\ell|\top}^\dagger(\Sigma) = 
        \begin{cases}
            [\Pi_{\rho_\ell|\top}^-(\Sigma):0]&\Sigma \in \partial S_{\rho_\bullet}^k\\
            [\Pi_{\rho_\ell|\top}^-(\Sigma):1]& \Sigma \in S_{\rho_\bullet}^\circ.
        \end{cases}.
        \] Furthermore, 
        \begin{enumerate}
            \item[$(a)_k$] $f_{\rho_\bullet}^k$ restricts to an isomorphism $S_{\rho_\bullet}^\circ \to D_{\rho_\bullet}^\circ$; 
            \item[$(b)_k$] $D_{\rho_\bullet}^k$ is mapped isomorphically onto $\prod_{i=0}^\ell U_i$ where $U_i$ is the complement of a linear subspace arrangement in $\bP(N_{\rho_i|\rho_{i+1}})$ for each $i$; and 
            \item[$(c)_k$] for any $\pi>\max(\rho_\bullet)$, $\widetilde{H}_\pi\cap D_{\rho_\bullet}^k$ corresponds to $(\prod_{i=0}^{\ell-1}U_i) \times Z(N_{\pi|\top} \cap U_{\ell}) \subset \prod_{i=0}^\ell \bP(N_{\rho_i|\rho_{i+1}})$.
        \end{enumerate}
        \item[$(3)_k$] For all $\pi$ of dimension $k+1$, one has $f_{k}^{-1}(\widetilde{H}_\pi) = S_\pi$. 
    \end{enumerate}
\end{hyp}

By \Cref{P:step1}, $P(1)$ holds. We will show $P(k) \Rightarrow P(k+1)$. 

\begin{lem}
\label{L:(1)lemma}
    If $P(k)$ holds then $(1)_{k+1}$ holds, in particular there is a map $f_{k+1}:\Mmscbar_n \to \Bl_{\le k+1}\bP(V)$ such that $\epsilon_{k+1}\circ f_{k+1} = f_k$. 
\end{lem}

\begin{proof}
    $(1)_{k+1}$ follows from the definition of $\epsilon_{k+1}$, $(3)_k$, and \Cref{L:univpropblowup}.\endnote{$\epsilon_{k+1}:\Bl_{\le k+1}\bP(V)\to \Bl_{\le k}\bP(V)$ is constructed by blowing up the strict transforms in $\Bl_{\le k}\bP(V)$ of all $H_\pi$ of dimension $k+1$. So, to construct the morphism by \Cref{L:univpropblowup} we need to know that $f_{k}^{-1}(\widetilde{H}_\pi)$ is a smooth divisor in $\Mmscbar_n$ for each such $\pi$. However, this follows from $(3)_k$.}
\end{proof}

\begin{lem}
\label{L:preimageofD}
    If $P(k)$ holds then for all $\pi \in (L_n^-)_{\le k+1}$, one has $f_{k+1}^{-1}(D_\pi) = S_\pi$. Furthermore, for all $\rho_\bullet \in \mathfrak{Ch}(L_n^-)_{\le k+1}$, one has $f_{k+1}^{-1}(D_{\rho_\bullet}) = S_{\rho_\bullet}$ and $f_{k+1}^{-1}(D_{\rho_\bullet}^{k+1}) = S_{\rho_\bullet}^{k+1}$. 
\end{lem}

\begin{proof}
    By $(3)_{k}$, for all $\pi$ of rank $k+1$, $f_k^{-1}(\widetilde{H}_\pi) = S_\pi$ and since $\epsilon_{k+1}^{-1}(\widetilde{H}_\pi) = D_\pi$, one has $f_{k+1}^{-1}(D_\pi) = S_\pi$.\endnote{By $(1)_{k+1}$, we have $\epsilon_{k+1}\circ f_{k+1} = f_k$ so $S_\pi = f_{k+1}^{-1}(\epsilon_{k+1}^{-1}(\widetilde{H}_\pi)) = f_{k}^{-1}(D_\pi)$.} Suppose next that $\pi$ is of rank $\le k$. If $f_{k+1}(\Sigma) \in \widetilde{D}_\pi$, then $f_k(\Sigma) = \epsilon_{k+1}(f_{k+1}(\Sigma)) \in \epsilon_{k+1}(\widetilde{D}_\pi) = D_\pi$. By $(2)_k$, $\Sigma \in S_\pi$ and so $f_{k+1}^{-1}(D_\pi) \subset S_\pi$. Because $\widetilde{D}_\pi \to D_\pi$ is an isomorphism over an open set of $D_\pi$, it follows that $f_{k+1}^{-1}(\widetilde{D}_\pi)$ contains an open subset of $S_\pi$ and consequently equals $S_\pi$. For $\rho_\bullet \in \mathfrak{Ch}(L_n^-)_{\le k+1}$ since $D_{\rho_\bullet} = \bigcap_{i=1}^\ell D_{\rho_i}$, one has $f_{k+1}^{-1}(D_{\rho_\bullet}) = S_{\rho_\bullet}$. Similarly, $D_{\rho_\bullet}^{k+1} = D_{\rho_\bullet} \setminus (\bigcup_{\dim \lambda \le k+1, \lambda \not\in \rho_\bullet} D_\lambda)$ now implies that $f_{k+1}^{-1}(D_{\rho_\bullet}^{k+1}) = S_{\rho_\bullet}^{k+1}$. 
\end{proof}

\begin{lem}
\label{L:P(k)mapdescription}
    If $P(k)$ holds then for all $\rho_\bullet = \{\rho_1<\cdots<\rho_\ell\} \in \mathfrak{Ch}(L_n^-)_{\le k+1}$ there is a commutative diagram:
    \begin{equation}
    \label{E:P(k)mapdescription}
        \begin{tikzcd}
            S_{\rho_\bullet}^{k+1}\arrow[r,"f_{\rho_\bullet}^{k+1}"]\arrow[dr,dashed]& D_{\rho_\bullet}^{k+1} \arrow[d,"\mathrm{open}",hook]\\
            &\prod_{i=0}^\ell \bP(N_{\rho_i|\rho_{i+1}})
        \end{tikzcd}
    \end{equation}
    such that $f_{\rho_\bullet}^{k+1}$ is the restriction of $f_{k+1}$ to $S_{\rho_\bullet}^{k+1}$ and the dashed arrow is given by $\Sigma \mapsto ([\Pi_{\rho_i|\rho_{i+1}}(\Sigma)]_{i=0}^{\ell-1},\Pi_{\rho_\ell|\top}^\dagger(\Sigma))$ where $\Pi_{\rho_\ell|\top}^\dagger(\Sigma)$ is defined as in \Cref{H:hypothesis}. Furthermore, $(a)_{k+1},(b)_{k+1}$, and $(c)_{k+1}$ hold.
\end{lem}

\begin{proof}
    We break the proof into a series of steps. By \Cref{L:preimageofD}, $f_{k+1}$ restricts to a map $f_{\rho_\bullet}^{k+1}:S_{\rho_\bullet}^{k+1}\to D_{\rho_\bullet}^{k+1}$. 

    \textbf{Step 1:} The result holds when $\dim \rho_\ell \le k$. 
    
    Consider the following diagram:
    
    \[
        \begin{tikzcd}
            \Bl_{\le k+1} \bP(V) \arrow[r,"\epsilon_{k+1}"]&\Bl_{\le k} \bP(V)&\\
            D_{\rho_\bullet}^{k+1}\arrow[r,hook,"\text{open}"]\arrow[u,hook]&D_{\rho_\bullet}^k \arrow[u,hook]\arrow[r,hook,"\text{open}"] & \prod_{i=0}^\ell \bP(N_{\rho_i|\rho_{i+1}})\\
            S_{\rho_\bullet}^{k+1}\arrow[u,"f_{k+1}"]\arrow[r,hook,"\text{open}"]&S_{\rho_\bullet}^k.\arrow[u,"f_{k}"]\arrow[ur]&
        \end{tikzcd}
    \]
    The ``open'' arrows are open immersions, the middle one being the restriction of $\epsilon_{k+1}$ away from the exceptional locus, the middle right by applying $(b)_k$, and the bottom by $S_{\rho_\bullet}^{k+1}\subset S_{\rho_\bullet}^k$ being an open subspace. The description of the composite map $S_{\rho_\bullet}^{k+1}\to \prod_{i=0}^\ell \bP(N_{\rho_i|\rho_{i+1}})$, $(a)_{k+1}$, and $(c)_{k+1}$ now follow from $(a)_k$ and $(c)_k$ respectively. $(b)_{k+1}$ also follows, except that the $U_\ell$ obtained by induction must be shrunk by removing $U_\ell \cap Z(N_{\pi|\top})$ for all $\pi>\rho_\ell$ of dimension $k+1$. 

    \textbf{Step 2:} Assume now $\dim \rho_\ell = k+1$. We show that $(b)_{k+1}$ and $(c)_{k+1}$ hold.
    
    By \cite{Ulyanovpoly}*{Lem. 1}, any pair from $\{D_\pi:\rk(\pi)\le k\}\cup \{\widetilde{H}_\pi:\rk(\pi)>k\}$ intersects cleanly in $\Bl_{\le k}\bP(V)$. Thus, applying \Cref{L:strictransform}, $D_{\rho_\bullet}$ in $\Bl_{\le k+1} \bP(V)$ is the strict transform of $D_{\rho_1,\ldots, \rho_{\ell-1}}\subset \Bl_{\le k} \bP(V)$ intersected with $D_{\rho_\ell}$ which is the exceptional divisor associated to $H_{\rho_\ell}$. Write $\widetilde{H}_{\rho_\ell}$ for the strict transform of $H_{\rho_\ell}$ in $\Bl_{\le k}\bP(V)$, put $Z_{\rho_\bullet} := D_{\rho_1,\ldots, \rho_{\ell-1}}^k \cap \widetilde{H}_{\rho_\ell} \hookrightarrow D_{\rho_1,\ldots, \rho_{\ell-1}}^k$, and consider the diagram:
    \[
    \begin{tikzcd}
        &\widetilde{D}_{\rho_1,\ldots, \rho_{\ell-1}}^k \cap D_{\rho_\ell} \arrow[d]\arrow[r,"\sim",dotted]&E_\iota \arrow[d] \\
        S_{\rho_1,\ldots, \rho_{\ell-1}}^k \cap S_{\rho_\ell} \arrow[r]\arrow[ur,dashed]\arrow[d,hook]\arrow[dr,phantom,"\square"]&Z_{\rho_\bullet}\arrow[r,"\sim"] \arrow[d,hook]&(\prod_{i=0}^{\ell-2} U_i) \times Z(N_{\rho_\ell|\top}) \cap U_{\ell-1}\arrow[d,hook,"\iota"]\\
        S_{\rho_1,\ldots, \rho_{\ell-1}}^k \arrow[r]&D_{\rho_1,\ldots, \rho_{\ell-1}}^k\arrow[r,"\sim"]&\prod_{i=0}^{\ell-1} U_i =: U_{\le \ell-1}.
    \end{tikzcd}
    \]
    
    The isomorphism $D_{\rho_1,\ldots, \rho_{\ell-1}}^k \to U_{\le \ell-1}$ is by $(b)_k$. Commutativity of the bottom right square follows from $(c)_k$. $\widetilde{D}_{\rho_1,\ldots, \rho_{\ell-1}}^k \cap D_{\rho_\ell}$ is identified with the exceptional divisor of $\Bl_{Z_{\rho_\bullet}}(D_{\rho_1,\ldots, \rho_{\ell-1}}^k)$ by \Cref{L:propertransformblowup}. The bottom right square  induces an isomorphism from $\widetilde{D}_{\rho_1,\ldots, \rho_{\ell-1}}^k \cap D_{\rho_\ell}$ to the blowup of $U_{\le \ell-1}$ along $\iota$. This isomorphism identifies the exceptional divisors as indicated by the top dotted arrow, where $E_\iota$ denotes the exceptional divisor of $\Bl_\iota(U_{\le \ell-1})$.

    By $(c)_k$, $\Bl_\iota (U_{\le \ell-1})$ is the blowup of $U_{\le \ell-1}$ along $U_{\le \ell-1}'$, where $U_{\le \ell-1}' := U_{\le \ell-2}\times U_{\ell-1}'$ and $U_{\ell-1}' := Z(N_{\rho_\ell|\top})\cap U_{\ell-1}$. By \Cref{L:linearstricttransform}, $E_\iota$ is identified with $U_{\le \ell-1}' \times \bP(N_{\rho_\ell|\top})$. 
    
    By $(c)_k$ and \Cref{L:linearstricttransform}, for all $\pi > \rho_\ell$, $\widetilde{H}_\pi$ intersects $\widetilde{D}^k_{\rho_1,\ldots, \rho_{\ell-1}}\cap D_{\rho_\ell}$ in the locus corresponding to $U_{\le \ell-1}'\times Z(N_{\pi|\top})$, whence $(c)_{k+1}$. 
    
    Now, remove $\widetilde{H}_\pi$ for all $\dim \pi = k+1$ to form $U_\ell'\subset \bP(N_{\rho_\ell|\top})$, and $D_{\rho_\bullet}^{k+1} \cong U_{\le \ell-1}'\times U_\ell'$ via the restricted map $D_{\rho_\bullet}^{k+1} \to U_{\le \ell-1}' \times \bP(N_{\rho_\ell|\top})$, whence $(b)_{k+1}$.

    \textbf{Step 3:} Constructing the induced map $S_{\rho_1,\ldots, \rho_{\ell-1}}^k \cap S_{\rho_\ell}\to \widetilde{D}_{\rho_1,\ldots, \rho_{k-1}}^k \cap D_{\rho_k}$.

    By \Cref{L:univpropblowup}, it suffices to show that the bottom left diagram is Cartesian. The composite map $\gamma:S_{\rho_1,\ldots, \rho_{\ell-1}}^k\to U_{\le \ell-1}$ is $\gamma(\Sigma) = ([\Pi_{\rho_i|\rho_{i+1}}(\Sigma)]_{i=0}^{\ell-2},\Pi_{\rho_{\ell-1}|\top}^\dagger(\Sigma))$ by $(2)_k$. 
    Suppose $\Sigma\in \gamma^{-1}(U_{\le \ell-1}')$. As $U_{\ell-1}'\subset Z(T)$, $\gamma(\Sigma) = ([\Pi_{\rho_i|\rho_{i+1}}(\Sigma)]_{i=0}^{\ell-2},[\Pi_{\rho_{\ell-1}|\top}^-(\Sigma):0])$. 
    
    By $(a)_k$, $\Sigma \in \partial S_{\rho_1,\ldots, \rho_{\ell-1}}^k$ and there exists $(i,j)\in N_{\rho_{\ell-1}|\top}$ such that $\Pi_{ij}(\Sigma)\ne 0$ and $(\Pi_{\alpha\beta}/\Pi_{ij})(\Sigma) = 0$ for all $(\alpha,\beta)\in N_{\rho_\ell|\top}$. Write $\mathfrak{c}(\Sigma) = \{\rho_1<\cdots<\rho_{\ell-1}<\eta<\cdots\}$. By definition, $i\sim^{\rho_{\ell-1}} j$ so $p_i\wedge p_j$ is on the level of $\rho_{\ell-1}$ or above. Since $\rho_{\ell-1}<\rho_\ell$, if $\alpha \sim^{\rho_\ell}\beta$, then $p_\alpha \wedge p_\beta$ is on or above the level of $\rho_{\ell-1}$. $(\Pi_{\alpha\beta}/\Pi_{ij})(\Sigma) = 0$, so $p_\alpha \wedge p_\beta$ is at a strictly higher level of $\Gamma(\Sigma)$ than $\rho_{\ell-1}$. Thus, for all $\alpha \sim^{\rho_\ell}\beta$, $p_\alpha \wedge p_\beta$ is above $p_i\wedge p_j$. Therefore, $\alpha\sim^{\rho_\ell}\beta$ and it follows that $\eta \le \rho_\ell$. If $\eta = \rho_\ell$, then $\Sigma \in S_{\rho_\ell}$ and we are done. If $\eta < \rho_\ell$, we proceed, noting that $\dim \eta <\dim \rho_\ell = k+1$, and consequently $\Sigma \in S^{\dim \eta}_{\rho_1<\cdots<\rho_{\ell-1}<\eta}$. So, by \Cref{L:preimageofD}, $\Sigma \in f_{k+1}^{-1}(D_{\rho_1<\cdots<\rho_{\ell-1}<\eta}^{\dim \eta}\cap \widetilde{H}_{\rho_\ell})$. Repeating the same argument using the fact that $\dim \rho_\ell - \dim \eta$ decreases at each iteration yields the result. It follows that $\gamma^{-1}(U_{\le \ell-1}') = S_{\rho_1,\ldots, \rho_{\ell-1}}^k \cap S_{\rho_\ell}$.

    \textbf{Step 4:} By \Cref{L:strictrestrict}, $S_{\rho_1,\ldots, \rho_{\ell-1}}^k \cap S_{\rho_\ell} \to \widetilde{D}_{\rho_1,\ldots, \rho_{k-1}}^k \cap D_{\rho_k}$ is the restriction of $f_{k+1}$. We have thus constructed maps $S_{\rho_1,\ldots, \rho_{\ell-1}}^{k}\cap S_{\rho_\ell}\to D_{\rho_\bullet}^{k+1} \hookrightarrow \prod_{i=0}^\ell \bP(N_{\rho_i|\rho_{i+1}})$, where the second arrow is an open immersion. Denote the restricted map $r:S_{\rho_\bullet}^{k+1}\to \prod_{i=0}^\ell \bP(N_{\rho_i|\rho_{i+1}})$

    \textbf{Step 5:} $r(\Sigma) =  ([\Pi_{\rho_i|\rho_{i+1}}(\Sigma)]_{i=0}^{\ell-1},\Pi_{\rho_\ell|\top}^\dagger(\Sigma))$.

    $\gamma:S_{\rho_1,\ldots, \rho_{\ell-1}}^k \cap S_{\rho_\ell} \to U_{\le \ell-1}'$ is $\gamma(\Sigma) =  ([\Pi_{\rho_i|\rho_{i+1}}(\Sigma)]_{i=0}^{\ell-2},[\Pi_{\rho_{\ell-1}|\top}^-(\Sigma):0])$. $U_{\le \ell-1}'$ is covered by $\{V_{ij} = D(s_{ij})\cap U_{\le \ell-1}':(i,j)\in N_{\rho_{\ell-1}|\rho_\ell}\}$. We claim 
    
    \begin{equation}
    \label{E:restrictionexpl}
        r(\Sigma) = 
        \begin{cases}
            [\Pi_{\rho_i|\rho_{i+1}}(\Sigma)]_{i=0}^{\ell-1},[\Pi_{\rho_\ell|\top}^-(\Sigma):0]& \Sigma \in \partial S_{\rho_\bullet}^{k+1}\cap \gamma^{-1}(V_{ij})\\
            [\Pi_{\rho_i|\rho_{i+1}}(\Sigma)]_{i=0}^{\ell-1},[\Pi_{\rho_\ell|\top}^-(\Sigma):1]& \Sigma \in S_{\rho_\bullet}^\circ \cap \gamma^{-1}(V_{ij}).
        \end{cases}
    \end{equation}
    One verifies that $\gamma^*(T/s_{ij}) = \Pi_{ij}^{-1}$ and so $Z(N_{\rho_\ell|\top}) \subset U_{\le \ell-1}$ pulls back under $\gamma$ to $\Pi_{\rho_\ell|\top}\cdot \Pi_{ij}^{-1} = 0$.\endnote{By $(2)_k$, when restricted to $S_{\rho_1,\ldots, \rho_{\ell-1}}^\circ\subset S_{\rho_1,\ldots, \rho_{\ell-1}}^k$, which is open and dense, we have 
    \[
    \Sigma \mapsto ([\Pi_{\rho_i|\rho_{i+1}}(\Sigma)]_{i=0}^{\ell-2},[\Pi_{\rho_{\ell-1}|\top}^-(\Sigma):1]).
    \]
    In particular, for all $(i,j)\in N_{\rho_{\ell-1}|\top}$ one has $\gamma^*(s_{ij}/T) = \Pi_{ij}$ and so $\gamma^*(T/s_{ij}) = \Pi_{ij}^{-1}$. The result now follows.} Now, a calculation using \Cref{C:goodCOVgeneric} implies $Z(d(\Pi_{ij}^{-1})) \cap \gamma^{-1}(V_{ij}) = \partial S_{\rho_\bullet}^{k+1}\cap \gamma^{-1}(V_{ij})$.\endnote{Consider $\Sigma \in S_{\rho_\bullet}^{k+1}$. By definition, $\mathfrak{c}(\Sigma)$ is of the form $\{\rho_1<\cdots<\rho_\ell<\pi_1<\cdots<\pi_p\}$. $h(i)\wedge h(j)$ is on level $\ell-1$ because $(i,j)\in N_{\rho_{\ell-1}|\rho_\ell}$. On $U_\Gamma$, we can thus write $\Pi_{ij}^{-1} = t_\ell \cdots t_{\ell+p}/z_{ij}$ by \Cref{C:goodCOVgeneric}. Now, if $p\ge 1$ (which is equivalent to $\Sigma\in \partial S_{\rho_\bullet}^{k+1}$), one has $d(\Pi_{ij})^{-1} = d(t_\ell/z_{ij})\cdot (t_{\ell+1}\cdots t_{\ell+p}) + (t_\ell/z_{ij})\cdot d(t_{\ell_1}\cdots t_{\ell+p}) = 0$. If $p=0$, one has $d(t_\ell/z_{ij}) = \frac{1}{z_{ij}}d(t_\ell)$ which is nonzero on $U_\Gamma$.} The claim now follows from \Cref{L:blowupmapexplicit}.
    
    \textbf{Step 6:} We verify $(a)_{k+1}$. 
    
    $D_{\rho_\bullet}^\circ$ corresponds to $D(T) \subset  U_{\le \ell-1}' \times \bP(N_{\rho_\ell|\top})$\endnote{By construction, $D_{\rho_{\bullet}}$ does not intersect any $D_\pi$ for $\dim\pi \le k+1$ along $U_{\le \ell-1}'\times \bP(N_{\rho_\ell|\top})$. It only remains to remove the strict transforms of those $\pi$ with $\dim \pi \ge k+2$. However, by $(c)_{k+1}$ we can see that these are exactly in $Z(T)$; indeed, the strict transform of $H_\top$ intersects $D_{\rho_\bullet}$ along $Z(T)$ and the other strict transforms all satisfy $T = 0$.} and by \eqref{E:restrictionexpl} it follows that $(f_{\rho_\bullet}^{k+1})^{-1}(D_{\rho_\bullet}^\circ) = S_{\rho_\bullet}^\circ$. The isomorphism claim is now immediate by \eqref{E:restrictionexpl} and \Cref{L:injectivemap}.
\end{proof}

\begin{lem}
\label{L:(3)lemma}
    For all $\pi$ of dimension $k+2$, $f_{k+1}^{-1}(\widetilde{H}_\pi) = S_\pi$.
\end{lem}

\begin{proof}
    $\widetilde{H}_\pi = \epsilon_{\le k+1}(H_\pi^\circ) \sqcup (\bigsqcup_{\rho_\bullet \in \mathfrak{Ch}_{\le k+1}} D_{\rho_\bullet}^{k+1}\cap \widetilde{H}_\pi)$ by \Cref{L:stratificationofBlk}.\endnote{By \Cref{L:stratificationofBlk}, we know $\Bl_{\le k+1} \bP(V) = U\sqcup \bigsqcup_{\rho_\bullet \in \mathfrak{Ch}(L_n^-)_{\le k+1}} D_{\rho_\bullet}^{k+1}$ and so $\widetilde{H}_\pi = U\cap \widetilde{H}_\pi \sqcup \bigsqcup_{\rho_\bullet} D_{\rho_\bullet}^{k+1} \cap \widetilde{H}_\pi$. $U$ is the preimage under $\epsilon_{\le k+1}$ of the complement of $\bigcup_{\dim \pi \le k+1} H_\pi$. In particular, $U\cap \widetilde{H}_\pi = \epsilon_{\le k+1}^{-1}(H_\pi^{k+1})$.} Thus, $f_{k+1}^{-1}(\widetilde{H}_\pi) = f_0^{-1}(H_\pi^\circ) \sqcup \bigsqcup_{\rho_\bullet \in \mathfrak{Ch}_{\le k+1}} f_{k+1}^{-1}(D_{\rho_\bullet}^{k+1}\cap \widetilde{H}_\pi)$. By \Cref{L:preimageofsubspace}, we have $f_0^{-1}(H_\pi^\circ) = S_\pi^{k+2}$. In particular, $f_0^{-1}(H_\pi^\circ)$ contains $S_\pi^\circ$ so $f_{k+1}^{-1}(\widetilde{H}_\pi) \supset S_\pi$ by closedness. It remains to show that for all $\rho_\bullet \in \mathfrak{Ch}(L_n^-)_{\le k+1}$ one has $f_{k+1}^{-1}(D_{\rho_\bullet}^{k+1}\cap \widetilde{H}_\pi) \subset S_\pi$. For this, apply the argument in the penultimate paragraph of the proof of \Cref{L:P(k)mapdescription}, replacing the chain $\rho_1<\ldots<\rho_{\ell-1}$ by $\rho_1<\cdots<\rho_\ell$, $k$ by $k+1$, $\rho_\ell-1$ by $\rho_\ell$, and $\rho_\ell$ by $\pi$.\endnote{Consider $\Sigma \in f_{k+1}^{-1}(D_{\rho_\bullet}^{k+1}\cap \widetilde{H}_\pi)$. Using the identifications of \Cref{L:P(k)mapdescription}, $f_{k+1}(\Sigma) \in V(N_{\pi|\top})$ and $\Sigma \in \partial S_{\rho_\bullet}^{k+1}$. So, there exists $(i,j)\in N_{\rho_\ell|\top}$ such that $\Pi_{ij}(\Sigma)\ne 0$ and $(\Pi_{\alpha\beta}/\Pi_{ij})(\Sigma) = 0$ for all $(\alpha,\beta)\in N_{\pi|\top}$.\endnote{We show that there exists $(i,j)\in N_{\rho_\ell|\top}$ such that $\Pi_{ij}(\Sigma) \ne 0$. Note that the $s_{ij}$ of \Cref{N:mainnotation}(5) give homogeneous coordinates on $\bP(N_{\rho_\ell|\top})$ and that $f_{k+1}(\Sigma) \in V(N_{\rho_\ell|\top})$. Therefore, $T(\Sigma) = 0$, but some $s_{ij}(\Sigma)\ne 0$. In particular, the coordinate $T/s_{ij}$ pulls back to $\Pi_{ij}^{-1}$ and we see that $\Pi_{ij}^{-1}(\Sigma) = 0$, whence $\Pi_{ij}(\Sigma) \ne 0$.} Consider $\mathfrak{c}(\Sigma) = \{\rho_1<\cdots<\rho_\ell<\eta<\cdots\}$. By hypothesis, $i\sim^{\rho_\ell} j$, so $p_i\wedge p_j$ is on the level of $\rho_\ell$ or above. Since $\rho_\ell <\pi$, if $\alpha \sim^\pi \beta$, then $p_\alpha \wedge p_\beta$ is on or above the level of $\rho_\ell$. Since $(\Pi_{\alpha \beta}/\Pi_{ij})(\Sigma) = 0$, $p_\alpha \wedge p_\beta$ is at a strictly higher level of $\Gamma(\Sigma)$ than $\rho_\ell$. Thus, for all $\alpha \sim^\pi \beta$ we see $p_\alpha \wedge p_\beta$ is above the level of $p_i\wedge p_j$. Consequently, $\alpha \sim^\eta \beta$ and it follows that $\eta \le \pi$. If $\eta = \pi$, then $\Sigma \in S_\pi$. 
    
    If $\eta < \pi$, we proceed, noting that $\dim \eta < \dim\pi = k+2$ and consequently $\Sigma \in S_{\rho_1<\cdots<\rho_\ell<\eta}^{\dim\eta}$. So, by \Cref{L:preimageofD}, $\Sigma \in f_{k+1}^{-1}(D_{\rho_1<\cdots<\rho_\ell<\eta}^{\dim \eta}\cap \widetilde{H}_\pi)$. So, repeating the same argument noting that $\dim \pi - \dim \eta$ decreases at each iteration yields the result.}
\end{proof}

\begin{cor}
\label{C:induction}
    For all $1\le k \le n-2$, $P(k) \implies P(k+1)$. 
\end{cor}

\begin{proof}
    This is a combination of \Cref{L:(1)lemma}, \Cref{L:P(k)mapdescription}, and \Cref{L:(3)lemma}.
\end{proof}

\begin{proof}[Proof of \Cref{T:isothm}]
    Applying \Cref{C:induction}, $P(n-1)$ holds, so there exists a morphism $f:=f_{n-1}:\Mmscbar_n \to W_n$ over $\bP(V)$ such that for any $\rho_\bullet \in \mathfrak{Ch}(L_n^-)$, $f$ restricts to an isomorphism $S_{\rho_\bullet}^\circ \to D_{\rho_\bullet}^\circ$. As $f:\Mmscbar_n \to W_n$ is a birational morphism of proper and irreducible varieties, it is surjective. We only need to verify injectivity, but this follows from $\Mmscbar_n = \bigsqcup_{\rho_\bullet \in \mathfrak{Ch}(L_n^-)} S_{\rho_\bullet}^\circ$, injectivity of $f$ restricted to $S_{\rho_\bullet}$, and disjointness of the $D_{\rho_\bullet}^\circ\subset W_n$.
\end{proof}

\subsection{Consequences of the isomorphism}

\begin{cor}
\label{C:projectivity}
    $\Mmscbar_n$ is projective.
\end{cor}

\begin{proof}
    $W_n$ is projective, being a sequence of blowups of $\bb{P}(V)$.
\end{proof}

By \Cref{T:isothm} we can obtain a precise description of the generators and relations for the Chow ring $\mathrm{CH}^*(\Mmscbar_n)$ using the results of \cite{Huhetalsemismall}. The authors define the \emph{augmented Chow ring} of a matroid $M$ as follows \cite{Huhetalsemismall}*{p.4}. $S_M$ is the polynomial ring $\bQ[y_e:e\in E]\otimes \bQ[x_F:F\text{ is a proper flat of } M]$ and $\mathrm{CH}(M) = S_M/(I_M+J_M)$ where 
\begin{enumerate}   
    \item $I_M$ is the ideal generated by $y_e - \sum_{e\not \in F}x_F$ for each $e\in E$; and 
    \item $J_M$ is the ideal generated by 
    \begin{enumerate}
        \item[(a)] $x_{F_1}x_{F_2}$ for any pair of incomparable flats of $M$; and 
        \item[(b)] $y_ex_F$ for any $e\in E$ and any proper flat $F$ of $M$ such that $e\not\in F$.
    \end{enumerate}
\end{enumerate}
Note that by convention the empty set is a proper flat. 

When $M$ is realizable, one can construct a variety $X_M$ via an explicit blowup of a subspace arrangement of $\bP(V\oplus \bC)$, for some complex vector space $V$, such that $\mathrm{CH}(M) \cong \mathrm{CH}^*(X_M)$; here, the right hand side is the Chow ring of the algebraic variety $X_M$ and the isomorphism is as graded rings. $X_M$ is called the \emph{augmented wonderful variety} associated to the matroid $M$ \cite{Huhetalsemismall}*{pp. 16-17}.

The graphical matroid $M_n = M(K_n)$, as introduced in \Cref{E:matroidofKn} is our main example. The associated augmented wonderful variety is $W_n$ that we have been considering. The lattice of flats of $M_n$ is identified with $L_n$ as in \Cref{P:lattice}. The whole set $\{1,\ldots,n\}$ regarded as a flat corresponds to $\bot$. $E$ is the edge set of $K_n$, and can be identified with $\{(i,j):1\le i<j\le n\}$. As such, we write 
\[
    S_M = \bQ[y_{ij}:1\le i<j\le n]\otimes \bQ[x_\rho:\rho \in L_n^-].
\]

\begin{cor}
\label{C:chow}    
    $\mathrm{CH}^*(\Mmscbar_n) \cong \bQ[x_\rho:\rho \in L_n^-]/I$, where $I$ is the ideal of relations generated by
    \begin{enumerate}
        \item $x_\rho x_\pi$ for $\pi$ and $\rho$ not comparable; and 
        \item $x_\pi \cdot \sum_{\rho:i\not\sim^\rho j} x_\rho$ for all $1\le i<j\le n$ such that $i\not \sim^\pi j$. 
    \end{enumerate}
    Furthermore, under the isomorphism $x_\rho$ corresponds to the class of $S_\rho$ for all $\rho \in L_n^-$ and $1$ to the class of $S_\bot = \Mmscbar_n$.
\end{cor}

\begin{proof}
    We rewrite $\mathrm{CH}(M_n) = S_M/(I_M+J_M) \cong \frac{S_M/I_M}{(I_M+J_M)/I_M}$ to obtain $\mathrm{CH}(M_n)\cong \bQ[x_\rho:\rho \in L_n^-]/I$.\endnote{$S_M/I_M$ is generated by the (image of) $\{x_F\}$. The relations of $J_M$ become $x_{\rho}\cdot x_{\pi}$ for $\pi$ and $\rho$ not comparable. The second set of relations is can be rewritten using $y_{ij} = \sum_{\rho:i\not\sim^\rho j} x_\rho$ as in (2).} Under the isomorphism $\mathrm{CH}(M_n)\to \mathrm{CH}^*(W_n)$, $x_\rho$ maps to the class of the divisor $D_\rho$ and $1\mapsto \Mmscbar_n$ (see \cite{Huhetalsemismall}*{Rem. 2.13}) and the result follows from \Cref{T:isothm}. 
\end{proof}

\section{Multiscaled lines as a resolution of singularities}

\subsection{Comparison of moduli problems}

In this section, we connect $\overline{P}_n$ as defined in \cite{Zahariucmarked} to $\Mmscbar_n$. We first recall the definition of $\overline{P}_n$. 

\begin{thm}
\label{T:zahariucmarkedmain}
(\cite{Zahariucmarked}, Thm. 1.5) Let $F$ be the functor which associates to each Noetherian scheme $S$ the set of all collections of data as follows, modulo isomorphism:
\begin{itemize}
    \item a genus $0$ prestable curve $\pi:C\to S$
    \item smooth sections $x_1,\ldots, x_n,p_\infty:S\to C$ of $\pi$ (possibly not disjoint); and
    \item an $\cal{O}_C$-module homomorphism $\phi:\omega_{C/S}\to \cal{O}_C$ 
\end{itemize}
such that:
\begin{enumerate}
    \item $\phi$ factors through the inclusion $\cal{O}_C(-2p_\infty(S)) \to \cal{O}_C$;
    \item $x_i^*\phi:x_i^*\omega_{C/S}\to \cal{O}_S$ is an isomorphism for $i=1,\ldots, n$; and 
    \item the natural stability condition holds: for any geometric point $\overline{s}\to S$,
        \begin{enumerate}
            \item with the possible exception of the component which contains $p_{\infty,\overline{s}}$, no irreducible component of $C_{\overline{s}}$ intersects exactly two other components;
            \item any irreducible component of $C_{\overline{s}}$ which intersects exactly one other irreducible component contains at least one of the points $x_{1,\overline{s}},\ldots, x_{n,\overline{s}}$ but not the point $p_{\infty,\overline{s}}$.
        \end{enumerate}
\end{enumerate}
Then $F$ is represented by a projective locally complete intersection flat geometrically integral scheme over $\Spec \bZ$. 
\end{thm}

There exists a scheme, denoted $\overline{P}_n$, representing $F$ by \Cref{T:zahariucmarkedmain}. We will often work over $\Spec \bb{C}$ and write $\overline{P}_n(\bC)$ for the set of $\bC$-points of $\overline{P}_n$ regarded as a complex projective variety. We refer to these objects as \emph{scaled curves}. There is an open dense subset of $\overline{P}_n(\bC)$ where the points correspond to irreducible scaled curves. As in \S3, this locus is identified with $\bA^n/\bG_a$ (see \Cref{C:equivalentfunctors}). 

The underlying curve of any $C\in \overline{P}_n(\bC)$ is genus $0$ and so its dual graph, $\Gamma(C)$, is actually a tree. There is a canonical choice of root for the tree corresponding to the component on which $p_\infty$ lies. There is also a notion of ``distance'' of a component from the root, given by how many nodes separate them. Thus, it makes sense to call an irreducible component of $C$ \emph{terminal} if it is not connected to an irreducible component further from the root (see \S2). We call a component \emph{intermediate} if it is neither the root component nor terminal. Note, however, that in general $\Gamma(C)$ does not possess a level structure $\preceq$. We regard the root component as the ``bottom'' of the tree and the terminal components as the ``top.''

Recall that for any smooth subcurve $C'\subset C$ containing nodes $n_1,\ldots, n_k$, $\omega_C|_{C'}\cong \omega_{C'}(n_1+\cdots + n_k)$ (see \cite{MorrisonHarris}*{p.82}).\endnote{We can apply \cite[\href{https://stacks.math.columbia.edu/tag/0E34}{Lemma 0E34}]{Stacksproject} with $Y = C'$ and $X = C$. We have a short exact sequence of sheaves 
\[
    0\to \omega_{C'}\to i^*\omega_{C} \to \cO_{C'\cap Z} \to 0
\]
where $Z$ denotes the scheme theoretic closure of $C\setminus C'$. This is in particular the node points $n_{1},\ldots, n_{k}$ with their reduced subscheme structure. This implies $i^*\omega_C = \omega_{C'}(n_{1}+\cdots + n_{k})$ as was to be shown.} We next describe the points of $\overline{P}_n(\bC)$ (see also \cite{Zahariucmarked}).

\begin{lem}
\label{L:scaledline}
Consider $(C,p_\infty,s,p_1,\ldots, p_n) \in \overline{P}_n(\bC)$, where $s \in \Gamma(C,\omega_C^\vee(-2p_\infty))$ is the logarithmic vector field corresponding to $\phi:\omega_C\to \cO_C$.
\begin{enumerate} 
\item $s$ restricts to the zero section on every non-terminal component of $C$; and
\item on every terminal component $s$ vanishes only at the node connecting that component to the rest of the tree (or at $p_\infty$ if $C$ is irreducible).
\end{enumerate}
\end{lem}

\begin{proof}
If $C \cong \bb{P}^1$, then $\omega_C^\vee(-2p_\infty) \cong \cO_C$. Since $s(p_i)\ne 0$ for all $i$, it cannot vanish anywhere besides $p_\infty$. So, (1) and (2) hold. Suppose the length of the dual tree is $\ge 1$. There is a root component $C_0$ containing $p_\infty$ with $k\ge 1$ ascending nodes $n_1,\ldots, n_k$. $\omega_C^\vee(-2p_\infty)|_{C_0} \cong \cO_{C_0}(-k)$ and as $k\ge 1$, $\cO_{C_0}(-k)$ has no nonzero global sections, so $s|_{C_0}\equiv 0$. 

Next, suppose $C_\mu$ is intermediate. $C_\mu$ contains a node $n^+$ connecting it to a higher component and a node $n^-$ connecting it to a lower component. Since $C_\mu$ has at least 2 special points, $\omega_C^\vee(-2p_\infty)|_{C_\mu} \cong \cO_{C_\mu}(2-k)$ for $k\ge 2$ and there are only constant global sections. By inducting on the distance from the root component, we have $s(n)=0$ and $s|_{C_\mu}\equiv 0$. Finally, suppose $C_\tau$ is terminal. $C_\tau$ contains only one node, $n$, connecting it to a lower component and so $s(n)=0$. Since $C_\tau$ contains only one node, $\omega_C^\vee(-2p_\infty)|_{C_\tau} \cong \cO_{C_\tau}(1)$. $s|_{C_\tau}$ cannot vanish at any other point, or else it would be identically zero, contradicting $s(p_i)\ne 0$, where there exists a $p_i \in C_{\tau}$ by stability. 
\end{proof}

\begin{cor}
Let $(C,p_\infty,s,p_1,\ldots,p_n)$ be a stable scaled line. Every intermediate component $C_\mu$ is connected to at least two higher components. 
\end{cor}

\begin{proof}
$s|_{C_\mu}\equiv 0$ by \Cref{L:scaledline} and it follows from the stability assumption that $C_\mu$ must contain at least $3$ special points. However, as $s(p_i)\ne 0$ on the marked points $p_i$, and $p_\infty$ lies on the root component, there must be at least $3$ node points on $C_\mu$. One of these nodes must connect to an irreducible component closer to the root than $C_\mu$. If two nodes have this property, then we the dual graph contains a loop, contradicting $g(C) = 0$.
\end{proof}

$\overline{P}_n(\bC)$ is smooth for $n\le 3$, however for $n\ge 4$ it is mildly singular \cite{Zahariucresolution}. One of the main results of \cite{Zahariucresolution} is to construct a resolution of singularities of $\overline{P}_n(\bC)$. In the following theorem, $W_n$ is the augmented wonderful variety introduced in \S4.

\begin{thm}
\label{T:Zahariucresolution}
\emph{(\cite{Zahariucresolution}, Thm. 1.2)} There exists a small resolution of singularities $\gamma: W_n\to \overline{P}_n(\bC)$ if $n\ge 4$.
\end{thm}

We now relate $\overline{P}_n(\bC)$ to $\Mmscbar_n$ by reinterpreting the moduli functor in terms of meromorphic differentials, which we now define. Let $X$ denote a scheme and $\cL$ an invertible sheaf on $X$. Consider the rank $1$ projective bundle $\pi:\bP(\cO_X\oplus \cL)\to X$. If $f:Y\to X$ is given, a morphism $g:Y\to \bP(\cO_X\oplus \cL)$ over $X$ is equivalent to an invertible sheaf $\cF$ on $Y$ and a surjective morphism of sheaves $f^*(\cal{O}_X\oplus \cal{L})\twoheadrightarrow \cF$ \cite{HartshorneAG}*{II.7.12}. Thus, sections of $\pi$ correspond to invertible sheaves $\cF$ on $X$ with a surjection $\cO_X\oplus \cL\twoheadrightarrow \cF$. Write $\Hom^\pi_X(X,\bP(\cO_X\oplus \cL))$ for the space of sections of $\pi$. Specifying $s\in \Gamma(X,\cL)$ is equivalent to giving a morphism $s:\cO_X\to \cL$. We obtain an injection $\Gamma(X,\cL)\to \Hom_X^\pi(X,\bP(\cO_X\oplus \cL))$ by sending $s\mapsto (s,\id)\in \Hom(\cal{O}_X\oplus \cal{L},\cal{L})$. 

\begin{defn}
Given a scheme $X$ and $\cL\in \Pic(X)$, we call $s \in \Hom_X^\pi(X,\bb{P}(\cL\oplus \cO_X))$ a \emph{meromorphic section} of $\cL$. The meromorphic section $s_0^\cal{L}$ of $\cL$ corresponding to $(\id,0) \in \Hom(\cO_X\oplus \cL,\cO_X)$ is called the \emph{zero section}, while $s_\infty^{\cal{L}} = (0,\id) \in \Hom(\cO_X\oplus \cL,\cL)$ is called the \emph{section at infinity}. For a related definition in a similar context, see \cite{GSW}*{\S2}.
\end{defn}

A line bundle quotient $\cL\oplus \cO_X\twoheadrightarrow\cF$ corresponds to a line bundle quotient of $\cO_X\oplus \cL^\vee$ by twisting by $\cL^\vee$. This also gives a canonical isomorphism $\bP(\cL\oplus \cO_X) \cong \bP(\cO_X\oplus \cL^\vee)$.

\begin{lem}
\label{L:merocorrespondence}
The bijection between line quotients of $\cL\oplus \cO_X$ and those of $\cO_X\oplus \cL^\vee$ gives a bijection between meromorphic sections of $\cL$ and meromorphic sections of $\cL^\vee$. Under this bijection, $s_0^\cal{L}$ corresponds to $s_\infty^{\cal{L}^\vee}$ and $s_\infty^\cal{L}$ to $s_0^{\cal{L}^\vee}$. 
\end{lem}

\begin{proof}
    The first claim is true by the description of meromorphic sections in terms of line bundle quotients. For the second claim, we have $\Hom(\cO_X\oplus \cL,\cO_X) \cong \Hom(\cL^\vee \oplus \cO_X,\cL^\vee) \cong \Hom(\cO_X\oplus \cL^\vee,\cL^\vee)$ and under these maps $s_0^{\cL} = (\id,0)\mapsto(\id,0)\mapsto (0,\id) = s_\infty^{\cL^\vee}$.
\end{proof}

\begin{defn}
\label{D:meromorphicdifferential}
    Given a $C\to S$ flat family of prestable nodal curves, 
    \begin{enumerate}  
        \item a meromorphic section $\omega$ of $\omega_{C/S}$ is called a \emph{meromorphic differential} on $C\to S$. 
        \item For a smooth section $x:S\to C$ of $\pi$, we say $\omega$ has a \emph{pole} of order $n\ge 1$ along $x$ if the corresponding meromorphic section of $\omega_{C/S}^\vee$ vanishes to order $n$ along $x$. I.e., if the induced map $\omega^\vee:\omega_{C/S}\to \cO_C$ factors through $\cO_C(-nx(S))\to \cO_C$.
    \end{enumerate}
\end{defn}

\begin{prop}
\label{P:dualization}
The following two set-valued functors on $\Sch$ are isomorphic: $S\in \Sch \mapsto$ flat families of arithmetic genus $0$ curves $\pi:C\to S$ equipped with
\begin{enumerate} 
    \item a meromorphic section of $\omega_{C/S}^\vee$; 
    \item a meromorphic section of $\omega_{C/S}$.
\end{enumerate}
Suppose given a smooth section $x:S\to C$ of $\pi.$ Under this correspondence, a meromorphic section of $\omega_{C/S}$ having a pole of order $n\ge 1$ along $x$ corresponds to a meromorphic section of $\omega_{C/S}^\vee$ vanishing to order $n$ along $x$.
\end{prop}

\begin{proof}
    Denote the functor of (1) by $\Omega^\vee$ and that of (2) by $\Omega$. For $S\in \Sch$, define $\Omega^\vee(S)\to \Omega(S)$ by sending a line bundle quotient $\omega_{C/S}^\vee \oplus \cO_C\twoheadrightarrow \cF$ to $\cO_C\oplus \omega_{C/S}\twoheadrightarrow\cF\otimes \omega_{C/S}$. Functoriality of formation of the relative dualizing sheaf $\omega_{C/S}$ implies that this is an isom\-orphism of functors.
\end{proof}

\begin{cor}
\label{C:equivalentfunctors}
The functor $F$ of \Cref{T:zahariucmarkedmain} is equivalent to the functor $F^\vee$ associating to each scheme $S$ the set of isomorphism classes of the following data:
\begin{itemize}
    \item a genus $0$ prestable curve $\pi:C\to S$;
    \item smooth sections $x_1,\ldots, x_n :S\to C$ of $\pi$ (not assumed disjoint); and
    \item a nonvanishing meromorphic section $\omega$ of $\omega_{C/S}$,
\end{itemize}
such that 
\begin{enumerate}
    \item $\omega$ has an order $2$ pole along $p_\infty(S)$;
    \item $\omega$ does not have a pole along any of the sections $x_i$ for $i=1,\ldots,n$; and 
    \item the following stability condition holds: for any geometric point $\overline{s}\to S$, 
        \begin{enumerate}
            \item with the possible exception of the component containing $p_{\infty,\overline{s}}$, no irreducible component of $C_{\overline{s}}$ intersects exactly two other components;
            \item any irreducible component of $C_{\overline{s}}$ which intersects exactly one other irreducible component contains at least one of the points $x_{1,\overline{s}},\ldots, x_{n,\overline{s}}$ but not the point $p_{\infty,\overline{s}}$.
        \end{enumerate}
\end{enumerate}
\end{cor}

\begin{proof}
    This is a simple dualization of $F$, using the correspondence of \Cref{P:dualization}.    
\end{proof}

By \Cref{T:zahariucmarkedmain} and \Cref{C:equivalentfunctors}, $F^\vee$ is represented by $\overline{P}_n$. So, an element of $\overline{P}_n(\bC)$ corresponds to an $(n+1)$-marked nodal genus $0$ curve $C$ equipped with a meromorphic section $\omega$ of $\omega_C$, the dualizing sheaf, satisfying the stability conditions. Henceforth, we consider scaled curves $(C,p_\infty,\omega,p_1,\ldots,p_n) \in \overline{P}_n(\bb{C})$ to be equipped with meromorphic differentials.

By \Cref{L:scaledline}, $\omega$ restricted to non-terminal components is identically equal to $\infty$. Conseq\-uently, we may regard a closed point in $\overline{P}_n(\bC)$ as $(C,p_\infty, \omega_{\rm{term}},p_1,\ldots, p_n)$, where $\omega_{\rm{term}}$ is the data of a nonzero meromorphic differential with a unique pole of order $2$ at the descending node on each terminal component of $C$. We define a set map $\xi:\Mmscbar_n \to \overline{P}_n(\bC)$ by 
\[
    \xi(\Sigma,\preceq, p_\infty,\omega_\bullet, p_1,\ldots, p_n) = (\Sigma,p_\infty, \omega_{\rm{term}}, p_1,\ldots, p_n).
\]
$\xi$ forgets the level structure $\preceq$ on the dual tree to $\Sigma$ and forgets the meromorphic differentials on all but the terminal components of $\Sigma$.

\begin{prop}
\label{P:factsaboutxi}
    $\xi:\Mmscbar_n \to \overline{P}_n(\bC)$ is $G$-equivariant and restricts to the identity between $\bA^n/\bG_a\subset \Mmscbar_n$ and $\bA^n/\bG_a\subset \overline{P}_n(\bC)$. 
\end{prop}

\begin{proof}
    $\bA^n/\bG_a$ includes into $\Mmscbar_n$ by identifying $[z_1,\ldots, z_n] \in \bA^n/\bG_a$ with $(\bP^1,\infty,dz,z_1,\ldots, z_n)$ (see \S 3). The inclusion $\bA^n/\bG_a\to \overline{P}_n$ is specified in exactly the same manner. It follows that $\xi$ is the identity when restricted to $\bA^n/\bG_a$. $\Mmscbar_n$ is a $G$-equivariant compactification of $\bA^n/\bG_a$ by \Cref{P:GactiononMn}, and the same is true of $\overline{P}_n$ by \cite{Zahariucmarked}*{pp. 4-5}. Comparing the definitions of the actions immediately gives that $\xi$ is $G$-equivariant. 
\end{proof}

We can now state the main theorem for this part of the paper. Here, $\gamma:W_n\to \overline{P}_n$ denotes the resolution of \cite{Zahariucresolution}. 

\begin{thm}
\label{T:diagramcommutes}
    Under the isomorphism $f:\cA_n\to W_n$ of \Cref{T:isothm}, $\xi$ correponds to $\gamma$; i.e. $\gamma \circ f = \xi$. In particular, $\xi:\Mmscbar_n \to \overline{P}_n(\bC)$ is a $G$-equivariant resolution of singularities.
\end{thm}

\subsection{Proof of \Cref{T:diagramcommutes}}

Write $\sigma = \gamma \circ f$. $\sigma$ is a surjective birational morphism of algebraic varieties. Given a level tree $(\Gamma,\preceq)$, we write $\lvert \Gamma\rvert$ to emphasize the underlying tree without level structure. Given $\Sigma \in \Mmscbar_n$, by definition of $\xi$ the dual tree of $\xi(\Sigma)$ is $\lvert \Gamma(\Sigma)\rvert$.

\begin{lem}
\label{L:dualtreepreserved}
    Given $\Sigma \in \Mmscbar_n$ with dual level tree $\Gamma(\Sigma)$, the dual tree of $\sigma(\Sigma)$ is $\lvert \Gamma(\Sigma)\rvert$.
\end{lem}

\begin{proof}
    $\Sigma\in \Mmscbar_n$ has dual level tree $\Gamma(\rho_\bullet)$ if and only if $\Sigma \in S_{\rho_\bullet}^\circ$ for $\rho_\bullet \in \mathfrak{Ch}(L_n^-)$. However, $f(S_{\rho_\bullet}^\circ) = D_{\rho_\bullet}^\circ \subset W_n$ by \Cref{T:isothm}. On the other hand, by \cite{Zahariucresolution}*{Remark 5.8} for $x\in D_{\rho_\bullet}^\circ$, the dual tree of $\gamma(x)$ is $\lvert \Gamma(\rho_\bullet)\rvert$.
\end{proof}

For $1\le i<j\le n$, let $U_{ij}$ denote the set of $\Sigma \in \Mmscbar_n$ such that $p_i$ and $p_j$ lie on the same terminal component of $\Sigma$ and likewise for $V_{ij}(\bC)\subset \overline{P}_n(\bC)$. We will often drop the $\bC$ from the notation, writing $V_{ij}$ in lieu of $V_{ij}(\bC)$.

\begin{cor}
\label{C:ijsets}
    For any $1\le i <j \le n$, $U_{ij}$ is open, $U_{ij} = \sigma^{-1}(V_{ij})$, and $V_{ij}$ is open.
\end{cor}

\begin{proof}
    $U_{ij}^c$ consists of those $\Sigma$ for which $p_i$ and $p_j$ lie on different components. One can verify that $U_{ij}^c = \bigcup_{\rho:i\not\sim^\rho j}S_\rho$ and thus $U_{ij}^c$ is a closed subspace of $\Mmscbar_n$. $C\in V_{ij}$ if and only if the half edges corresponding to $p_i$ and $p_j$ are attached to the same vertex of $\Gamma(C)$. Therefore, by \Cref{L:dualtreepreserved} we have $\sigma^{-1}(V_{ij}) = U_{ij}$.\endnote{Let $\Sigma \in \Mmscbar_n$ be given. We know $|\Gamma(\Sigma)| = |\Gamma(\sigma(\Sigma))|$ and since $p_i$ and $p_j$ lying on the same component is visible at the level of the dual tree, we see $\Sigma \in U_{ij}$ if and only if $\sigma(\Sigma)\in V_{ij}.$} This also implies that $\sigma^{-1}(V_{ij}^c) = U_{ij}^c$ and by surjectivity of $\sigma$ we have $\sigma(U_{ij}^c) = V_{ij}^c$. However, $U_{ij}^c$ is proper, being a closed subvariety of a proper variety and thus $V_{ij}^c$ is closed. 
\end{proof}

For $C\in V_{ij}(\bC)$, define $I_{ij}(C) = \int_{p_i}^{p_j} \omega_v$, where $\omega_v$ is the differential on the terminal component containing both $p_i$ and $p_j$ and the integral is over any curve in the smooth locus of $C_v$. As in the case of $\Mmscbar_n$, this is well-defined.\endnote{Note, however, that here we can only define integrals between marked points lying on the same terminal component.} We now consider $\overline{P}_n(\bC)$ as a complex analytic space, and regard $I_{ij}$ as a function $V_{ij}(\bC)\to \bC$.

\begin{prop}
\label{L:continuousintegration}
    $I_{ij}:V_{ij}\to \bb{C}$ is continuous for the analytic topology. 
\end{prop}

\begin{proof}
    The open immersion $V_{ij}\to \overline{P}_n$ classifies a family $\pi:\mathcal{C}_{ij}\to V_{ij}$ of scaled curves. Consider $\cC^{\rm{sm}}_{ij}:= \cC_{ij}\setminus \cC^{\rm{sing}}_{ij}$. The fibres of $\pi:\cC^{\rm{sm}}_{ij}\to V_{ij}$ are smooth and $\pi$ is flat, being a composition of flat morphisms. So, by \cite[\href{https://stacks.math.columbia.edu/tag/01V8}{Tag 01V8}]{Stacksproject} $\pi:\cC^{\rm{sm}}_{ij}\to V_{ij}$ is smooth. 

    Henceforth, we consider the induced morphism of analytic spaces (omitting $(-)^{\rm{an}}$ from the notation). The marked points $p_i$ and $p_j$ defined fibrewise correspond to sections $s_i,s_j:V_{ij}\to \cC^{\rm{sm}}$ of $\pi$. Now, fix a point $0\in V_{ij}$ and denote the corresponding fibre by $\cC_0$. $s_i(0)$ and $s_j(0)$ can be connected by a $C^\infty$ curve in the irreducible component of $\cC_{0}^{\rm{sm}}$ on which they lie.\endnote{Indeed, the irreducible of $\cC_0$ in which they lie is isomorphic to $\bP^1$ and meets the singular locus only at its single node point connecting it to the rest of the tree. So, the corresponding irreducible component of $\cC^{\rm{sm}}_0$ is isomorphic to $\bA^1$.} Parametrize this curve as $\eta:[0,1]\to \cC_0^{\rm{sm}}$ such that $\eta(0) = s_i(0)$ and $\eta(1) = s_j(0)$. At any $p\in \eta(I)$, $\cC_0$ is smooth and by \cite{Fischer}*{Thm. p.159} $\pi$ is an analytic submersion at $p$. Consequently, by \cite{Fischer}*{pp. 99-100} there exists an open neighborhood $U_p$ around $p$ such that $\pi$ factors as 
    \begin{equation}
    \label{E:trivialized}
        \begin{tikzcd}
            U_p \arrow[r,"\sim"]\arrow[d,"\pi",swap]& \Delta_p \times D\arrow[dl,"\rm{pr}_1"]\\
            \Delta_p
        \end{tikzcd}
    \end{equation}
    where $\Delta_p \subset V_{ij}$ is an open neighborhood of $0$ and $D\subset \bC$ the unit disk. By a compactness argument, there exists an open cover $\{U_i\}_{i=1}^n$ of $\eta(I)$ and $0 = x_0<x_1<\cdots<x_n = 1$ in $I$ such that for each $1\le i \le n$, $\eta(x_{i-1})$ and $\eta(x_i)$ lie in $U_i$ and $\pi$ restricted to $U_i$ is trivialized as in \eqref{E:trivialized}. Put $\Delta = \bigcap_{i=1}^n \Delta_{p_i}$ and for each $1\le i \le n$, construct local sections $\sigma_i:\Delta \to U_i \cong \Delta \times D$ by $t\mapsto (t,0)$. Over $\Delta$, $I_{ij}(t) = \sum_{k=1}^n \int_{\sigma_{k-1}(t)}^{\sigma_k(t)} \omega_t$, so we prove that $\int_{\sigma_{k-1}(t)}^{\sigma_k(t)}\omega_t$ is continuous in $t$ for each $k$. Using the trivialization, this corresponds to an integral $\int_0^{a(t)}f(z,t)dz$ for $t\in \Delta$ such that $\lim_{t\to 0} a(t) = a$ and $f\in C^0(\Delta \times D,\bC)$. However, it follows from the dominated convergence theorem that for a sequence $(t_n)\to 0$ in $\Delta$ one has $\lim_{n\to\infty} \int_0^{a(t_n)} f(z,t_n)dz = \int_0^a f(z,t)$.\endnote{Up to shrinking $\Delta\times D$, we may assume that $\Delta \times D$ is precompact and that $f$ extends continuously to the boundary. Therefore, $f:\Delta \times D\to \bC$ is bounded. Consequently, choose $C$ so that $\lvert f\rvert \le C$ on $\Delta \times D$. Now, 
    \[
    \int_{a}^{a(t)}f(z,t)dz = \int_0^a f(z,t)dz + \int_a^{a+\epsilon(t)} f(z,t)dz
    \]
    where $\lim_{t\to0} \epsilon(t) = 0$. $\lvert \int_a^{a+\epsilon(t)} f(z,t) dz \rvert \le \epsilon(t)\cdot C$ and so 
    \[
    \lim_{n\to\infty} \int_a^{a(t_n)}f(z,t_n) dz = \lim_{n\to\infty} \int_0^af(z,t_n) dz
    \]
    and the dominated convergence theorem implies the claim.}
\end{proof}

\begin{cor}
\label{C:scalingsagree}
    $\sigma^*(I_{ij}) = \Pi_{ij}:U_{ij}\to \bb{C}$ for all $1\le i < j \le n$.    
\end{cor}

\begin{proof}
    $\sigma^*(I_{ij})$ is a continuous function that agrees with $\Pi_{ij}$ on $\bA^n/\bG_a \cap U_{ij}$. $\bA^n/\bG_a \cap U_{ij}$ is dense in $U_{ij}$ and consequently $\sigma^*(I_{ij}) = \Pi_{ij}$ on $U_{ij}$. 
\end{proof}

\begin{defn}
    Given $\rho \in L_n$, we define subfunctors of $F$ on $T$-points by
    \begin{enumerate}   
        \item $F_\rho(T) \subset F(T)$ parametrizes families such that $i\sim^\rho j \Rightarrow p_i = p_j$;
        \item $G_\rho(T)\subset F(T)$ parametrizes families such that $p_i = p_j\Rightarrow i\sim^\rho j$; and
        \item $F_\rho^\circ(T) = F_\rho(T)\times_{F(T)} G_\rho(T)$, i.e. $p_i = p_j \iff i\sim^\rho j$. 
    \end{enumerate}
\end{defn}

\begin{lem}
\label{L:movingstrata}
For any $\rho \in L_n$, $F_\rho$ is a closed subfunctor of $F$, $G_\rho$ is an open subfunctor of $F$, and $F_\rho^\circ$ is a locally closed subfunctor of $F$. 
\end{lem}

\begin{proof}
    The third claim follows from the first two. Given a $T$-point of $F$, $(\pi:C\to T,\omega,s_\infty,s_1,\ldots, s_n)$, where the $s_i$ are sections of $\pi$, the following square is Cartesian:
    \[
        \begin{tikzcd}
            W\arrow[d]\arrow[r]&T\arrow[d]\\
            G_\rho\arrow[r]&F,
        \end{tikzcd}
    \]
    where, $W = \bigcap_{i\not\sim^\rho j} E_{ij}^c$, for $E_{ij}\subset T$ the equalizer subscheme of $s_i$ and $s_j$. In particular, by  \cite[\href{https://stacks.math.columbia.edu/tag/01KM}{Tag 01KM}]{Stacksproject} each $E_{ij}$ is closed.\endnote{Note that we can apply the stronger version of the cited lemma because $\pi:C\to T$ is separated by definition.} It follows that $G_\rho$ is an open subfunctor of $F$. A similar argument implies that $F_\rho$ is a closed subfunctor of $F$.\endnote{In the case of $F_\rho \subset F$, the scheme in the fibre product diagram is the (scheme theoretic) intersection of the elements of $\{E_{ij}\}_{i\sim^\rho j}$ where $E_{ij}\subset T$ is the equalizer subscheme of $T$ as in the body of the proof.}
\end{proof}

\Cref{L:movingstrata} allows us to make the following definitions.

\begin{defn}
    For $\rho \in L_n$, let $Z_\rho$ denote the closed subscheme of $\overline{P}_n$ representing $F_\rho$, $Y_\rho$ the open subscheme of $\overline{P}_n$ representing $G_\rho$, and $Z_\rho^\circ = Y_\rho \cap Z_\rho$ the locally closed subscheme representing $F_\rho^\circ$. We have $Z_\top = \overline{P}_n$ and $Z_\top^\circ$ is the open subscheme of $\overline{P}_n$ parametrizing curves on which no marked points collide.
\end{defn}

In what follows, $\overline{M}_{0,n+1}$ denotes the Grothendieck-Knudsen moduli space of stable $(n+1)$-marked genus $0$ curves \cite{Knudsencurves}. It is smooth and projective over $\Spec \bZ$ of dimension $n-2$. Given a finite set $A$, we also consider $\overline{M}_{0,A}$ which is the moduli space of stable $A$-marked genus $0$ curves. There are noncanonical isomorphisms $\overline{M}_{0,A}\to \overline{M}_{0,|A|}$ coming from choosing a bijection $\{1,\ldots, |A|\}\to A$. $\mathfrak{M}_{0,n+1}$ denotes the stack of prestable $(n+1)$-marked genus $0$ curves (cf. \cite{BehrendGWII}) and $\mathfrak{M}_{0,A}$ is defined analogously. For any $\rho \in L_n$ it will be convenient to put $B^+(\pi) := B(\pi) \cup \{\infty\}$. 

Given any $\rho \in L_n$, there is a forgetful map $Y_\rho\to \mathfrak{M}_{0,B^+(\rho)}$, which on $T$-points is given by $(C\to T,\omega,p_\infty,p_1,\ldots, p_n)\mapsto (C\to T,p_\infty,\{p_{\mu(b)}:b\in B(\rho)\})$. \cite{Knudsencurves} constructs a contraction morphism $c:\mathfrak{M}_{0,n+1} \to \overline{M}_{0,n+1}$ which on closed points contracts unstable irreducible components to points. In particular, when $C\in \mathfrak{M}_{0,n+1}(\bC)$ has a component with one node point and one marked point, this component is deleted and a new marked point is placed in the position of the node point. 

\begin{defn} 
    For any $\rho \in L_n$, define $\psi_\rho:Y_\rho \to \overline{M}_{0,B^+(\rho)}$ by the composition $Y_\rho \to \mathfrak{M}_{0,B^+(\rho)} \to \overline{M}_{0,B^+(\rho)}$ of the two maps above.
\end{defn}

On $\bC$-points, $\psi_\rho$ is given by sending $(C,\omega,p_\infty,p_1,\ldots, p_n)\mapsto (C,p_\infty,\{p_{\mu(b)}:b\in B(\rho)\}) \mapsto c(C,p_\infty,\{p_{\mu(b)}:b\in B(\rho)\})$.

\begin{lem}
\label{L:preimage}
    For any $\rho$, we have $\sigma^{-1}(Z_\rho(\bC)) = T_\rho$, $\sigma^{-1}(Y_\rho(\bC))$ and $\sigma^{-1}(Z_\rho^\circ(\bC)) = T_\rho^\circ$.
\end{lem}

\begin{proof}
    For example $C\in Z_\rho(\bC)$ if $I_{ij}(C) = 0$ for all $i\sim^\rho j$. However, by \Cref{C:scalingsagree}, the same is true of any $\Sigma \in \sigma^{-1}(C)$ and conversely. The same argument applies for $Z_\rho^\circ$. 
\end{proof}

Recall that $T_\top^\circ$ denotes the open dense subset of $\Mmscbar_n$ in which no marked points collide (see \Cref{P:collisionstratification}). By \Cref{L:preimage}, $\sigma$ maps $T_\top^\circ \to Z_\top^\circ = Y_\top^\circ$.\endnote{$Y_\top^\circ$ parametrizes curves for which $p_i = p_j$ implies $i\sim^\top j$, i.e. curves where no collision occurs. $Z_\top$ parametrizes curves for which $i\sim^\top j$ implies $p_i = p_j$, i.e. with no condition imposed. So, $Z_\top = \overline{P}_n$. Thus, $Z_\top^\circ = Y_\top^\circ$.}   

\begin{prop}
\label{P:knudsenmapcomp}
    Viewed as maps $T_\top^\circ\to \overline{M}_{0,n+1}(\bC)$ on closed points, $\psi_\top \circ \sigma = \psi_\top \circ \xi$. 
\end{prop}

\begin{proof}
    We induct on the length of the dual tree. In the case of length $0$, the maps agree because $\sigma$ identifies $T_\top^\circ \cap (\bA^n/\bG_a)$ with $Y_\top^\circ \cap (\bA^n/\bG_a)$. Suppose the result is known for all points of $\Mmscbar_n$ with dual tree of length $\le k$. Consider $\Sigma \in \Mmscbar_n$ with dual tree of length $k+1$, corresponding to $\rho_\bullet = \{\rho_1<\cdots<\rho_{k+1}\} \in \mathfrak{Ch}(L_{n}^-)$. Write $\Gamma = \Gamma(\Sigma)$ and consider $U_\Gamma$ as in \S3. Define $\beta:\bA^1\to U_\Gamma$ by $z_{ij}(t) = z_{ij}(\Sigma)$ for all $t$, $t_i(t) = 0$ for all $1\le i \le \ell-1$, and $t_\ell = t$. $\beta$ restricts to a morphism $\bA^1\setminus \{0\}\to S_{\rho_1,\ldots,\rho_k}^\circ$ and $\beta(0) = \Sigma$. By induction, $\psi_\top\circ \sigma \circ \beta = \psi_\top\circ \xi \circ \beta$ on $\bA^1\setminus \{0\}$. As $\psi_\top\circ \sigma \circ \beta$ is continuous in the analytic topology, $(\psi_\top\circ \sigma)(\Sigma) = \lim_{t\to0}(\psi_\top\circ\sigma\circ\beta)(t)$. 
    
    $(\psi_\top\circ\sigma\circ \beta): \bA^1\setminus \{0\} \to \overline{M}_{0,n+1}$ classifies a family $C^*\to \bA^1\setminus \{0\}$ whose fibre over $t$ is the contraction of the underlying curve of $\beta(t)$; this amounts to contracting the terminal components containing a single marked point. Forgetting the marked points, $C^*\to \bA^1\setminus \{0\}$ is a trivial family and admits a trivial flat extension $C\to \bA^1$.

    Consider a terminal component of $\beta(t)$ corresponding to $B\in B(\rho_k)$, which we may assume is not a singleton since such components have been contracted. There is an associated family $\bA^1\times (\bA^1\setminus \{0\}) \to \bA^1\setminus \{0\}$ with sections corresponding to the marked points specified by regular functions on the base: $\{z_b/t + z_i: i\in b\in B(\rho_{k+1})\text{ and } b\subset B\}$, with $z_b \ne z_{b'}$ for $b\ne b'$. We obtain an equivalent family by rescaling by $t$ with sections $\{z_b + z_i t: i\in b\in B(\rho_{k+1})\text{ and } b\subset B\}.$\endnote{Each terminal component with the ascending node removed corresponds to a trivial flat family $\bA^1\times (\bA^1\setminus \{0\})\to \bA^1\setminus \{0\}$ and we can obtain an equivalent family by $\bG_m$ fibrewise by $(t,x,t)\mapsto (tx,t).$}

    As $t\to 0$, a pair of marked points on this component collides if and only if their constant terms are equal; in this case, they limit to the common value $z_b$. Blow up $\{z_b\}_{b\subset B}$ on this terminal component. The strict transform of $z_b+tz_i$ intersects the exceptional divisor over $z_b$ at the point $[z_b:1]$ in projective coordinates.\endnote{The subscheme corresponding to the family of terminal components corresponding to a given $B\in B(\rho_k)$ is isomorphic to $\bA^2\to \bA^1 = \Spec \bC[t]$ with sections determined by $z_b+tz_i\in\bC[t]$. Consider $\{z_b+tz_i:i\in b\}$ for some $b\in B$. As $t\to0$ these points intersect $z_b$ in the fibre over $t=0$, which we blowup. 
    
    Slightly more precisely, the section $\bA^1\to \bA^2 = \Spec \bC[s,t]$ has image given by the subvariety $f_i(s,t) = 0$ for $f_i(s,t) = z_b+z_it-s$. $T_{(0,z_b)} \bA^2$ has basis given by $\{\partial_t,\partial_s\}$. Now, $df_i = z_idt - ds$ and the strict transform of this section intersects $\bP(T_{(0,z_b)}\bA^2)$ in the zero locus $df_i = 0$. In homogeneous coordinates, this is $[1:z_b]$ as was to be shown.} Denote the resulting $(n+1)$-marked nodal curve in the central fibre by $\widetilde{C}_0$. We stabilize $\widetilde{C}_0$ by contracting any exceptional divisor that contains a unique marked point, i.e. those corresponding to $b\subset B$ which is a singleton. Comparing this to the stabilization of $\Sigma$, we have $(\psi_\top\circ \sigma)(\Sigma) = (\psi_\top\circ \xi)(\Sigma)$.
\end{proof}


\begin{lem}
\label{L:recoveringcurve}
    $(C,\omega,p_\infty,p_1,\ldots,p_n)\in Z_\rho^\circ(\bC)$ is determined up to isomorphism by the data of $\Gamma(C),\psi_\rho(C)$, and $\{I_{ij}(C)\}$, where $i,j$ ranges over all pairs for which the half edges of $p_i$ and $p_j$ lie on the same vertex of $\Gamma(C)$.
\end{lem}

\begin{proof}
    Given $(\Gamma(C),\psi_\rho(C),\{I_{ij}(C)\})$, we reconstruct $C$ as follows: $\psi_\rho(C)$ is a contraction of $(C,p_\infty,\{p_{\mu(b)}:b\in B(\rho)\})$ in the sense of \cite{Knudsencurves}. By the hypothesis that $(C,\omega,p_\infty,p_1,\ldots,p_n) \in Z_\rho^\circ(\bC)$, each terminal component of $(C,p_\infty,\{p_{\mu(b)}:b\in B(\rho)\})$ contains at least one marked point.\endnote{Given a terminal component $C_\tau$ of $C$, there exists $b\in B(\rho)$ such that $i\in B$ implies $p_i\in C_\tau$. Consequently, $p_{\mu(b)}\in C_\tau$.} So, contraction only deletes terminal components $C_\tau$ containing a unique marked point $p_{\mu(b)}$ and places a new marked point in place of the node connecting $C_\tau$ to the rest of the curve. Associated to this is a surjective map $\Gamma(C)\twoheadrightarrow\Gamma(\psi_\rho(C))$ which collapses edges corresponding to deleted nodes. 
    
    Suppose a marked point $p_{\mu(b)}$ on $\psi_\rho(C)$ corresponds to a half edge on $\Gamma(\psi_\rho(C))$ attached to a vertex with more than $1$ preimage under $\Gamma(C) \twoheadrightarrow \Gamma(\psi_\rho(C))$.\endnote{The contraction map in this case only contracts terminal components with a single marked point. Consequently, $v\in V(\Gamma(\psi_\rho(C)))$ has multiple preimage points if and only if a terminal vertex of $\Gamma(C)$ has been contracted to it.} Attach $(\bP^1,\infty,0)$ to $\psi_\rho(C)$ by identifying $\infty$ with $p_{\mu(b)}$. Applying this procedure for all such $p_{\mu(b)}$ recovers $(C,p_\infty,\{p_{\mu(b)}:b\in B(\rho)\})$.

    Now, on each terminal component $C_\tau$ of $C$, $\{I_{ij}:p_i,p_j\in C_\tau\}$ allows us to recover the structure of the irreducible scaled line $(C_\tau,\omega_\tau,p_i\in C_\tau)$. Reattach the relevant marked points and equip $C_\tau$ with differential $\omega_\tau$ for each terminal component $C_\tau$. This procedure recovers $(C,\omega,p_\infty,p_1,\ldots,p_n)$ up to isomorphism of scaled lines.
\end{proof}

\begin{proof}[Proof of \Cref{T:diagramcommutes}]
    Suppose $\Sigma \in T_\top^\circ$ so that by \Cref{P:knudsenmapcomp} one has $\psi_\top(\xi(\Sigma)) = \psi_\top(\sigma(\Sigma))$. By \Cref{L:dualtreepreserved}, the dual tree of $\sigma(\Sigma)$ is $|\Gamma(\Sigma)|$ and by definition this is also the case for $\xi(\Sigma)$. Finally, by \Cref{C:scalingsagree} $I_{ij}(\sigma(\Sigma)) = \Pi_{ij}(\Sigma) = I_{ij}(\xi(\Sigma))$ for all $i,j$ for which $p_i$ and $p_j$ lie on the same component of $\Sigma$. By \Cref{L:preimage}, $\sigma(\Sigma)$ and $\xi(\Sigma)$ are in $Z_\top^\circ$ so \Cref{L:recoveringcurve} implies that $\sigma(\Sigma) = \xi(\Sigma)$.

    Now, we induct on the number of colliding points, having proved the $k=0$ case. Let $T_{\le k}\subset \Mmscbar_n$ be the set of $\Sigma$ for which $\le k$ marked points collide. Suppose $\xi|_{T_{\le k}} = \sigma|_{T_{\le k}}$ for $0\le k \le n-1$. Consider $(\Sigma,p_\infty,\omega_\bullet,\preceq,p_1,\ldots, p_n)\in T_\rho^\circ$ for which exactly $k+1$ points collide\endnote{I.e., $\rho \in L_n$ and $p_i = p_j$ if and only if $i\sim^\rho j$.} and put $\Gamma = \Gamma(\Sigma)$. Up to reindexing and applying an automorphism, say $p_1 = p_2 = 0$. In $U_\Gamma$, we have $z_{12} = \Pi_{12}$. 
    
    We define a family of stable $n$-marked multiscaled lines $(\Sigma_t,p_\infty(t),\omega_\bullet(t),\preceq_t,p_1(t),\ldots,p_n(t))$ depending on $t\in \bA^1$. Let $(\Sigma_t,p_\infty(t),\omega_\bullet(t)\preceq_t) = (\Sigma,p_\infty,\omega_\bullet,\preceq)$ for all $t$. Put $p_1(t) = 0$, $p_2(t) = t$, and let $p_i(t) = p_i$ for all $3\le i \le n$. $t\mapsto \Sigma_t$ defines a map $\varphi:\bA^1 \to U_{\Gamma}$ which in coordinates is given by $t_i(t) = 0$ for all $i$, $z_{12}(t) = t$, and $z_{ij}(t) = z_{ij}(\Sigma)$ for all other $i<j$. 
    
    Write $\pi$ for $\max \mathfrak{c}(\Sigma)$ and note that $\pi \le \rho$. In particular, each terminal component of $\Sigma$ contains at least one $p_{\mu(b)}$ for $b\in B(\rho)$. Write $B\in B(\pi)$ for the block containing $1$ and $2$ and $\Sigma_B$ for the corresponding irreducible component of $\Sigma$. There is a nonempty open $W\subset \bA^1$ such that $\varphi(W)\subset T_{\le k}$ and so $(\sigma \circ \varphi)|_{W} = (\xi \circ \varphi)|_{W}$ by induction.\endnote{We define $W$ by removing the finitely many times at which $p_2$ collides with another $p_i$.} There are now two cases:
    \begin{enumerate}
        \item $B$ contains only one block of $\rho$; or
        \item $B$ contains multiple blocks of $\rho$.
    \end{enumerate}

    In case $(1)$, $(\psi_\rho \circ \sigma)(t)$ is constant for all $t\in W$ and $(\psi_\rho\circ \xi)(t)$ is constant for all $t\in \bA^1$.\endnote{The only dependence of $\varphi(t) = \Sigma_t$ on $t$ is the distance between $p_1$ and $p_2$ on the terminal component corresponding to $B\in B(\pi)$. However, applying $\psi_\rho$ we see that this component is contracted to a point independently of $t$. Since $\xi\circ \varphi = \sigma\circ \varphi$ on $\bA^1\setminus \{0\}$, this implies the claim.} This implies that $\psi_\rho(\xi(\Sigma)) = \lim_{t\to 0} (\psi_\rho\circ \sigma)(t) = \psi_\rho(\sigma(\Sigma))$.\endnote{$\psi_\rho(\sigma(\varphi(t)))$ is constantly equal to $\psi_\rho(\xi(\varphi(0)))$ for all $t\ne 0$ and then as $\psi_\rho \circ \sigma \circ \varphi$ is continuous, we have the result.}

    In case (2), the terminal component of $\xi(t)$ corresponding to $B\in B(\pi)$ is not contracted by $\psi_\rho$ for any $t\in W$. We describe $\psi_\rho(\sigma(t))$ by describing the configuration of marked points on $\Sigma_B$ as a function of $t$. Write this as $(\Sigma_B,p_{\mu(b)}(t):b\subset B)$. By continuity of $\sigma$, the limit is given by $(\Sigma_B,p_{\mu(b)}(0):b\subset B)$. However, this equals the corresponding configuration for $\xi(0) = \xi(\Sigma).$

    We have shown that $\psi_\rho(\sigma(\Sigma)) = \psi_\rho(\xi(\Sigma))$. By \Cref{L:dualtreepreserved} they both have dual tree given by $|\Gamma(\Sigma)|$ and by \Cref{C:scalingsagree} we know that $I_{ij}(\xi(\Sigma)) = I_{ij}(\sigma(\Sigma))$ whenever $i$ and $j$ lie on the same component. As $\sigma(\Sigma)$ and $\xi(\Sigma)$ both lie in $Z_\rho^\circ$, we see that $\sigma(\Sigma) \cong \xi(\Sigma)$ by \Cref{L:recoveringcurve} and hence agree as elements of $\overline{P}_n(\bC)$.
\end{proof}

\appendix

\section{Blowup lemmas}

In this appendix we record lemmas related to the geometry of blowups of varieties and strict transforms. First we recall a version of the universal property of blowing up. For existence of blowups of Noetherian schemes, see \cite{HartshorneAG}.

\begin{lem}
\label{L:univpropblowup}
    Let $W$ be a $k$-scheme. Let $Z\subset W$ denote a closed subspace and $\cC$ the full subcategory of $\mathrm{Sch}_{/W}$ consisting of morphisms $Y\to W$ such that the inverse image of $Z$ is an effective Cartier divisor on $Y$. The blowup $\pi:\Bl_Z W\to  W$ is a terminal object in $\cal{C}$.
\end{lem}

\begin{proof}
    See \cite{Stacksproject}*{Tag 085U}.
\end{proof}

In what follows, we restrict ourselves to the case of a smooth variety $W$ with smooth subvariety $Z$ and $\epsilon:\Bl_Z W\to W$ the associated blowup morphism. Let $V\subset W$ be a smooth subvariety such that $V\not\subset Z$ and $\widetilde{V}$ let denote strict its transform. Let $\cI$ denote the ideal sheaf of $Z$. Since $Z$ is smooth in $W$, $\cI/\cI^2 = \cN_Z^\vee$ is locally free on $Z$ of rank $c = \rm{codim}_Z X$. 

\begin{lem}
\label{L:blowupmapexplicit}
    Suppose $Z = V(g_1,\ldots, g_c)$ for $g_i \in H^0(W,\cO_W)$ for all $1\le i \le c$. Let $f:X\to W$ be given such that $f^{-1}(Z) = D$ is a smooth divisor and put $f_i = f^*(g_i)$ for all $i$.
    \begin{enumerate}  
        \item $\widetilde{f}:X\to \Bl_Z W$ from \Cref{L:univpropblowup} maps $D\to E = \bP(\cN_{Z|W})$, by $\widetilde{f}|_D(x) = (f(x),[df_{1,x}:\cdots:df_{c,x}])$ where $df_i \in H^0(Z,\cN_Z^\vee)$ for all  $1\le i \le c$.\\
        \item Write $f^*(g_c) = t$ and suppose $f_i = h_i\cdot t$ with $h_i \in H^0(X,\cO_X)$ such that $v_D(h_i) = 0$ for all $1\le i \le c-1$.\endnote{This is to say that the discrete valuation $v_D$ associated to $D$ vanishes on $h_i$; i.e., $h_i$ does not vanish along $D$ for each $i$.} Then
        \begin{equation*}
            \widetilde{f}|_D(x) = 
            \begin{cases}
                [h_1(x):\cdots:h_{c-1}(x):1] &\text{ if } dt_x \ne 0\\
                [h_1(x):\cdots:h_{c-1}(x):0] &\text{ if } dt_x = 0.
            \end{cases}
        \end{equation*}
    \end{enumerate}
\end{lem}

\begin{proof}
    (1) Since $Z$ is codimension $c$ and cut out by $g_1,\ldots, g_c$, $\cN_Z^\vee$ is trivialized by the global sections $dg_1,\ldots, dg_c$. By definition, $\Bl_{Z}(X) = \mathbf{Proj}_X(\bigoplus_{n\ge 0}\cI^n)$ and the restriction to $Z$ is $\bP(\cN_Z) = \mathbf{Proj}_Z(\bigoplus_{n\ge 0}\cI^n/\cI^{n+1})$ with homogeneous fibre coordinates over $q\in Z$ given by $dg_{1,q},\ldots, dg_{c,q}$. By \cite[\href{https://stacks.math.columbia.edu/tag/085D}{Tag 085D}]{Stacksproject}, $\widetilde{f}$ is classified by the map of graded sheaves $\bigoplus_{n\ge 0}f^*(\cI^n)\to \bigoplus_{n \ge 0} \cI_D^n$, where $\cI_D$ is the ideal sheaf of $D$. One can verify that $\widetilde{f}|_D:D\to E$ is then classified by the surjection $f^*(\cN_Z^\vee) = f^*(\cI/\cI^2)\twoheadrightarrow\cI_D/\cI_D^2$. By the universal property of $\bP(\cN_Z)$ \cite{HartshorneAG}*{Prop. 7.12}, $\widetilde{f}|_D:D\to \bP(\cN_Z) \cong Z\times \bP^{c-1}$ is $\widetilde{f}(x) = (f(x),[df_{1,x}:\cdots:df_{c,x}]),$ where $df_i$ is regarded as a section of the line bundle $\cN_D^\vee$. 

    (2) Suppose $dt_x \ne 0$. We compute $df_{i,x}/dt_x\in \cN_{D,x}^\vee/\mathfrak{m}_x \cN_{D,x}^\vee \cong \bC$. $f_i = t\cdot h_i$ so the Leibniz rule gives $df_{i} = dt\cdot h_i + t\cdot dh_i$. Evaluating at $x$ gives $df_{i,x} = dt_x\cdot h_i(x)$. In the case where $dt_x = 0$, consider $df_i/df_j$ for $i\ne j$. $df_i = dt\cdot h_i + t\cdot dh_i$, however since $df_i$ is considered as a section along $Z$, $t$ vanishes identically. Consequently, $df_{i,x}/df_{j,x} = (h_i/h_j)(x)$ where we assume without loss of generality that $h_j(x)\ne 0$.
\end{proof}

\begin{lem}
[\cite{Hucompactification} \cite{Ulyanovpoly}]
\label{L:propertransformblowup}
Let $Z$ and $V$ be smooth closed subvarieties of a smooth variety $W$ that intersect cleanly. There is a commutative diagram
\[
\begin{tikzcd}
    \Bl_{Z\cap V} V\arrow[r,hook]\arrow[d]& \Bl_Z W\arrow[d]\\
    V\arrow[r,hook]&W
\end{tikzcd}
\]
where the horizontal arrows are closed immersions and the image of $\Bl_{Z\cap V}V$ is $\widetilde{V}\subset \Bl_Z W$.
\end{lem}

\begin{proof}
    This follows from \cite{HartshorneAG}*{Cor. 7.15}. Indeed, if $\cI$ is the ideal sheaf of $Z$ then $\Bl_Z W$ is the blowup of $\cI$ as in \emph{loc. cit.} Let $i:V\to W$ denote the inclusion and consider the inverse image ideal sheaf $i^{-1}\cI\cdot \cO_V$ on $V$. By the clean intersection hypothesis, $i^{-1}\cI\cdot \cO_V$ is the ideal sheaf of $Z\cap V$ in $V$ and consequently by \emph{loc. cit.} the result follows.
\end{proof}

Write $E$ for the exceptional divisor in $\Bl_Z W$. By the construction of the blowup, it is a projective bundle $\bP(\cN_{Z|W})\to Z$. One can also use \cite{HartshorneAG}*{Cor. 7.15} to prove that $\widetilde{V} \cap E$ is the subbundle $\bP(\cN_{V\cap Z|V})$ of $\bP(\cN_{Z|W})|_{V\cap Z}$.

\begin{lem}
\label{L:Appstricttransform}
    If a proper irreducible variety $S$ is given with a map $g:S\to \Bl_Z W$ such that $(\pi \circ g)(S) \subset V$ and $(\pi \circ g)(S)\not\subset Z$, then $g:S\to \Bl_Z W$ has image contained in $\widetilde{V}$. 
\end{lem}

\begin{proof}
    $(\pi\circ g)^{-1}(V\setminus Z) = U$ is a nonempty open subset of $S$. $g(U)\subset \pi^{-1}(V\setminus Z)$ and thus since $S = \overline{U}$, one has $g(S) \subset \overline{g(U)} \subset \overline{\pi^{-1}(V\setminus Z)} = \widetilde{V}$. 
\end{proof}

\begin{lem}
\label{L:strictrestrict}
    Suppose $V$ and $Z$ meet cleanly. Let $X$ be a smooth variety with a morphism $f:X\to W$, $D$ a divisor in $X$, and $S\subset X$ a smooth and proper subvariety such that 
    \begin{enumerate}
        \item $f^{-1}(Z) = D$; and
        \item $f^{-1}(V) = S$ and $f(S) \not\subset Z$. 
    \end{enumerate}
    By \Cref{L:univpropblowup}, there is an induced $\widetilde{f}: X\to \Bl_Z W$ which restricts to a morphism $\widetilde{f}|_S:S\to \widetilde{V}$ by \Cref{L:Appstricttransform}. By \Cref{L:propertransformblowup}, $\widetilde{V} = \Bl_{Z\cap V} V$, and as $f|_S^{-1}(Z\cap V) = S\cap D$, there is an induced map $\widetilde{g}:S\to \Bl_{Z\cap V} V$ by \Cref{L:univpropblowup}. The following diagram commutes:
    \[
        \begin{tikzcd}
            S \arrow[r,"\widetilde{f}|_S"]\arrow[dr,"\widetilde{g}",swap] & \widetilde{V}\arrow[d,"\psi"]\\
            & \Bl_{Z\cap V} V.
        \end{tikzcd}
    \]
\end{lem}

\begin{proof}
    To verify that $\psi \circ \widetilde{f}|_S$ is induced by the universal property, it suffices to use the diagram 
    \[
        \begin{tikzcd}
            S \arrow[r,"\widetilde{f}|_S"]\arrow[dr,"\widetilde{g}",swap] & \widetilde{V}\arrow[d,"\psi"] \arrow[r]& V\\
            & \Bl_{Z\cap V} V\arrow[ur,"\pi",swap]&
        \end{tikzcd}
    \]
    to observe that the top row's composition is $f|_S$, and thus that $\widetilde{g}$ and $\psi \circ \widetilde{f}|_S$ are both lifts of $f|_S:S\to V$, whence they agree by \Cref{L:univpropblowup}.
\end{proof}

\begin{lem}
\label{L:linearstricttransform}
   Suppose $L = Z(f_1,\ldots, f_c)$ for $f_1,\ldots, f_c\in H^0(\bP^n,\cO(1))$ linearly independent. Let $E$ denote the exceptional divisor of $\Bl_L\bP^n$. Then,
    \begin{enumerate}
        \item $\cN_{L} \cong \cO_L(1)^{\oplus c}$ and consequently $\bP(\cN_{L}) \cong L\times \bP^{c-1}$. 
        \item $df_1,\ldots, df_c$ define global sections of $\bP(\cN_L^\vee)$ which are trivializing. Consequently, they define elements of $\bP H^0(\bP(\cN_L),\cO_{\cN_L}(1))$.
        \item If $M = Z(g_1,\ldots, g_b)$ for $g_i = \sum_{j=1}^c a_{ij}f_j$ for each $1\le i \le b$ and $a_{ij}\in \bC$ then $\widetilde{M} \cap E$ corresponds to $L \times Z(dg_1,\ldots, dg_b) \subset L\times \bP^{c-1}$ where $dg_i = \sum_{j=1}^c a_{ij}df_j$ is regarded as an element of $\bP H^0(\bP(\cN_L),\cO_{\cN_L}(1))$.
    \end{enumerate}
\end{lem}

\begin{proof}
    (1) If $Y\subset \bP^n$ is a degree $d$ hypersurface, then $\cN_Y = \cO_Y(d)$. As $L$ is a complete intersection of $c$ hyperplanes, $\cN_{L} = \cO_L(1)^{\oplus c}$. So, $\cN_L = \cO_L(1)^{\oplus c}$ and $\bP(\cN_L) = \bP(\cO_L(1)^{\oplus c}) \cong \bP(\cO_L^{\oplus c})$, i.e. the trivial bundle. 

    (2) Consider $x\in L$ and an open neighborhood $x\in U \subset \bP^n$. Up to shrinking $U$, consider local nonvanishing sections $s,t\in \cO(1)(U)$. We obtain a system of linear equations $\phi_i = f_i/s$ and $\psi_i = f_i/t$ in $\cO(U)$ for $1\le i \le c$ defining $U\cap L$. Let $i:L\hookrightarrow \bP^n$ denote the inclusion morphism. Regarded as sections of $i^*\Omega_{\bP^n}(U\cap L)$, one has $d\phi_i = (t/s)\cdot d\psi_i$ where $t/s\in \cO_L^*(U\cap L)$. Thus, the local sections $d\phi_i$ satisfy the cocycle condition necessary to globally trivialize $\bP(\cN_{L}^\vee)$. We denote the corresponding global sections of $\bP(\cN_L^\vee)$ by $df_1,\ldots, df_c$.

    (3) For any $p\in L$, $df_{1,p},\ldots, df_{c,p}$ define elements of $\bP H^0(\bP(\cN_{L,p}),\cO(1))$. By the definition of $\Bl_L \bP^n$, $E = \bP(\cN_L)$. Suppose given $p\in L$, so that $E_p = \bP(\cN_{L,p})$. By the properties of the blowup, $\widetilde{M} \cap E_p = \bP(T_pM\cap \cN_{L,p}) \subset \bP(\cN_{L,p})$. However, one verifies in local coordinates that this is exactly $Z(dg_1,\ldots, dg_b)$
\end{proof}

\begin{lem}
\label{L:strictransform}
    Suppose $D_1,\ldots, D_k$ are smooth normal crossings divisors in a smooth variety $X$ with $Z\subset X$ a smooth subvariety and $D_i\not\subset Z$ for all $i$. Consider $\pi:\Bl_Z(X) \to X$ and let $E$ denote the exceptional divisor. One has $\widetilde{D}_1\cap \cdots \cap\widetilde{D}_k\cap E = (D_1\cap \cdots \cap D_k)^{\sim} \cap E$.
\end{lem}

\begin{proof}
    The question is local, so suppose each $D_i$ is cut out by a function $f_i$ for each $i$. Without loss of generality, suppose there exists $p \in D_1\cap \cdots \cap D_k\cap Z$. $\pi^{-1}(p) = E_p =\bP(\cN_{Z|X,p})$ and $\widetilde{D}_i \cap E_p = Z(df_{i,p})$, where $df_{i,p} \in T_p^*X$ defines an element of $\cN_{Z|X,p}^*$ by hypothesis. The strict transform of $D_{\le k} = D_1\cap \cdots \cap D_k$ is given by $\bP(T_{D_{\le k}}\cap \cN_{Z|X}) \subset \bP(\cN_{Z|X})|_{D_{\le k}}$. However, this is just given fibrewise by $Z(df_{1,p},\ldots, df_{k,p})$.
\end{proof}



\printendnotes

\bibliography{refs}{}
\bibliographystyle{plain}

\end{document}